\documentclass[onefignum,onetabnum]{siamart190516}

\usepackage{amsmath,amssymb,amsfonts,latexsym,stmaryrd}
\usepackage{graphicx,float,epsfig} 
\usepackage{mathrsfs} 
\usepackage{color}
\usepackage{setspace} 
\usepackage{lipsum}
\usepackage{graphicx,float}
\usepackage{color}
\usepackage{epstopdf}
\usepackage{multirow}
\usepackage{svg}
\usepackage{url}
\usepackage{diagbox}
\ifpdf
  \DeclareGraphicsExtensions{.eps,.pdf,.png,.jpg}
\else
  \DeclareGraphicsExtensions{.eps}
\fi

\def\ds{\displaystyle}

\def\O{\Omega}

\def\l{\lambda}

\renewcommand\sp{\mathop{\mathrm{Sp}}\nolimits}

\newcommand{\set}[1]{\lbrace #1 \rbrace}

\newcommand{\jumpp}[1]{\left\llbracket #1 \right\rrbracket}

\newcommand{\norm}[1]{\lVert#1\rVert}

\usepackage{booktabs}
\usepackage{float}

\newcommand\bu{\boldsymbol{u}}
\newcommand\bv{\boldsymbol{v}}
\newcommand\bw{\boldsymbol{w}}

\newcommand\bn{\boldsymbol{n}}

\newcommand\curl{\mathop{\mathbf{curl}}\nolimits}

\def\hdel{\widehat{\delta}}
\def\CM{\mathcal{X}}
\def\CN{\mathcal{Y}}

\newcommand\bT{\boldsymbol{T}}

\newcommand\0{\mathbf{0}}

\newcommand\bxi{\boldsymbol{\xi}}

\def\CT{{\mathcal T}}

\newcommand\bsig{\boldsymbol{\sigma}}
\newcommand\btau{\boldsymbol{\tau}}
\newcommand\bPi{\boldsymbol{\Pi}}
\newcommand\bphi{\boldsymbol{\varphi}}

\newcommand\R{\mathbb{R}}

\renewcommand\H{\mathrm{H}}

\renewcommand\L{\mathrm{L}}

\renewcommand\O{\Omega}
\newcommand\DO{\partial\O}

\newcommand\bdiv{\mathop{\mathbf{div}}\nolimits}

\renewcommand\div{\mathop{\mathrm{div}}\nolimits}

\renewcommand\sp{\mathop{\mathrm{sp}}\nolimits}

\newcommand\LO{\L^2(\O)}

\newcommand\HsO{\H^s(\O)}

\newcommand{\vertiii}[1]{{\left\vert\kern-0.25ex\left\vert\kern-0.25ex\left\vert #1 
    \right\vert\kern-0.25ex\right\vert\kern-0.25ex\right\vert}}

\newsiamremark{remark}{Remark}
\newsiamremark{hypothesis}{Hypothesis}
\crefname{hypothesis}{Hypothesis}{Hypotheses}
\newsiamthm{claim}{Claim}

\headers{Mixed methods for the Stokes spectral problem}{Felipe Lepe,  Gonzalo Rivera and Jesus Vellojin}

\title{Error estimates for a vorticity-based velocity-stress formulation of the Stokes eigenvalue problem\thanks{Submitted to the editors DATE.
\funding{The first author was partially supported by
	DIUBB through project 2120173 GI/C Universidad del B\'io-B\'io and ANID-Chile through FONDECYT project 11200529 (Chile).
}}}

\author{Felipe Lepe\thanks{GIMNAP-Departamento de Matem\'atica, Universidad del B\'io - B\'io, Casilla 5-C, Concepci\'on, Chile. \email{flepe@ubiobio.cl}.}
\and Gonzalo Rivera\thanks{Departamento de Ciencias Exactas,
	Universidad de Los Lagos, Casilla 933, Osorno, Chile. \email{gonzalo.rivera@ulagos.cl}.}
\and Jesus Vellojin\thanks{GIMNAP-Departamento de Matem\'atica, Universidad del B\'io - B\'io, Casilla 5-C, Concepci\'on, Chile.. \email{jesus.vellojinm@usm.cl}.}}

\usepackage{amsopn}

\ifpdf
\hypersetup{
  pdftitle={Mixed methods for the Stokes spectral problem},
  pdfauthor={Felipe Lepe, Gonzalo Rivera and Jesus Vellojin}
}
\fi

\begin{document}

\maketitle

\begin{abstract}
The aim of this paper is to analyze a mixed formulation for the two dimensional Stokes eigenvalue problem where the unknowns are the stress and the velocity, whereas the  pressure can be recovered with a simple postprocess of the stress. The stress tensor is written in terms of the vorticity of the fluid, leading to an alternative  mixed formulation that incorporates this physical feature. We propose a mixed numerical method where the stress is approximated with suitable N\'edelec finite elements, whereas	the  velocity is approximated with piecewise polynomials of degree $k\geq 0$. With the aid of the compact operators theory we derive convergence of the method and spectral correctness. Moreover, we propose a reliable and efficient  a posteriori error estimator for our spectral problem. We report numerical tests in different domains, computing the spectrum and convergence orders, together with a computational analysis for the proposed estimator. In addition, we use the corresponding error estimator to drive an adaptive scheme, and we report the results of a  numerical test, that allow us to assess the performance of this approach.
\end{abstract}



\section{Introduction}\label{sec:intro}
The Stokes problem is a system of equations that describes the motion of a certain fluid. For an open domain $\O\subset\mathbb{R}^2$ with Lipschitz boundary $\partial \O$, we are interested in the Stokes eigenvalue problem
\begin{equation}\label{def:stokes_source}
	\left\{
	\begin{array}{rcll}
		-\mu\Delta\bu+\nabla p & = & \lambda\bu&  \text{ in } \quad \Omega, \\
		\div\bu & = & 0 & \text{ in } \quad \Omega, \\
		\bu & = & \boldsymbol{0} & \text{ on } \quad \partial\Omega,
	\end{array}
	\right.
\end{equation}
where $\mu>0$ is the kinematic viscosity,  $\bu$ is the velocity and $p$ is the pressure. 

In the nowadays there are a number of papers where different formulations, together with  numerical methods, have been proposed 
in order to solve \eqref{def:stokes_source} \cite{MR3197278, MR3864690, MR4071826,MR4077220, LRVSISC,Lepe2021,MR2473688, MR3335223}.  Each of these contributions are concerned in the analysis of mathematical formulations and  the development of robust numerical methods that, with high accuracy, are capable to approximate the spectrum of \eqref{def:stokes_source}, namely the eigenvalues and its associated eigenfunctions. Not only the computation of the spectrum has been a subject of study, but also adaptivity strategies when the eigenvalues are not smooth enough and hence, convergence orders of approximation are affected by this lack of regularity.

The development of numerical methods to solve eigenvalue problems, particularly \eqref{def:stokes_source}, is a current subject of study 
in the community of numerical analysis, since the knowledge of the physical eigenmodes of the spectral Stokes system are important in certain applications for the development and design of pipes, structures, containers, dams, etc.. Moreover, not only the velocity and pressure are quantities of interest, but also others as the stress, vorticity, stream functions, just to mention some of the most relevant. It is this need that motivates the analysis of mixed formulations and, hence, mixed numerical methods. These mixed methods have been plenty analyzed for load problems 
\cite{ MR4263828, MR3164557, MR4082227, MR3626409, MR2835711,MR2594823} where finite elements, virtual elements, discontinuous methods, just to mention a few,  have been considered. These methods and formulations can be explored in order to solve the Stokes spectral problem with an important accuracy.

On the other hand, adaptive mesh refinement strategies based on a posteriori error indicators play a relevant role in the numerical solution of partial differential equations in a general sense. Several approaches have been considered to design error estimators based on the residual equations (see \cite{MR1885308,MR1284252} and the references therein). In particular, for the load problem associated to the Stokes  equations, we can mention as recent developments \cite{MR4327458,MR4263661,MR4044773}, whereas for the Stokes spectral problems we refer to \cite{MR3197278,MR4022421,MR3095260,MR4372679}, and the references therein.

Now, following with our research program related to mixed formulations and their discretizations  for eigenvalue problems, the present work introduces a  formulation that is inspired in  \cite{MR2835711} for the load problem, where an  augmented method is introduced in order to approximate the velocity, pressure and stress. Despite the fact that this formulation is more expensive since the pressure is directly computed with the method instead of eliminate it, the augmented method is flexible on the choice of families of finite elements. In our case, since we are interested in the spectral problem associated to the Stokes problem, a more simple formulation is enough for this purpose, since the computational costs are reduced, and hence, the computational solvers for eigenvalue problems compute the solutions in less time without loss of accuracy. More precisely, since the solution operator is defined for the velocity component only (cf. Section \ref{sec:model_problem}),  our primary goal is to compute this unknown (and its associated eigenvalues) and then derive the others quantities of interest in postprocessing. In fact, the pressure and vorticity can be computed using a linear combination of stress and velocity.  This is a clear advantage compared with the recent work \cite{LRVSISC}, where the pressure is incorporated in the formulation and the proposed numerical methods, implying more expensive mixed methods. However, the analysis presented in the present study can be perfectly adapted to include the pressure in the preprocessing, with a consequent increase in computational cost.

The  contents of our papers are presented as follows: in Section \ref{sec:model_problem} we present the Stokes eigenvalue problem, introducing the stress which we write in terms of the vorticity, leading to a variational formulation where the unknowns are the aforementioned stress and the velocity field. We introduce the solution operator and recall some regularity properties for the eigenfunctions. In Section \ref{sec:mixed_fem} we introduce the mixed finite elements in which our method is based. Here, we present the finite element spaces, approximation properties and hence, the discrete eigenvalue problem. The discrete solution operator is also defined. In Section \ref{sec:conv_error} we develop the convergence analysis with the compact operators approach. Error estimates for the eigenvalues and eigenfunction are derived.  Section \ref{sec:apost} is dedicated to an a posteriori error analysis, where we present a residual based a posteriori error analysis, together with the corresponding reliability and efficiency for the proposed estimator. Finally, in Section \ref{sec:numerics} we report a series of numerical tests to illustrate all our theoretical results. We start with a priori results on different geometries, testing the robustness of our scheme in different orders of approximation. This is followed by an adaptivity test, where we use a non-convex geometry to verify the performance of the proposed estimator.


\subsection{Preliminaries and notations}
Given
any Hilbert space $X$, let $X^2$ and $\mathbb{X}$ denote, respectively,
the space of vectors and tensors  with
entries in $X$. In particular, $\mathbb{I}$ is the identity matrix of
$\R^{2\times 2}$, and $\mathbf{0}$ denotes a generic null vector or tensor. 
Given $\btau:=(\tau_{ij})$ and $\bsig:=(\sigma_{ij})\in\R^{2\times 2}$, 
we define, as usual, 
the tensor inner product $\btau:\bsig:=\sum_{i,j=1}^2 \tau_{ij}\sigma_{ij}$. 

Let $\O$ be a polygonal Lipschitz bounded domain of $\R^2$ with
boundary $\DO$. For $s\geq 0$, $\norm{\cdot}_{s,\O}$ stands indistinctly
for the norm of the Hilbertian Sobolev spaces $\HsO$, $\HsO^2$ or
$\mathbb{H}^s(\O):=\HsO^{2\times 2}$ for scalar, vectorial and tensorial fields, respectively, with the convention $\H^0(\O):=\LO$, $\H^0(\O)^{2}=\L^2(\O)^2$ and $\mathbb{H}^0(\O):=\mathbb{L}^2(\O)$. We also define for
$s\geq 0$ the Hilbert space 
$\mathbb{H}(\curl;\O):=\set{\btau\in\mathbb{L}^2(\O):\ \curl(\btau)\in\mathrm{L}(\O)^2}$, whose norm
is given by $\norm{\btau}^2_{\curl,\O}
:=\norm{\btau}_{0,\O}^2+\norm{\curl(\btau)}^2_{0,\O}$.
The relation $\texttt{a} \lesssim \texttt{b}$ indicates that $\texttt{a} \leq C \texttt{b}$, with a positive constant $C$ which is 
independent of $\texttt{a}$,  $\texttt{b}$ and the mesh size $h$, which will be introduced in Section \ref{sec:mixed_fem}.

Let us define the tensors
$$
\mathbb{J}:=\begin{pmatrix}0&1\\-1&0\end{pmatrix}\quad\text{and}\quad\btau^r:=\btau-\frac{1}{2}(\btau:\mathbb{J})\mathbb{J}\quad\forall\btau\in\mathbb{R}^{2\times 2},
$$
where the relation $\btau^r:\mathbb{J}=0$ holds.

Let $\boldsymbol{\varphi}=(\varphi_1,\varphi_2)^{\texttt{t}}$ and $\btau=(\tau_{ij})$ be vector- and tensor -valued fields, respectively, we define
$$
\underline{\curl}(\boldsymbol{\varphi}):=\begin{pmatrix} \displaystyle -\frac{\partial \varphi_1}{\partial x_2} &\displaystyle\frac{\partial \varphi_1}{\partial x_1}\\ \displaystyle-\frac{\partial \varphi_2}{\partial x_2} &\displaystyle\frac{\partial \varphi_2}{\partial x_1}\end{pmatrix},\,\,\,\text{and}\,\,\curl(\btau):=\begin{pmatrix}\displaystyle\frac{\partial\tau_{12}}{\partial x_1}-\frac{\partial\tau_{11}}{\partial x_2}\\
	\displaystyle\frac{\partial\tau_{22}}{\partial x_1}-\frac{\partial\tau_{21}}{\partial x_2}\end{pmatrix}.
$$
Finally, through our paper, we denote by $\div$ and $\bdiv$ the divergence operator when is applied to vectorial and tensorial fields, respectively.

\section{The model problem}
\label{sec:model_problem}

Let  $\O\subset\mathbb{R}^2$ be an open bounded domain with Lipschitz boundary $\partial\O$. Let us write  the stress tensor $\bsig$ in terms of the vorticity as follows 
$\bsig:=\mu\underline{\curl}(\bu)-p\mathbb{J}$. From this relation, we observe that the vorticity of the fluid can be recovered with the relation
$
\underline{\curl}(\bu)=\frac{1}{\mu}(\bsig+p\mathbb{J}).
$

Since $\Delta\bu=\curl(\underline{\curl}(\bu))$ and $\curl(p\mathbb{J})=\nabla p$, the first equation on system \eqref{def:stokes_source} is rewritten as $\curl(\bsig)=-\lambda\bu$ in $\O$.

On the other hand, the identity  $\div(\bu)=\underline{\curl}(\bu):\mathbb{J}$ holds, and hence,  the second equation on \eqref{def:stokes_source} is rewritten  as $\underline{\curl}(\bu):\mathbb{J}=0$ in $\O$. With these relations at hand, \eqref{def:stokes_source} now reads as follows: Find the stress $\bsig$, the velocity $\bu$ and the pressure $p$ such that 
\begin{equation}\label{def:stokes_sup}
	\left\{
	\begin{array}{rcll}
		\bsig-\mu\underline{\curl}(\bu)+p\mathbb{J}& = & \boldsymbol{0}&  \text{ in } \quad \Omega, \\
		\curl(\bsig)& = & -\lambda\bu & \text{ in } \quad \Omega, \\
		\underline{\curl}(\bu):\mathbb{J}&=&0&\text{ in } \quad \Omega, \\
		\bu & = & \boldsymbol{0} & \text{ on } \quad \partial\Omega.
	\end{array}
	\right.
\end{equation}

Algebraic manipulations reveal that the pressure satisfies $p=-1/2(\bsig:\mathbb{J})$. Hence, we can eliminate the pressure on  \eqref{def:stokes_sup} leading to the following  equivalent system
\begin{equation}\label{def:stokes_sup2}
	\left\{
	\begin{array}{rcll}
		\bsig^r-\mu\underline{\curl}(\bu)& = & \boldsymbol{0}&  \text{ in } \quad \Omega, \\
		\curl(\bsig)& = & -\lambda\bu & \text{ in } \quad \Omega, \\
		\bu & = & \boldsymbol{0} & \text{ on } \quad \partial\Omega.
	\end{array}
	\right.
\end{equation}

Now, a variational formulation for \eqref{def:stokes_sup2} reads as follows: Find $\lambda\in\mathbb{R}$ and  $\0 \neq(\bsig,\bu)\in\mathbb{H}(\curl,\O)\times \L^2(\O)^2$ such that
\begin{align}
	\label{ec1}a(\bsig,\btau)+b(\btau,\bu)&=0\,\,\,\,\quad\quad\quad\forall \btau\in\mathbb{H}(\curl,\O),\\
	\label{ec2}b(\bsig,\bv)&=-\lambda(\bu,\bv)_{0,\O}\quad\forall\bv\in \L^2(\O)^2.
\end{align}

Let us define the spaces $\mathbb{H}:=\mathbb{H}(\curl,\O)$ and $\mathbf{Q}:=\L^2(\O)^2$. With these definitions at hand, we introduce the bilinear bilinear forms $a:\mathbb{H}\times\mathbb{H}\rightarrow\mathbb{R}$ and 
$b:\mathbb{H}\times \mathbf{Q}\rightarrow\mathbb{R}$ defined as follows
$$
a(\bxi,\btau):=\frac{1}{\mu}\int_{\O}\bxi^r:\btau^r\quad 
\,\,\text{and}\,\,\,\,\,
b(\bxi,\bv):=\int_{\O}\bv\cdot\curl(\bxi)\quad\forall\boldsymbol{\xi},\btau\in\mathbb{H},\,\,\forall\bv\in\mathbf{Q}.
$$
For our analysis, let us consider the decomposition 
\begin{equation}
	\label{descomposicion-H0}
	\mathbb{H}(\curl,\O)=\mathbb{H}_0\oplus\mathbb{R}\mathbb{J},
\end{equation}
where
$$
\mathbb{H}_0:=\left\{\btau\in\mathbb{H}\,:\,\int_{\O}\btau:\mathbb{J}=0\right\}.
$$
The need of this space is motivated due the non-uniqueness of solution of \eqref{ec1}--\eqref{ec2}. To make matter precise, for any $c\in\mathbb{R}$, the duo $(c\mathbb{J},\boldsymbol{0})$ is a solution of the homogeneous problem associated to  \eqref{ec1}--\eqref{ec2}.
Now,  we write the following eigenvalue problem: find $\lambda\in\mathbb{R}$ and $\0 \neq(\bsig,\bu)\in\mathbb{H}_0\times \mathbf{Q}$ such that 
%
\begin{align}
	\label{ec1_kappa0}a(\bsig,\btau)+b(\btau,\bu)&=0\,\,\,\,\quad\quad\quad\forall \btau\in\mathbb{H}_0,\\
	\label{ec2_kappa0}b(\bsig,\bv)&=-\lambda(\bu,\bv)_{0,\O}\quad\forall\bv\in \mathbf{Q}.
\end{align}

We recall some important results that allows us to establish the well posedness of our mixed formulation. The following result is instrumental.
\begin{lemma}
	\label{lmm:bound_r}
	There exists a constant $C>0$, depending on $\O$, such that for all $\btau\in\mathbb{H}_0$ there holds
	$$
	C\|\btau\|_{0,\O}^2\leq \|\btau^r\|_{0,\O}^2+\|\curl(\btau)\|_{0,\O}^2.
	$$
\end{lemma}
\begin{proof}
	See  \cite[Lemma 2.3]{MR2835711}.
\end{proof}

Let us introduce the  the kernel of $b(\cdot,\cdot)$, defined as the  following space
$$
\mathbb{V}:=\{\btau\in\mathbb{H}_0\,:\, b(\btau,\bv)=0\,\,\forall\bv\in \L^2(\O)^2\}=\{\btau\in\mathbb{H}_0:\,\curl(\btau)=\boldsymbol{0}\,\,\text{in}\,\O\},
$$
in which, according to Lemma \ref{lmm:bound_r}, $a(\cdot,\cdot)$ is coercive (see \cite[Theorem 2.4]{MR2835711}). Also, the bilinear form $b(\cdot,\cdot)$ satisfies the  following inf-sup condition  (see \cite[Theorem 2.2]{MR2835711})
\begin{equation}
	\label{eq:inf_sup}
	\displaystyle\sup_{\boldsymbol{0}\neq\btau\in\mathbb{H}_0}\frac{\displaystyle\int_{\O}\bv\cdot\curl(\btau)}{\|\btau\|_{\curl,\O}}\geq\beta\|\bv\|_{0,\O}\quad\forall\bv\in \mathbf{Q},
\end{equation}
where $\beta$ is a positive constant.

With these ingredients at hand, we are in position to introduce the solution operator
$$
\bT:\mathbf{Q} \rightarrow \mathbf{Q} ,\qquad \boldsymbol{f}\mapsto \bT\boldsymbol{f}:=\widehat{\bu}, 
$$
where the pair  $(\widehat{\bsig}, \widehat{\bu})\in\mathbb{H}_0\times \mathbf{Q}$ is the solution of the following well posed source problem
\begin{align}
	\label{eq1_source}a(\widehat{\bsig}, \btau)+b(\btau,\widehat{\bu})&=0\,\,\,\,\quad\quad\quad\forall \btau\in \mathbb{H}_0,\\
	\label{eq2_source}b(\widehat{\bsig},\bv)&=-(\boldsymbol{f},\bv)_{0,\O}\quad\forall\bv\in \mathbf{Q},
\end{align}
implying  that $\bT$ is well defined due to the Babu\^{s}ka-Brezzi theory. Moreover, 
we have the following estimate
$$
\|\widehat{\bsig}\|_{\curl,\O}+\|\widehat{\bu}\|_{0,\O} \lesssim\|\boldsymbol{f}\|_{0,\O}.
$$
Therefore, recalling that the continuous dependence result given above is equivalent to the global inf-sup condition for the continuous formulation \eqref{eq1_source}--\eqref{eq2_source}. i.e:
\begin{equation}
	\label{eq:complete_infsup}
	\|(\boldsymbol{\tau}\hspace*{-0.045cm}, \bv)\|_{\mathbb{H}_{0}\times \mathbf{Q}}\lesssim \hspace*{-0.1cm}\displaystyle\sup_{\underset{(\boldsymbol{\xi},\boldsymbol{w})\neq \boldsymbol{0}}{(\boldsymbol{\xi},\boldsymbol{w})\in \mathbb{H}_{0}\times\mathbf{Q}}}\frac{a(\boldsymbol{\tau},\boldsymbol{\xi})\hspace*{-0.045cm}+\hspace*{-0.045cm}b(\boldsymbol{\xi},\bv)\hspace*{-0.045cm}+\hspace*{-0.045cm}b(\btau,\boldsymbol{w})}{\|(\boldsymbol{\xi},\boldsymbol{w})\|_{\mathbb{H}_{0}\times \mathbf{Q}}}.
\end{equation}

Elementary computations reveal that $\bT$ is selfadjoint respect to the $\L^2$ inner product.
We also observe that the triplet $(\lambda, (\bsig,\bu))\in\mathbb{R}\times \mathbb{H}_0\times \mathbf{Q}$ solves \eqref{ec1_kappa0}--\eqref{ec2_kappa0} if and only if $(\kappa,\bu)$ is an eigenpair of $\bT$, i.e. $\ds \bT\bu=\kappa\bu$ with $\kappa:=1/\lambda$.

From  \cite{MR975121,MR1600081} we have the following regularity result for the Stokes spectral problem.
\begin{theorem}
\label{th:reg_velocity}
If $(\bu,p,\lambda)\in \H_0^1(\Omega)^2\times \L_0^2(\Omega)\times\mathbb{R}$ solves \eqref{def:stokes_source}, there exists $s>0$ such that $\bu\in \H^{1+s}(\Omega)^2$ and $p\in \H^s(\Omega)$.
\end{theorem}

We observe that Theorem \ref{th:reg_velocity}, together with the first and second equations of \eqref{def:stokes_sup2} reveal that $\curl(\bsig)\in \H^{1+s}(\O)^2$ and $\bsig\in \mathbb{H}^{s}(\O)$, respectively. This additional regularity for the stress tensor is a key ingredient for the numerical approximation.

\begin{remark}\label{remark-del-uhat-sigmahat}
Note that the estimate
$$
\|\widehat{\bsig}\|_{s,\O}+\|\widehat{\bu}\|_{1+s,\Omega}\lesssim\|\boldsymbol{f}\|_{0,\Omega}
$$
holds. This allows us to conclude that $\bT$ is compact, where its spectrum  satisfies $\sp(\bT)=\{0\}\cup\{\mu_k\}_{k\in\mathbb{N}}$, where $\{\mu_k\}_{k\in\mathbb{N}}\in (0,1)$ is a sequence of real positive eigenvalues which converges to zero, repeated according their respective multiplicities.
\end{remark}

\section{The mixed finite element method}
\label{sec:mixed_fem}
In this section we introduce and analyze the mixed finite element method to approximate the 
eigenvalues and eigenfunctions of \eqref{ec1_kappa0}--\eqref{ec2_kappa0}. With this goal in mind, we begin by introducing a regular family of triangulations of $\O$ denoted by $\{\CT_h\}_{h>0}$. Let $h_T$ the diameter of a triangle $T$ of the triangulation and let us define $h:=\max\{h_T\,:\, T\in \CT_h\}$.

\subsection{The finite element spaces}
Let us introduce suitable spaces to approximate the stress, the velocity and pressure. For $\upsilon\geq 0$ and $B\subset\mathbb{R}^2$ being a subset of the plane we denote by $\textrm{P}_\upsilon(B)$ the space of polynomials of degree at most $\upsilon$ defined on $B$ and by $\widetilde{\textrm{P}}_\upsilon(B)$ the subspace of homogeneous polynomials of degree $\upsilon$. 

We consider the local Nédelec space of the first type and order $k\geq 0$,
$$
\mathbb{NED}_k^{(1)}(T):=\textrm{P}_{k}(T)^2\oplus\widetilde{\textrm{P}}_{k+1}(T)^2.
$$
Hence, the global Nédelec space of the first type is defined by
$$
\mathbb{NED}_k^{(1)}(\CT_h):=\left\{\btau\in \mathbb{H}(\curl,\Omega)\;:\;\btau\vert_T^\texttt{t}\in\mathbb{NED}_k^{(1)}(T),\,\forall T\in\CT_h \right\},
$$
where $\btau\vert_T^\texttt{t}$ must be understood as $(\tau_{i1},\tau_{i2})$, for $i=1,2$.

Similarly, the local Nédelec space of the second type and order $k+1$ is given by 
$$
\mathbb{NED}_{k+1}^{(2)}(T):=\textrm{P}_{k+1}(T)^{2}\,\,  \text{ with } k\geq0,
$$
whereas the corresponding global space is defined by
$$
\mathbb{NED}_{k+1}^{(2)}(\CT_h):=\left\{\btau\in \mathbb{H}(\curl,\Omega)\;:\;\btau\vert_T^\texttt{t}\in\mathbb{NED}_{k+1}^{(2)}(T),\,\forall T\in\CT_h \right\}.
$$

We also consider the the space of piecewise polynomials of degree at most $k$,
$$
\textrm{P}_k(\CT_h):=\{q\in\L^2(\O)\,:\, q|_T\in\textrm{P}_k(T)\,\,\forall T\in\CT_h\}.
$$

\noindent\begin{remark}
It is well known from the literature that  $\mathbb{RT}_{k-1}\subset \mathbb{BDM}_{k}\subset\mathbb{RT}_k$ for all $k\geq 1$ (see \cite[Section 2]{MR3097958}). This is important to notice since $\mathbb{NED}_k^{(1)}$ and $\mathbb{NED}_k^{(2)}$ are just rotated Raviart-Thomas and Brezzi-Douglas-Marini families, respectively. This allows us to conclude that 
the number of degrees of freedom per edge is the same for both finite elements. However, the number of internal degrees of freedom of  
$\mathbb{NED}_{k}^{(2)}$ elements is  less than that of standard finite elements of the same order such as $\mathbb{NED}_k^{(1)}$. A count of the internal degrees of freedom for in two dimensions gives 
$$
\mathbb{NED}_k^{(2)} :2(k-1)(k+1)\qquad\mathbb{NED}_k^{(1)} :2k(k+1).
$$
\end{remark}
\subsection{Approximation errors} 
\label{subsec:app}
In the following, some approximation results for discrete spaces are presented. To make matters precise, since we consider two spaces to approximate the stress tensor, we need to introduce suitable interpolators for the Nédelec spaces defined above. We begin with the classical approximation property for piecewise polynomials   (see \cite{MR3097958}). Let $\mathcal R_h:\LO^2\rightarrow \textrm{P}_{k}(\CT_h)^2$. The following estimate  holds
$$
\|\bv-\mathcal R_h\bv\|_{0,\O}\lesssim h^{\min\{t,k+1\}}\|\bv\|_{t,\O}\qquad\forall \bv\in\H^t(\O)^{2}\cap\LO^2.
$$
%

Let $\bPi_h^{\mathbb{NED}^{(\ell)}}:\mathbb{H}(\Omega)^t\rightarrow \mathbb{NED}_{\ell+k-1}^{(\ell)}$, with $t>1/2$, be the tensorial Nédelec interpolation operator (see \cite[Section 5.5]{monk2003finite}), where the superindex $\ell\in\{1,2\}$ represents any of the N\'edelec families that we are considering.

The following commuting diagram property holds
\begin{equation}
\label{ned1_diagram}
\curl(\bPi_h^{\mathbb{NED}^{(\ell)}}(\btau))=\mathcal{R}_h(\curl(\btau)).
\end{equation}
Moreover, the following estimate holds (see \cite[Theorem 2]{MR592160} and \cite[Proposition 3]{MR864305})
\begin{equation}
\label{asympRT}
\norm{\btau - \bPi_h^{\mathbb{NED}^{(\ell)}} \btau}_{0,\O} \lesssim h^{\min\{t, \ell+k\}} \norm{\btau}_{t,\O} \qquad \forall \btau \in \mathbb{H}^t(\O), \quad t\geq 1-(\ell-1)/2.
\end{equation}
Also,  thanks to \eqref{ned1_diagram}, if $\curl(\btau)\in\mathrm{H}^t(\O)^{2}$ with $t\geq 0$ we have the following result 
\begin{equation}\label{asympdivRT}
\norm{\curl (\btau - \bPi_h^{\mathbb{NED}^{(\ell)}} \btau) }_{0,\O}=\norm{\curl(\btau) - \mathcal R_h (\curl(\btau) ) }_{0,\O}   
\lesssim h^{\min\{t, k+1\}} \norm{\curl(\btau)}_{t,\O}.
\end{equation}

Defining $\bPi_h^{\mathbb{NED}^{(\ell)}}$ as  $\bPi_h^{\mathbb{NED}^{(\ell)}}: \mathbb{H}^t(\O)\cap \mathbb{H}(\curl,\O) \to\mathbb{NED}_k^{(\ell)}$ for all $t\in (0, 1-(\ell-1)/2]$, the following estimate holds (see \cite[Theorem 5.41]{monk2003finite} and \cite[Proposition 3]{MR864305})
\begin{equation}\label{asymp00RT}
\norm{\btau - \bPi_h^{\mathbb{NED}^{(\ell)}} \btau}_{0,\O} \lesssim h^t (\norm{\btau}_{t,\O}
+ \norm{\curl(\btau)}_{0,\O})  \quad \btau \in \mathbb{H}^t(\O)\cap \mathbb{H}(\curl,\O). 
\end{equation}

For $\ell\in\{1,2\}$, we introduce the following spaces
$$
\mathbb{H}_{0,h}:=\left\{\btau\in\mathbb{NED}_{\ell+k-1}^{(\ell)} \,:\,\,\int_{\Omega}\btau_h:\mathbb{J}=0 \right\},\qquad \mathbf{Q}_h:=\textrm{P}_k(\CT_h)^2.
$$

\subsection{Discrete eigenvalue problems}
In what follows, we present the finite element discretization  of the spectral problem 
\eqref{ec1_kappa0}--\eqref{ec2_kappa0}. With the finite element spaces defined previously, we have the following discrete problem: Find $\lambda_h$ and $\0 \neq (\bsig_h,\bu_h)\in\mathbb{H}_{0,h}\times\mathbf{Q}_h$ such that
\begin{align}
\label{ec1_h}a(\bsig_h,\btau_h)+b(\btau_h,\bu_h)&=0\,\,\,\,\quad\quad\quad\,\,\qquad\,\,\,\forall \btau_h\in\mathbb{H}_{0,h},\\
\label{ec2_h}b(\bsig_h,\bv_h)&=-\lambda_h(\bu_h,\bv_h)_{0,\O}\quad\forall\bv_h\in \mathbf{Q}_h.
\end{align}
%

Now our interest is to analyze the well posedness of \eqref{ec1_h}--\eqref{ec2_h}. With this purpose, we begin with the following discrete inf-sup  for $b(\cdot,\cdot)$, whose proof is inspired by \cite[Lemma 3.2]{MR2594823}.
\begin{lemma}\label{lm:inf-sup_h}
There exists a positive constant $\widehat{\beta}$, independent of $h$, such that 
\begin{equation}
\label{eq:inf_sup_disc}
\displaystyle\sup_{\boldsymbol{0}\neq\btau_h\in\mathbb{H}_{0,h}}\frac{\displaystyle\int_{\O}\bv_h\cdot\curl(\btau_h)}{\|\btau_h\|_{\curl,\O}}\geq\widehat{\beta}\|\bv\|_{0,\O}\quad\forall\bv_h\in \mathbf{Q}_h.
\end{equation}
\end{lemma}
\begin{proof}
Since we already have the continuous inf-sup condition \eqref{eq:inf_sup}, it will be enough to construct a Fortin operator to guarantee that $b(\cdot,\cdot)$ satisfies a discrete inf-sup condition. Indeed, let $\widetilde{\O}$ be  a convex polygonal domain such that $\O\subseteq\widetilde{\O}$. Given $\btau\in\mathbb{H}_0$, let $\boldsymbol{z}\in \H_0^1(\widetilde{\O})^2$ be the unique solution to the boundary value problem
\begin{equation}\label{eq:aux_inf}
\Delta\boldsymbol{z}=\left\{\begin{aligned}
&\curl(\btau),\quad\text{ in }\O\\
& 0, \quad \text{ in } \widetilde{\O}\backslash\O
\end{aligned}\right.,\quad \boldsymbol{z}=0,\;\text{ in } \partial\widetilde{\O}.
\end{equation}
Standard elliptic regularity results, states that the solution of \eqref{eq:aux_inf} is such that  $\boldsymbol{z}\in\H^2(\O)^{2}$ and the estimate
$$
\Vert \boldsymbol{z}\Vert_{2,\Omega}\lesssim \Vert \curl(\btau)\Vert_{0,\O},
$$
where the hidden constant depends on the domain, holds. Note that $\underline{\curl}(\boldsymbol{z})\in \mathbb{H}^1(\O)$ and  $\curl(\underline{\curl}(\boldsymbol{z}))=\Delta\boldsymbol{z}=\curl(\btau)$ in $\O$. Moreover,  we have
\begin{equation}
\label{inf-sup-discreta-001}
\Vert \underline{\curl}(\boldsymbol{z})\Vert_{1,\O}\leq \Vert\boldsymbol{z}\Vert_{2,\O}\leq \Vert \curl(\btau)\Vert_{0,\O}.
\end{equation}

Define the operator $\mathcal{F}_h:\mathbb{H}_0\rightarrow\mathbb{H}_{0,h}$ that maps  $\btau\in\mathbb{H}_0$ into its $\mathbb{H}_0$-component of $\bPi_h^{\mathbb{NED}^{(\ell)}}(\underline{\curl}(\boldsymbol{z}))$, which is determined by the decomposition \eqref{descomposicion-H0}. More precisely, we have
\begin{equation*}
\mathcal{F}_h\btau:=\bPi_h^{\mathbb{NED}^{(\ell)}}(\underline{\curl}(\boldsymbol{z})) - \left(\frac{1}{2\vert\O\vert}\int_{\O}\bPi_h^{\mathbb{NED}^{(\ell)}}(\underline{\curl}(\boldsymbol{z})):\mathbb{J}\right)\mathbb{J}.
\end{equation*}
The above, together with \eqref{ned1_diagram},  allows us to obtain
$$
\curl(\mathcal{F}_h\btau)=\curl(\bPi_h^{\mathbb{NED}^{(\ell)}}(\underline{\curl}(\boldsymbol{z})))=\mathcal{R}_h(\curl(\underline{\curl}(\boldsymbol{z}))=\mathcal{R}_h\curl(\btau).\;
$$
Applying this equivalence, we  deduce
\begin{equation}
\label{inf-sup-discreta002}
b(\mathcal{F}_h\btau,\bv_h)=\int_{\O}\boldsymbol{v}_h\cdot\curl(\mathcal{F}_h\btau)=\int_{\O}\boldsymbol{v}_h\cdot\mathcal{R}_h\curl(\btau)=\int_{\O}\boldsymbol{v}_h\cdot\curl(\btau)=b(\btau,\boldsymbol{v}_h),
\end{equation}
for all $\btau\in\mathbb{H}_0$ and for all $\boldsymbol{v}_h\in \mathbf{Q}_h$. 

On the other hand,  from the stability of the decomposition \eqref{descomposicion-H0}, \eqref{asympRT}  and \eqref{inf-sup-discreta-001}, we have that
\begin{align*}
\Vert \mathcal{F}_h\btau\Vert_{\curl,\O}^2&\leq
\Vert \bPi_h^{\mathbb{NED}^{(\ell)}}\underline{\curl}(\boldsymbol{z})\Vert_{0,\O}^2 + \Vert \curl(\bPi_h^{\mathbb{NED}^{(\ell)}}\underline{\curl}(\boldsymbol{z}))\Vert_{0,\O}^2\\
&= \Vert \bPi_h^{\mathbb{NED}^{(\ell)}}\underline{\curl}(\boldsymbol{z})\Vert_{0,\O}^2+\Vert \mathcal{R}_h\curl(\btau)\Vert_{0,\O}^2\\
&\leq \Vert \underline{\curl}(\boldsymbol{z})-\bPi_h^{\mathbb{NED}^{(\ell)}}\underline{\curl}(\boldsymbol{z})\Vert_{0,\O}^2
+\Vert \underline{\curl}(\boldsymbol{z})\Vert_{0,\O}^2+\Vert\curl(\btau)\Vert_{0,\O}^2\\
&\lesssim \Vert\curl(\btau)\Vert_{0,\O}^2,
\end{align*}
for all $\btau\in\mathbb{H}_0$. Hence $\mathcal{F}_h$ is uniformly bounded. This, along with \eqref{inf-sup-discreta002} imply that $\mathcal{F}_h$ is a  Fortin operator. This concludes the proof. 
\end{proof}

Let us introduce the discrete kernel of $b(\cdot,\cdot)$, defined by
$$
\mathbb{V}_h:=\{\btau_h\in \mathbb{H}_{0,h}\,:\, b(\btau_h,\bv_h)=0\,\,\forall\bv_h\in\mathbf{Q}_h\}=\{\btau_h\in \mathbb{H}_{0,h}\,:\,\curl(\btau_h)=\boldsymbol{0}\,\,\,\text{in}\,\,\O\}.
$$
It is easy to check $a(\cdot,\cdot)$ is coercive in $\mathbb{V}_h$. Indeed, given $\btau_h\in\mathbb{H}_{0,h}$ we have
$$
a(\btau_h,\btau_h)=\frac{1}{\mu}\|\btau_h^r\|_{0,\O}^2\geq \frac{C^2}{\mu}\|\btau_h\|_{\curl,\O}^2,
$$
where $C$ is the constant of Lemma  \ref{lmm:bound_r}. With these ingredients at hand, we are in position to introduce the discrete solution operator associated to \eqref{ec1_h}--\eqref{ec2_h}
$$
\bT_h:\mathbf{Q}\rightarrow \mathbf{Q}_h,\qquad\boldsymbol{f}\mapsto \bT_h\boldsymbol{f}:=\widehat{\bu}_h, 
$$
where $(\widehat{\bsig}_h,\widehat{\bu}_h)\in \mathbb{H}_{0,h}\times \mathbf{Q}_{h}$ is the solution of the following well posed source problem
(see \cite{MR3097958})
\begin{align}\label{reduced_disc_source1h}
a(\widehat{\bsig}_h,\btau_h)+b(\btau_h,\widehat{\bu}_h)&=0\,\,\,\,\quad\quad\quad\forall \btau_h\in\mathbb{H}_{0,h},\\\label{reduced_disc_source2h}
b(\widehat{\bsig}_h,\bv_h)&=-(\boldsymbol{f},\bv_h)\quad\forall\bv_h\in \mathbf{Q}_h.
\end{align} 
\section{Convergence and error estimates}
\label{sec:conv_error}
For the convergence analysis we take advantage of the compactness of the solution operator $\bT$ in order to obtain the convergence of $\bT_h$ to $\bT$ in norm, as $h$ goes to zero. To do this task, we resort to the well established theory of  \cite{MR1115235} for compact operators. 

We begin with the following approximation result
\begin{lemma}
\label{lm:app}
Let $\boldsymbol{f}\in\mathbf{Q}$.  Then,  the following estimate  holds
$$
\|(\bT-\bT_h)\boldsymbol{f}\|_{0,\O}\lesssim \|\widehat{\bsig}-\bPi_h^{\mathbb{NED}^{(\ell)}}(\widehat{\bsig})\|_{0,\O}+\|\widehat{\bu}-\mathcal{R}_h\widehat{\bu}\|_{0,\O},
$$
where the hidden constant are independent of $h$ and $\ell\in\{1,2\}$.
\end{lemma}
\begin{proof}
Let $\boldsymbol{f}\in\mathbf{Q}$ be such that  $\bT\boldsymbol{f}=\widehat{\bu}$ and $\bT_h\boldsymbol{f}=\widehat{\bu}_h$ where $\widehat{\bu}$ is the solution of \eqref{eq1_source}--\eqref{eq2_source} and $\widehat{\bu}_h$ is the solution of \eqref{reduced_disc_source1h}--\eqref{reduced_disc_source2h}, we have 
\begin{equation}
\label{eq_firsteqP1}
\displaystyle \|(\bT-\bT_h)\boldsymbol{f}\|_{0,\O}=\|\widehat{\bu}-\widehat{\bu}_h\|_{0,\O}\leq \|\widehat{\bu}-\mathcal R_h\widehat{\bu}\|_{0,\O}+ \|\mathcal R_h\widehat{\bu}-\widehat{\bu}_{h}\|_{0,\O}.
\end{equation}
Now our task is to control each of the terms on the right hand side of \eqref{eq_firsteqP1}. We begin with the second term.  Invoking the discrete inf-sup condition \eqref{eq:inf_sup_disc}, and setting  $\bv_{h}:=\mathcal{R}_h\widehat{\bu}-\widehat{\bu}\in \mathbf{Q}_h^{\bu}$, we obtain
$$
\|\mathcal{R}_h\widehat{\bu}-\widehat{\bu}_{h}\|_{0,\O}\leq \dfrac{1}{\beta}\displaystyle\sup_{\boldsymbol{0}\neq\btau_h\in \mathbb{H}_{0,h}}\frac{\displaystyle\int_{\Omega}\curl(\btau_h)\cdot(\mathcal{R}_h\widehat{\bu}-\widehat{\bu}_{h})}{\|\btau_h\|_{\curl,\O}}.
$$
Clearly  $\curl(\btau_{h})\in \mathbf{Q}_h$. Then, since $\mathcal{R}_h$ is the $\L^2(\O)$-orthogonal projector, and invoking \eqref{eq1_source} and   \eqref{reduced_disc_source1h},  straightforward calculations reveal
$$
b(\btau_h,\mathcal{R}_h\widehat{\bu}-\widehat{\bu}_{h})=b(\btau_h,\widehat{\bu})-b(\btau_h,\widehat{\bu}_{h})=a(\widehat{\bsig}_h,\btau_h)-a(\widehat{\bsig},\btau_h)
\lesssim \|\widehat{\bsig}_h-\widehat{\bsig}\|_{0,\O}\|\btau_h\|_{0,\O},
$$
and 
\begin{equation}
\label{eq_firsteqP2}
\|\mathcal{R}_h\widehat{\bu}-\widehat{\bu}_{h}\|_{0,\O}\lesssim \|\widehat{\bsig}_h-\widehat{\bsig}\|_{0,\O}.
\end{equation}
From the triangle inequality we have
\begin{equation}
\label{eq_firsteqP3}
\|\widehat{\bsig}-\widehat{\bsig}_{h}\|_{0,\O}\leq \|\widehat{\bsig}-\bPi_h^{\mathbb{NED}^{(\ell)}}(\widehat{\bsig})\|_{0,\O}+\|\bPi_h^{\mathbb{NED}^{(\ell)}}(\widehat{\bsig})-\widehat{\bsig}_{h}\|_{0,\O}
\end{equation}
Now, using that $\bPi_h^{\mathbb{NED}^{(\ell)}}(\widehat{\bsig})-\widehat{\bsig}_{h}\in \mathbb{H}_{0,h}$, the commuting diagram property \eqref{ned1_diagram}, together with \eqref{eq2_source} and  \eqref{reduced_disc_source2h}, we obtain $ \curl(\bPi_h^{\mathbb{NED}^{(\ell)}}$$(\widehat{\bsig}))=$$\mathcal{R}_h(\curl(\widehat{\bsig}))$$
=\mathcal{R}_h(-\boldsymbol{f})$$=\curl(\widehat{\bsig}_{h}),$
implying  directly that  
$\curl\left(\bPi_h^{\mathbb{NED}^{(\ell)}}(\widehat{\bsig})-\widehat{\bsig}_{h}\right)\in \mathbb{V}_h$. Since  $a_0(\cdot,\cdot)$ is $\mathbb{V}_h$-elliptic, there exists $\widehat{\alpha}>0$ such that 
$$
\widehat{\alpha}\|\bPi_h^{\mathbb{NED}^{(\ell)}}(\widehat{\bsig})-\widehat{\bsig}_{h}\|_{0,\O}^{2}\lesssim\| \bPi_h^{\mathbb{NED}^{(\ell)}}(\widehat{\bsig})-\widehat{\bsig}\|_{0,\O}\|\bPi_h^{\mathbb{NED}^{(\ell)}}(\widehat{\bsig})-\widehat{\bsig}_{h}\|_{0,\O}.
$$
These calculations imply that 
\begin{equation}
\label{eq_firsteqP4}
\|\bPi_h^{\mathbb{NED}^{(\ell)}}(\widehat{\bsig})-\widehat{\bsig}_{h}\|_{0,\O}
\lesssim \|\bPi_h^{\mathbb{NED}^{(\ell)}}(\widehat{\bsig})-\widehat{\bsig}\|_{0,\O}.
\end{equation}
Then, from  \eqref{eq_firsteqP1}, \eqref{eq_firsteqP2}, \eqref{eq_firsteqP3} and \eqref{eq_firsteqP4}, we have
$$
\|(\bT-\bT_h)\boldsymbol{f}\|_{0,\O} \lesssim \|\widehat{\bu}-\mathcal{R}_h\widehat{\bu}\|_{0,\O}+\|\bPi_h^{\mathbb{NED}^{(\ell)}}(\widehat{\bsig})-\widehat{\bsig}\|_{0,\O}.
$$
where the hidden constant is independent of $h$. This concludes the proof.
\end{proof}

It is important to remark that the previous result  is valid for both $\mathbb{NED}_k^{(1)}$ and $\mathbb{NED
}_{k+1}^{(2)}$ schemes, since the key ingredient to obtain the desire bound lies in the commutative diagram property that both elements satisfy.  Now,  with this result at hand, and following the proof of \cite[Corollary 4.2]{LRVSISC} together with \eqref{asympRT}--\eqref{asymp00RT}, for each finite element scheme, we have the following 
approximation result for the solution operators
\begin{equation}\label{eq:T-Th}
\|(\bT-\bT_h)\boldsymbol{f}\|_{0,\O}\lesssim h^{s}\|\boldsymbol{f}\|_{0,\O},
\end{equation}
where the hidden constant is independent of $h$.

Finally, all  the previous results, together with the application of the theory in \cite{MR0203473}, state that our numerical methods are spurious free, as is stated in the following result.
\begin{theorem}
\label{thm:spurious_free}
Let $V\subset\mathbb{C}$ be an open set containing $\sp(\bT)$. Then, there exists $h_0>0$ such that $\sp(\bT_h)\subset V$ for all $h<h_0$.
\end{theorem}

\subsection{A priori error estimates}
\label{sec:conv}
Now our aim is to obtain error estimates for the eigenfunctions and eigenvalues.
Let us remark that, according to \eqref{eq:T-Th}, if $\kappa \in (0,1)$ is an isolated eigenvalue of $\bT$ with multiplicity $m$, and $\mathcal{E}$ its associated eigenspace, then, there exist $m$
eigenvalues $\kappa_{h}^{(1)},...,\kappa_{h}^{(m)}$ of $\bT_{h}$, repeated according to their respective multiplicities, which converge to $\kappa$.
Let $\mathcal{E}_{h}$ be the direct sum of their corresponding associated eigenspaces (see \cite{MR0203473}) and let us define the  \textit{gap} $\hdel$ between two closed
subspaces $\CM$ and $\CN$ of $\L^2(\O)$ by
$$
\hdel(\CM,\CN)
:=\max\big\{\delta(\CM,\CN),\delta(\CN,\CM)\big\}, \text{ where } \delta(\CM,\CN)
:=\sup_{\underset{\left\|x\right\|_{0,\O}=1}{x\in\CM}}
\left(\inf_{y\in\CN}\left\|x-y\right\|_{0,\O}\right).
$$
With these definitions and hand, we derive the following error estimates for eigenfunctions and eigenvalues. Since the proof is direct from 
applying the  results of  \cite{MR1115235,MR1655512,MR1642801}, we do not incorporate further details.

\begin{theorem}
\label{millar2015}
For $k\geq0$, the following error estimates for the eigenfunctions and eigenvalues hold
$$
\widehat{\delta} ( \mathcal{E}, \mathcal{E}_{h} )\lesssim\,h^{ \min\{ s,k+1 \} } \quad \mbox{and} \quad | \mu-\mu_{h}(i) | \lesssim\,h^{ \min\{s,k+1 \} },
$$
where the hidden constants are independent of $h$.
\end{theorem}

Now  we improve the error estimate of  Theorem \ref{millar2015} for the eigenvalues, showing  that the order of convergence is in fact quadratic. This is contained in the following result.
\begin{theorem}
For $k\geq 0$, there exists a strictly positive constant $h_0$ such that, for $h<h_0$ there holds
$$
|\lambda-\lambda_h|\lesssim h^{2\min\{s,k+1\}},
$$
where the hidden constant is independent of $h$.
\end{theorem}
\begin{proof}
Let $(\lambda,\bsig,\bu)$ be the solution of problem \eqref{eq1_source}--\eqref{eq2_source}, where its finite element approximation $(\lambda_{h},\bsig_{h},\bu_{h})$ corresponds to the solution of problem \eqref{ec1_h}--\eqref{ec2_h} with $\|\bu_{h}\|_{0,\O}=\|\bu\|_{0,\O}=1$. Proceeding as in \cite[Lemma 4]{MR1722056}, we deduce the following identity
$$
\lambda-\lambda_h=\dfrac{1}{\mu}\|\bsig^r-\bsig_h^r\|_{0,\O}^2-\lambda_h\|\bu-\bu_h\|_{0,\O}^2,
$$ 
implying that 
$$
|\lambda-\lambda_h|\lesssim\|\bsig-\bsig_h\|_{0,\O}^2+\|\bu-\bu_h\|_{0,\O}^2,
$$
where the hidden constant is independent of $h$.
The proof is complete using the same arguments of \cite[Theorem 4.6]{LRVSISC}.
\end{proof}

Since we have proved that our method does not introduce spurious eigenvalues, it is possible to conclude  that for $h$ small enough, except 
for $\lambda_h$, the rest of the eigenvalues of  \eqref{ec1_h}--\eqref{ec2_h} are well separated from $\lambda$, as  is stated in \cite{MR3647956}.

\begin{proposition}
\label{separa_eig}
Let us enumerate the eigenvalues of problems  \eqref{ec1_h}--\eqref{ec2_h} and \eqref{ec1}--\eqref{ec2} in increasing order as follows: $0<\lambda_1\leq\cdots\lambda_i\leq\cdots$ and 
$0<\lambda_{h,1}\leq\cdots\lambda_{h,i}\leq\cdots$. Let us assume  that $\lambda_J$ is a simple eigenvalue of \eqref{ec1_h}--\eqref{ec2_h}. Then, there exists $h_0>0$ such that
$$
|\lambda_J-\lambda_{h,i}|\geq\frac{1}{2}\min_{j\neq J}|\lambda_j-\lambda_J|\quad\forall i\leq \dim\mathbb{H}_h,\,\,i\neq J,\quad \forall h<h_0.
$$
\end{proposition}

In what follows, we assume that $\lambda$ is a simple eigenvalue and we normalize $\bu$ so that $\|\bu\|_{0,\O} = 1$. Then, for all $\CT_h$, there exists a solution $(\lambda_h, \boldsymbol{\sigma}_h,\bu_h)$ be a solution of problem \eqref{ec1_h}--\eqref{ec2_h} such that $\lambda_h\rightarrow\lambda$ as $h$ goes to zero and $\|\bu_h\|_{0,\O}=1$.

We conclude this section by presenting a summary of the approximation properties for functions and eigenvalues for the lowest order.
\begin{remark}
\label{lema:apriorie}
For $k=0$, if $(\lambda, \boldsymbol{\sigma},\bu)$ is the solution of Problem \eqref{ec1_kappa0}--\eqref{ec2_kappa0} with $\|\bu\|_{0,\O}=1$ and $(\lambda_h, \boldsymbol{\sigma}_h,\bu_h)$ is the solution of problem \eqref{ec1_h}--\eqref{ec2_h} with $\|\bu_h\|_{0,\O}=1$, then
$$
\|\boldsymbol{\sigma}-\boldsymbol{\sigma}_h\|_{0,\O}+\|\bu-\bu_h\|_{0,\O}\lesssim  h^{s}
\quad
\text{and}
\quad
|\lambda-\lambda_h|\lesssim \|\boldsymbol{\sigma}-\boldsymbol{\sigma}_h\|_{0,\O}^2+\|\bu-\bu_h\|_{0,\O}^2,
$$
where the hidden constant  independent of $h$.
\end{remark}

\section{A posteriori error analysis}
\label{sec:apost}
The aim of this section is to introduce   and analysis of  an a posteriori error estimator for our single eigenpair  of the mixed eigenvalue problem.
The main difficulty in the a posteriori error analysis for eigenvalue problems is to control the so called high order terms. To do this task, we adapt the results of \cite{MR3047040} in order to obtain a superconvergence result and hence, prove the desire estimates for our estimator. The results presented in this section are limited to the lower order case $k=0$, for which the required postprocessing operator is well defined.

\subsection{Properties of the mesh}
For $T\in\mathcal{T}_h$, let $\mathcal{E}(T)$ be the set of its edges, and let $\mathcal{E}_h$ be the set of all
the edges of the triangulation $\mathcal{T}_h$. With these definitions at hand, we write $\mathcal{E}_h:=\mathcal{E}_h(\O)\cup\mathcal{E}_h(\partial\O)$, where
$$
\mathcal{E}_h(\O):=\{e\in\mathcal{E}_h\,:\,e\subseteq\O\}\quad\text{and}\quad\mathcal{E}_h(\partial\O):=\{e\in\mathcal{E}_h\,:\, e\subseteq\partial\O\}.
$$
On the other hand, for each edge $e\in\mathcal{E}_h$ we fix a unit normal vector $\bn_e$ to $e$. Moreover, given $\btau\in\mathbb{L}^2(\Omega)$ and $e\in\mathcal{E}_h(\O)$, we let $\jumpp{\btau}$ be the corresponding normal jump across $e$, that is
$
\jumpp{\btau}:=(\btau|_T-\btau|_{T'})\big|_e\bn_e,
$
where $T$ and $T'$ are two elements of the triangulation with common edge $e$.

\subsection{Technical results}
We introduce some definitions and technical results that are necessary to perform the a posteriori analysis.  We begin with the following result that is an adaptation of those presented in  \cite[Lemma 9, Lemma 10, Lemma 11]{MR3712172}. For briefty we skip the details.

\begin{corollary}
For the eigenfunction approximation $\bu_h$ of the eigenvalue problem \eqref{ec1_kappa0}--\eqref{ec2_kappa0},  the following supercloseness result holds when the mesh size $h$ is small enough,
$$
\|\bu_h-\mathcal R_h\bu\|_{0,\O}\lesssim h^s(\|\bsig-\bsig_h\|_{0,\O}+\|\bu-\bu_h\|_{0,\O}),
$$
where the hidden constant is independent of $h$.
\end{corollary}
Let us introduce the following space 
$$
\mathbf{Y}_h:=\left\{\bv\in \H^1(\Omega)^2\,:\,\bv\in \mathrm{P}_1(T)^2, \quad\forall T\in\mathcal{T}_h\right\}.
$$
Now, for each vertex $z$ of the elements in $\mathcal{T}_h$, we define the patch
$
\omega_z:=\bigcup_{z\in T\in\mathcal{T}_h} T.
$
To perform the a posteriori error analysis of our spectral problem, we introduce the so called the postprocessing operator (see \cite{ MR3047040} for instance)  defined by $\Theta_h:\mathbf{Q}\rightarrow \mathbf{Y}_h$ where, for the defined patch $\omega_z$, we fit a piecewise linear function in the average sense, for any $\bv\in\mathbf{Q}$ at the degrees of freedom of element integrations by
$$
\displaystyle\Theta_h\bv(z):=\sum_{T\in\omega_z}\frac{\displaystyle\int_T\bv\,dx}{|\omega_z|}.
$$
Here, $|\omega_z|$ denotes the measure of the patch. Moreover,  $\Theta_h$ satisfies the following properties (see \cite[Lemma 3.2, Theorem 3.3]{MR3047040}).
\begin{lemma}[Properties of the postprocessing operator]\label{postprocessing}
The operator $\Theta_h$ defined above satisfies the following:
\begin{enumerate}
\item For $\bu\in \H^{1+s}(\Omega)^2$ with $s$ as in Theorem \ref{th:reg_velocity} and $T\in\mathcal{T}_h$, there holds $$\|\Theta_h\bu-\bu\|_{0,\O}\lesssim h_T^{1+s}\|\bu\|_{1+s,\O},$$
\item $\Theta_h\mathcal{P}_h^0\bv=\Theta_h\bv$,
\item $\|\Theta_h\bv\|_{0,\O}\lesssim\|\bv\|_{0,\O}$ for all $\bv\in\mathbf{Q}$,
\end{enumerate}
where the hidden constants are  positive and independent of $h$.
\end{lemma}
The following result, proved in  \cite[Theorem 3.3]{ MR3047040} states  a superconvergence property for $\Theta_h$.

\begin{lemma}[Superconvergence]
\label{lmm:super}
For $h$ small enough, there holds 
$$
\|\Theta_h\bu_h-\bu\|_{0,\O}\lesssim h^{s}\left(\|\boldsymbol{\rho}-\boldsymbol{\rho}_h\|_{0,\O}+\|\bu-\bu_h\|_{0,\O}\right)+\|\Theta_h\bu-\bu\|_{0,\O},
$$
where the hidden constant is  independent of $h$. 
\end{lemma}

Let us introduice the bubble functions for two dimensional elements. Given $T\in\mathcal{T}_h$ and $e\in\mathcal{E}(T)$, we let $\psi_T$ and $\psi_e$ be the usual triangle-bubble and edge-bubble functions, respectively (see \cite{MR1284252,MR3059294} for further details about these functions), which satisfy the following properties
\begin{enumerate}
\item $\psi_T\in\mathrm{P}_{3}(T)$,  $\text{supp}(\psi_T)\subset T$, $\psi_T=0$ on $\partial T$ and $0\leq\psi_T\leq 1$ in $T$;
\item $\psi_e|_T\in\mathrm{P}_{2}(T)$,  $\text{supp}(\psi_e)\subset \omega_e:=\cup\{T'\in\mathcal{T}_h\,:\, e\in\mathcal{E}(T')\}$, $\psi_e=0$ on $\partial T\setminus e$ and $0\leq\psi_e\leq 1$ in $\omega_e$.
\end{enumerate}

The following results establish standard estimates for the bubble functions which will be essential for testing the efficiency of the residual estimator (see  \cite[Lemma 1.3]{MR1284252}).
\begin{lemma}[Bubble function properties]
\label{lmm:bubble_estimates}
Given $\ell\in\mathbb{N}\cup\{0\}$, and for each $T\in\mathcal{T}_h$ and $e\in\mathcal{E}(T)$, there hold
$$
\|\psi_T q\|_{0,T}^2\leq \|q\|_{0,T}^2\lesssim  \|\psi_T^{1/2} q\|_{0,T}^2\quad\forall q\in\mathrm{P}_{\ell}(T),
$$
$$
\|\psi_e L(p)\|_{0,e}^2\leq \| p\|_{0,e}^2\lesssim  \|\psi_e^{1/2} p\|_{0,e}^2\quad\forall p\in\mathrm{P}_{\ell}(e),
$$
and 
$$
h_e\|p \|_{0,e}^2\lesssim  \|\psi_e^{1/2} L(p)\|_{0,T}^2\lesssim
h_e\|p\|_{0,e}^2\quad\forall p\in\mathrm{P}_{\ell}(e),
$$
where $L: C(e)\rightarrow C(T)$ with $L(p)\in\mathrm{P}_k(T)$ and $L(p)|_e=p$ for all $p\in\mathrm{P}_k(e)$, and  the hidden constants depend on $k$ and the shape regularity of the triangulation.
\end{lemma}
Also, we requiere the following technical result (see \cite[Theorem 3.2.6]{MR0520174}).

\begin{lemma}[Inverse inequality]\label{inversein}
Let $l,m\in\mathbb{N}\cup\{0\}$ such that $l\leq m$. Then, for each $T\in\mathcal{T}_h$ there holds
$$
|q|_{m,T}\lesssim h_T^{l-m}|q|_{l,T}\quad\forall q\in\mathrm{P}_k(T),
$$
where the hidden constant depends on $k,l,m$ and the shape regularity of the triangulations.
\end{lemma}
Finally, we will make use of the well known Cl\'ement interpolation operator $I_{h}:\H^{1}(\O)\rightarrow C_I$, where $
C_I:=\{v\in \mathcal{C}(\bar{\Omega}): v|_{T}\in \textrm{P}_{1}(T) \;\ \forall T\in\CT_{h}\}
$.

The following auxiliary results, available in \cite{MR3453481}, are necessary in our forthcoming analysis.

\begin{lemma}
\label{I:clemont}
For all $v\in \H^{1}(\O)$ there holds
$$
\|v-I_{h}v\|_{0,T}\lesssim h_{T}\|v\|_{1,\omega_{T}},\quad
\|v-I_{h}v\|_{0,e}\lesssim h_{e}^{1/2}\|v\|_{1,\omega_{e}},
$$
for all $T\in\CT_{h}$ and for all $e\in \mathcal{E}_h$, where the hidden constants are independent of $h$, the set $\omega_T$ is defined by   $$\omega_{T}:=\{T'\in\CT_{h}: T' \text { and } T \text{ share an edge}\},$$ 
and $\omega_{e}:=\{T'\in\CT_{h}: e\in \mathcal{E}_{T'}\}$.
\end{lemma}

\subsection{The local and global estimators} We are now in position to introduce our local estimators for the spectral problem \eqref{ec1_h}--\eqref{ec2_h}.

The proposed estimator is  of residual type, and our goal is to prove that is reliable and efficient. 
In what follows, let $(\l_h,\boldsymbol{\sigma}_h, \bu_h)\in\R\times\mathbb{H}_{0,h}\times\mathbf{Q}_h$ be the solution of \eqref{ec1_h}--\eqref{ec2_h}. Now, for each $T\in\mathcal{T}_h$ we define the local error indicator $\eta_{T}$ as follows
\begin{multline}
\label{eq:local_eta_reduced}
\eta_{T}^2:=\|\Theta_h\bu_h-\bu_h\|_{0,T}^2
+h_T^2\left\|\underline{\boldsymbol{\curl}}(\bu_h)-\frac{1}{\mu}\bsig_h^r \right\|_{0,T}^2+h_T^2\left\|\bdiv\left(\frac{1}{\mu}\bsig_h^r \right)\right\|_{0,T}^2\\
+\sum_{e\in\mathcal{E}(T)\cap\mathcal{E}_h(\O)}h_e
\left\|
\jumpp{\frac{1}{\mu}\bsig_h^r}
\right\|_{0,e}^2+\sum_{e\in\mathcal{E}(T)\cap\mathcal{E}_{h}(\partial \O)}h_e
\left\|
\frac{1}{\mu}\bsig_h^r\bn_e
\right\|_{0,e}^2,
\end{multline}
and the respective global estimator is defined by
\begin{equation}
\label{eq:global_est_reduced}
\eta:=\left\{ \sum_{T\in\mathcal{T}_h}\eta_{T}^2\right\}^{1/2}.
\end{equation}
\subsection{Reliability}
In this section we provide an upper bound for the proposed estimator  \eqref{eq:global_est_reduced}. We begin by proving the following technical estimate.
\begin{lemma}
\label{lema_cota_s}
Let $(\l,\boldsymbol{\sigma}, \bu)\in\R\times\mathbb{H}_{0}\times \mathbf{Q}$ be the solution of \eqref{ec1_kappa0}--\eqref{ec2_kappa0}  and let $(\l_{h}$,$\boldsymbol{\sigma}_h$,$\bu_h)\in\R\times\mathbb{H}_{0,h}\times\mathbf{Q}_h$ be its finite element approximation, given as the  solution of  \eqref{ec1_h}--\eqref{ec2_h}.  Then, for all $\btau\in \mathbb{H}_{0}$, we have.
\begin{multline}
\label{eq:error_bound1}
\|\boldsymbol{\sigma}-\boldsymbol{\sigma}_{h}\|_{\curl,\O}+\|\bu-\bu_{h}\|_{0,\O}\lesssim\displaystyle\sup_{\underset{\btau\neq\boldsymbol{0}}{\btau\in\mathbb{H}_{0}}}\frac{-a(\boldsymbol{\sigma}_{h},\btau)-b(\btau,\bu_h)}{\|\btau\|_{\bdiv,\O}}\\
+\underbrace{|\lambda_{h}-\lambda |+\|\bu-\Theta_h\bu_{h}\|_{0,\O}}_{\text{h.o.t}}+\|\Theta_h\bu_{h}-\bu_{h}\|_{0,\O},
\end{multline}
where the hidden constant is independent of $h$.
\end{lemma}
\begin{proof}
Applying the inf-sup condition \eqref{eq:complete_infsup} on the errors $\boldsymbol{\sigma}-\boldsymbol{\sigma}_h$ and $\bu-\bu_h$ we have
that that 
\begin{align*}
\|(\boldsymbol{\sigma}\hspace*{-0.045cm}-\hspace*{-0.045cm}\boldsymbol{\sigma}_{h},\bu\hspace*{-0.045cm}-\hspace*{-0.045cm}\bu_{h})\|_{\mathbb{H}_{0}\times \mathbf{Q}}\lesssim &\hspace*{-0.1cm}\displaystyle\sup_{\underset{(\btau,\bv)\neq \boldsymbol{0}}{(\btau,\bv)\in \mathbb{H}_{0}\times\mathbf{Q}}}\frac{a(\boldsymbol{\sigma}\hspace*{-0.045cm}-\hspace*{-0.045cm}\boldsymbol{\sigma}_{h},\btau)\hspace*{-0.045cm}+\hspace*{-0.045cm}b(\btau,\bu\hspace*{-0.045cm}-\hspace*{-0.045cm}\bu_h)\hspace*{-0.045cm}+\hspace*{-0.045cm}b(\boldsymbol{\sigma}\hspace*{-0.045cm}-\hspace*{-0.045cm}\boldsymbol{\sigma}_h,\bv)}{\|(\btau,\bv)\|_{\mathbb{H}_{0}\times \mathbf{Q}}}\\
\lesssim &\displaystyle\sup_{\underset{\btau\neq\boldsymbol{0}}{\btau\in\mathbb{H}_{0}}}\frac{-a(\boldsymbol{\sigma}_{h},\btau)-b(\btau,\bu_h)}{\|\btau\|_{\curl,\O}}+\displaystyle\sup_{\underset{\bv\neq\boldsymbol{0}}{\bv\in\mathbf{Q}}}\frac{b(\boldsymbol{\sigma}-\boldsymbol{\sigma}_h,\bv)}{\|\bv\|_{0,\O}},
\end{align*}
where we have used \eqref{ec1_kappa0}. Now, according to the definition of the bilinear operator $b(\cdot,\cdot)$, the equation  \eqref{ec2_kappa0} and that $\curl(\boldsymbol{\sigma}_{h})=-\l_h\bu_h$,  and finally using the Cauchy–Schwarz inequality, we obtain
\begin{multline*}
\displaystyle\sup_{\underset{\bv\neq\boldsymbol{0}}{\bv\in\mathbf{Q}}}\frac{b(\boldsymbol{\sigma}-\boldsymbol{\sigma}_h,\bv)}{\|\bv\|_{0,\O}}\leq \|\lambda_h\bu_h-\lambda\bu\|_{0,\O}
\leq |\l_h-\l|\|\bu_h\|_{0,\O}+|\l|\|\bu-\bu_h\|_{0,\O}\\
\leq |\l_h-\l|\|\bu_h\|_{0,\O}+|\l|\left(\|\bu-\Theta_h\bu_h\|_{0,\O}+\|\Theta_h\bu_h-\bu_h\|_{0,\O}\right).
\end{multline*}
Then, using the above estimate and recalling that $\|\bu_h\|_{0,\O}=1$ we have 
\begin{multline*}
\|\boldsymbol{\sigma}-\boldsymbol{\sigma}_{h}\|_{\curl,\O}+\|\bu-\bu_{h}\|_{0,\O}\lesssim\displaystyle\sup_{\underset{\btau\neq\boldsymbol{0}}{\btau\in\mathbb{H}_{0}}}\frac{-a(\boldsymbol{\sigma}_{h},\btau)-b(\btau,\bu_h)}{\|\btau\|_{\curl,\O}}\\
+\underbrace{|\lambda_{h}-\lambda |+\|\bu-\Theta_h\bu_{h}\|_{0,\O}}_{\text{h.o.t}}+\|\Theta_h\bu_{h}-\bu_{h}\|_{0,\O}.
\end{multline*}
This concludes the proof.
\end{proof}
\begin{remark}\label{eq:hot}
We note that, thanks to Lemmas \ref{lema:apriorie}, \ref{postprocessing} and \ref{lmm:super}, the estimate for the high order term
$$
\text{h.o.t}\leq C h^s \left(\|\boldsymbol{\sigma}-\boldsymbol{\sigma}_{h}\|_{0,\O}+\|\bu-\bu_{h}\|_{0,\O}\right)
+\|\bu-\Theta_h\bu\|_{0,\O}
\lesssim h^{2s},
$$
holds, where the constant $C$ is uniform on $h$.
\end{remark}

Our next goal is to bound the supremum in Lemma  \ref{lema_cota_s}. To do this task,
let $\btau\in \mathbb{H}_{0}$, we proceed as in the proof of Lemma \ref{lm:inf-sup_h}, and let $\boldsymbol{z}\in \H_0^1(\widetilde{\O})^2$ be the unique weak solution of the boundary value
problem \eqref{eq:aux_inf}, where $\widetilde{\O}$ is a bounded convex polygonal domain
containing $\overline{\Omega}$. Since $\curl(\btau-\underline{\boldsymbol{\curl}} (\boldsymbol{z}))=0$ in $\O$, and $\O$ is connected, there exists $\bphi:=(\varphi_1,\varphi_2)\in \H^1(\O)^2$, with $\int_\O\varphi_1=\int_\O\varphi_2=0$, such that  $\btau=\nabla \boldsymbol{\varphi}+\underline{\boldsymbol{\curl}}{(\boldsymbol{z})}$, and we have 
\begin{equation}\label{eq:Hel_desc}
\|\boldsymbol{z}\|_{2,\O}+\|\bphi\|_{1,\O}\lesssim\|\btau\|_{\curl,\O}.
\end{equation}  Now, we  let $\bphi_h:=(I_h(\varphi_1),I_h(\varphi_2))$ and define $\btau_h\in\mathbb{H}_{h}$ as
$$
\btau_{h}:=\nabla\bphi_h+\bPi_h^{\mathbb{NED}^{(\ell)}}\left(\underline{\boldsymbol{\curl}}(\boldsymbol{z})\right)-d_{h}\mathbb{J},
$$
where  $\bPi_h^{\mathbb{NED}^{(\ell)}}$ is the N\'edelec  interpolation operator that satisfies properties  \eqref{ned1_diagram}-\eqref{asymp00RT}. The constant $d_{h}$ is chosen in the following way
$$
d_{h}:=\dfrac{1}{2|\O|}\int_{\O}\btau_{h}:\mathbb{J}=\dfrac{1}{2|\O|}\int_{\O}\left(\nabla\bphi_h+\bPi_h^{\mathbb{NED}^{(\ell)}}\left(\underline{\boldsymbol{\curl}}(\boldsymbol{z})\right)\right):\mathbb{J},
$$
in order to admit that $\btau_{h}\in \mathbb{H}_{h,0}$. Notice that  we have used the fact that  $\btau\in  \mathbb{H}_{0}$ and its Helmoltz decomposition.

As a first step to bound the supremum appearing on the right hand side of \eqref{eq:error_bound1}, we  note that for all $\boldsymbol{\xi}_{h}\in\mathbb{H}_{0,h}$ and   \eqref{ec1_h}, we have
$$
a(\boldsymbol{\sigma}_{h},\boldsymbol{\xi}_{h})+b(\boldsymbol{\xi}_{h},\bu_h)=0.
$$
On the other hand, let $\boldsymbol{\xi}\in\mathbb{H}$ be such that $$\bxi:=\btau-\btau_{h}=\nabla\bphi-\nabla\bphi_h+\underline{\boldsymbol{\curl}}(\boldsymbol{z})-\bPi_h^{\mathbb{NED}^{(\ell)}}\left(\underline{\boldsymbol{\curl}}(\boldsymbol{z})\right)+d_{h}\mathbb{J}.$$
Since $\curl(\nabla\bphi-\nabla\bphi_h)=\curl(d_{h}\mathbb{J})=0$, and invoking the commutative diagram property \eqref{ned1_diagram}, identity above is written as follows $$\curl(\bxi)=\curl(\underline{\boldsymbol{\curl}}(\boldsymbol{z})-\bPi_h^{\mathbb{NED}^{(\ell)}}\left(\underline{\boldsymbol{\curl}}(\boldsymbol{z})\right))=\curl(\underline{\boldsymbol{\curl}}(\boldsymbol{z}))-\mathcal{R}_h(\curl(\underline{\boldsymbol{\curl}}(\boldsymbol{z}))).$$
Now, since $\mathcal{R}_h$ is the $\L^2(\O)$-orthogonal projector,  we have that
$b(\bxi,\bu_h)=0.$
Therefore, from the fact that $\boldsymbol{\sigma}_{h}\in\mathbb{H}_{0,h}$ we  obtain the following identity
\begin{align*}
-\left[a(\boldsymbol{\sigma}_{h},\btau)+b(\btau,\bu_h)\right]&=-\left[a(\boldsymbol{\sigma}_{h},\bxi)+b(\bxi,\bu_h)\right]=-a(\boldsymbol{\sigma}_{h},\bxi).
\end{align*}
Now, invoking the definition of $\bxi$ and that $a(\boldsymbol{\sigma}_{h},d_{h}\mathbb{J})=d_{h}a(\boldsymbol{\sigma}_{h},\mathbb{J})=0$ we obtain  
\begin{equation}\label{eq:sup}
-\left[a(\boldsymbol{\sigma}_{h},\btau)+b(\btau,\bu_h)\right]=\underbrace{-a(\boldsymbol{\sigma}_{h},\nabla(\bphi-\bphi_h))}_{\mathfrak{T}_1}
+\underbrace{-a(\boldsymbol{\sigma}_{h},\underline{\boldsymbol{\curl}}(\boldsymbol{z})-\bPi_h^{\mathbb{NED}^{(\ell)}}\left(\underline{\boldsymbol{\curl}}(\boldsymbol{z})\right))}_{\mathfrak{T}_{2}},\end{equation}
where the terms $\mathfrak{T}_{1}$ and $\mathfrak{T}_2$ must be bounded. We begin with $\mathfrak{T}_1$.
\begin{lemma}
\label{lm_a1}
There exists certain constant independent of $h$, such that
$$
\left| \mathfrak{T}_1\right|\lesssim \left\{\sum_{T\in\CT_{h}}\eta_{T}^{2} \right\}^{1/2}\|\btau\|_{\curl,\O}.
$$
\end{lemma}
\begin{proof}
First, we note that 
\begin{align*}
\mathfrak{T}_1=-\int_{\Omega}\frac{1}{\mu}\boldsymbol{\sigma}_{h}^r:(\nabla(\bphi-\bphi_h))^r=-\int_{\Omega}\frac{1}{\mu}\boldsymbol{\sigma}_{h}^r:\nabla(\bphi-\bphi_h).
\end{align*}
Now, integrating by parts on each $T\in\CT_{h}$, we obtain that
\begin{multline*}
\mathfrak{T}_1=\int_{\Omega}-\frac{1}{\mu}\boldsymbol{\sigma}_{h}^r:\nabla(\bphi-\bphi_h)= \sum_{T\in \CT_{h}}\int_T-\frac{1}{\mu}\boldsymbol{\sigma}_{h}^r:\nabla(\bphi-\bphi_h)\\
=\sum_{T\in \CT_{h}}\int_T\bdiv\left(\frac{1}{\mu}\boldsymbol{\sigma}_{h}^r\right)\cdot(\bphi-\bphi_h)
+\sum_{e\in\mathcal{E}_h(\O)}
\int_{e}
\jumpp{\frac{1}{\mu}\bsig_h^r}\cdot(\bphi-\bphi_h)\\
+\sum_{e\in\mathcal{E}_{h}(\partial \O)}\int_{e}
\frac{1}{\mu}\bsig_h^r\bn_e
\cdot(\bphi-\bphi_h).
\end{multline*}
Applying Cauchy-Schwarz inequality, recalling that $\bphi_h:=(I_h(\varphi_1),I_h(\varphi_2))$, and invoking the
approximation properties presented in Lemma \ref{I:clemont} and estimate \eqref{eq:Hel_desc}, we have
\begin{multline*}
|\mathfrak{T}_1|
\leq \sum_{T\in \CT_{h}}h_{T}\left\|\bdiv\left(\frac{1}{\mu}\boldsymbol{\sigma}_{h}^r\right)\right\|_{0,T}\|\bphi\|_{1,\omega_{T}}
+\sum_{e\in \mathcal{E}(T)\cap\mathcal{E}_{h}(\O)}h_{e}\left\|\jumpp{\frac{1}{\mu}\bsig_h^r}\right\|_{0,e}\|\bphi\|_{1,\omega_{e}}\\
+\sum_{e\in \mathcal{E}(T)\cap\mathcal{E}_{h}(\partial \O)}h_{e}\left\| \frac{1}{\mu}\bsig_h^r\bn_e\right\|_{0,e}\|\bphi\|_{1,\omega_{e}}\lesssim \left\{\sum_{T\in\CT_{h}}\eta_{T}^{2} \right\}^{1/2}\|\btau\|_{\curl,\O},
\end{multline*}
where the hidden constant is independent of $h$ and the discrete solution. This concludes the proof.
\end{proof}

The bound for  $\mathfrak{T}_2$ is contained in the following lemma.
\begin{lemma}
\label{lm_a2}
There exists certain constant, independent of  $h$, such that
$$
\left|\mathfrak{T}_2\right|\lesssim\left\{\sum_{T\in\CT_{h}}\eta_{T}^{2} \right\}^{1/2}\|\btau\|_{\curl,\O}.
$$
\end{lemma}
\begin{proof}
Using again that $\bu_h\in \textrm{P}_0(T)^2$, for all $T\in\CT_h$, we obtain 
\begin{align*}
\int_{\Omega}\underline{\boldsymbol{\curl}}(\bu_h):\left(\underline{\boldsymbol{\curl}}(\boldsymbol{z})-\bPi_h^{\mathbb{NED}^{(\ell)}}\left(\underline{\boldsymbol{\curl}}(\boldsymbol{z})\right)\right)=0.
\end{align*}
Then, we obtain that
\begin{multline*}
\mathfrak{T}_2
=-\sum_{T\in \CT_{h}}\left[\int_{T}\left(\underline{\boldsymbol{\curl}}(\bu_h)-\frac{1}{\mu}\boldsymbol{\sigma}_{h}^r\right):\left(\underline{\boldsymbol{\curl}}(\boldsymbol{z})-\bPi_h^{\mathbb{NED}^{(\ell)}}\left(\underline{\boldsymbol{\curl}}(\boldsymbol{z})\right)\right)\right]\\
\leq\sum_{T\in \CT_{h}}\left\|\underline{\boldsymbol{\curl}}(\bu_h)-\frac{1}{\mu}\boldsymbol{\sigma}_{h}^r\right\|_{0,T}\|\underline{\boldsymbol{\curl}}(\boldsymbol{z})-\bPi_h^{\mathbb{NED}^{(\ell)}}\left(\underline{\boldsymbol{\curl}}(\boldsymbol{z})\right)\|_{0,T}\\
\lesssim\left\{\sum_{T\in \CT_{h}}h_T^2\left\|\underline{\boldsymbol{\curl}}(\bu_h)-\frac{1}{\mu}\boldsymbol{\sigma}_{h}^r\right\|_{0,T}^2\right\}^{1/2}\|\boldsymbol{z}\|_{2,\O}\lesssim\left\{\sum_{T\in\CT_{h}}\eta_{T}^{2} \right\}^{1/2}\|\btau\|_{\curl,\O},
\end{multline*}
where we have used  Cauchy-Schwarz inequality and the approximation properties \eqref{asympdivRT} and  \eqref{eq:Hel_desc}. This concludes the proof.
\end{proof}

As a consequence of Lemma \ref{lema:apriorie}, Lemma \ref{lema_cota_s}, Remark \ref{eq:hot}, estimate \eqref{eq:sup}, Lemmas \ref{lm_a1} and \ref{lm_a2}, and the definition of the local estimator $\eta_T$, we have the following result 
\begin{lemma}
Let $(\l,\boldsymbol{\sigma}, \bu)\in\R\times\mathbb{H}_{0}\times \mathbf{Q}$ be the solution of \eqref{ec1_kappa0}--\eqref{ec2_kappa0}  and let $(\l_{h}$,$\boldsymbol{\sigma}_h$,$\bu_h)\in\R\times\mathbb{H}_{0,h}\times\mathbf{Q}_h$ be its finite element approximation, given as the  solution of  \eqref{ec1_h}--\eqref{ec2_h}. Then, there exists $h_0$, such that, for all $h < h_0$, there holds.
\begin{align*}
\|\boldsymbol{\sigma}-\boldsymbol{\sigma}_{h}\|_{\curl,\O}+\|\bu-\bu_{h}\|_{0,\O}&\lesssim\left\{\sum_{T\in\CT_{h}}\eta_{T}^{2} \right\}^{1/2}+\|\bu-\Theta_h\bu\|_{0,\O},\\
|\l-\l_h |&\lesssim\sum_{T\in\CT_{h}}\eta_{T}^{2}+\|\bu-\Theta_h\bu\|_{0,\O}^2,
\end{align*}
where the hidden constants are independent of $h$.
\end{lemma}

\subsection{Efficiency}
The aim of this section is to obtain a lower bound for the local indicator \eqref{eq:local_eta_reduced}. To do this task, we will apply
the localization technique based in bubble functions, together with inverse inequalities. In order to present the material, the efficiency 
will be proved in several steps, where each one of these correspond to one of the terms of \eqref{eq:local_eta_reduced}.

Now our task is to bound each of the contributions of $\eta_T$ in \eqref{eq:local_eta_reduced}. We begin with   the term 
$$h_T^2\left\|\underline{\boldsymbol{\curl}}(\bu_h)-\frac{1}{\mu}\boldsymbol{\sigma}_{h}^r \right\|^2_{0,T}.$$
Given  an element $T\in\CT_h$, and using that $\underline{\boldsymbol{\curl}}(\bu)=\bsig^r/\mu$, let us define $\Upsilon_T:=\underline{\boldsymbol{\curl}}(\bu_h)-\boldsymbol{\sigma}_{h}^r/\mu$. Then,  invoking the properties of the  bubble function $\psi_T$  defined in Lemma \ref{lmm:bubble_estimates} we have
\begin{align*}
\|\Upsilon_T\|_{0,T}^2&\lesssim\|\psi_T^{1/2}\Upsilon_T\|_{0,T}^2=\int_T\psi_T\Upsilon_T:\left(\underline{\boldsymbol{\curl}}(\bu_h-\bu)+\frac{1}{\mu}(\bsig^r-\boldsymbol{\sigma}_{h}^r)\right)\\
&\lesssim \|\curl(\psi_T\Upsilon_T)\|_{0,T}\|\bu-\bu_h\|_{0,T}+\|\psi_T\Upsilon_T\|_{0,T}\|\bsig-\boldsymbol{\sigma}_{h}\|_{0,T}\\
&\lesssim h_T^{-1}\|\bu-\bu_h\|_{0,T}+\|\bsig-\boldsymbol{\sigma}_{h}\|_{0,T}\|\Upsilon_T\|_{0,T}.
\end{align*}
Then we have that
\begin{equation}\label{eq:term1}
h_T^2\left\|\underline{\boldsymbol{\curl}}(\bu_h)-\frac{1}{\mu}\boldsymbol{\sigma}_{h}^r \right\|^2_{0,T}\lesssim\|\bu-\bu_h\|_{0,T}+h_T^2\|\bsig-\boldsymbol{\sigma}_{h}\|_{0,T}.\end{equation}

Now we prove  the following result.
\begin{lemma}\label{lmm:invcurl}
Let $\btau_h\in\mathbb{L}^2(\O)$ be a piecewise polynomial of degree $k\geq 0$ on each $T\in\mathcal{T}_h$ such that approximates $\btau\in\mathbb{L}^2(\O)$, where  $\bdiv(\btau)=\boldsymbol{0}$ on each $T\in\mathcal{T}_h$. Then, there holds 
$$
\|\bdiv(\btau_h)\|_{0,T}\lesssim  h_T^{-1}\|\btau-\btau_h\|_{0,T}\quad \forall T\in\mathcal{T}_h,
$$
where the hidden constant is independent of $h$.
\end{lemma}
\begin{proof}
From the bubble functions properties of Lemma \ref{lmm:bubble_estimates}, integrating by parts, using the fact that
$\psi_T= 0$ on $\partial T$, and applying Cauchy-Schwarz inequality, we obtain
\begin{align*}
\|\bdiv(\btau_h)\|_{0,T}^2&\lesssim \|\psi_T^{1/2}\bdiv(\btau_h)\|_{0,T}^2=\int_T\psi_T\bdiv(\btau_h)\cdot\bdiv(\btau_h-\btau)\\
&=-\int_T\nabla\left(\psi_T\bdiv(\btau_h)\right):(\btau_h-\btau)\lesssim \|\nabla\left(\psi_T\bdiv(\btau_h)\right)\|_{0,T}\|\btau_h-\btau\|_{0,T}\\
&\lesssim h_T^{-1}\|\btau_h-\btau\|_{0,T}\|\psi_T\bdiv(\btau_h)\|_{0,T}\lesssim h_T^{-1}\|\btau_h-\btau\|_{0,T}\|\bdiv(\btau_h)\|_{0,T}.
\end{align*}
This conclude the proof.
\end{proof}
\begin{lemma}\label{lmm:invcurl_fron}
Let $\btau_h\in\mathbb{L}^2(\O)$ be a piecewise polynomial of degree $k\geq 0$ on each $T\in\mathcal{T}_h$ such that approximates $\btau\in\mathbb{L}^2(\O)$, where  $\bdiv(\btau)=\boldsymbol{0}$ on each $T\in\mathcal{T}_h$. Then, there holds
$$
\left\|\jumpp{\btau_h}\right\|_{0,e}\lesssim h_e^{-1/2}\|\btau-\btau_h\|_{0,\omega_e}\quad \forall e\in\mathcal{E}_h,
$$
where the hidden constant is independent of $h$.
\end{lemma}
\begin{proof}
Given an edge $e\in \mathcal{E}_h$, we denote by $\bw_{h}:=\jumpp{\btau_{h}}$  the corresponding jump of $\btau_{h}$. Then, employing Lemma \ref{lmm:bubble_estimates} and integrating by parts on each triangle, we obtain
of $\omega_{e}$, we obtain
\begin{multline*}
\|\bw_{h}\|_{0,e}^2\lesssim \|\psi_e^{1/2}\bw_{h}\|_{0,e}^2=\|\psi_e^{1/2}L(\bw_{h})\|_{0,e}^2=\int_e\psi_e L(\bw_{h})\cdot\jumpp{\btau_h}\\
=\int_{\omega_{e}}\bdiv(\btau_{h})\cdot\psi_e L(\bw_{h})+\int_{\omega_{e}}\btau_{h}:\nabla\left(\psi_eL(\bw_h)\right).
\end{multline*}
Now, since $\jumpp{\btau}=\boldsymbol{0}$, we have
$$
0=\int_{\omega_{e}}\bdiv(\btau)\cdot\psi_e L(\bw_{h})+\int_{\omega_{e}}\btau:\nabla\left(\psi_eL(\bw_h)\right).
$$
Thus, we have the following estimate
\begin{multline*}
\|\bw_{h}\|_{0,e}^2\lesssim \int_{\omega_{e}}\bdiv(\btau_{h}-\btau)\cdot\psi_e L(\bw_{h})+\int_{\omega_{e}}(\btau_{h}-\btau):\nabla\left(\psi_eL(\bw_h)\right)\\
\lesssim \|\bdiv(\btau_{h})\|_{0,\omega_{e}}\|\psi_e L(\bw_{h})\|_{0,\omega_{e}}+\|\btau_{h}-\btau\|_{0,\omega_{e}}\|\nabla\left(\psi_eL(\bw_h)\right)\|_{0,\omega_{e}}.
\end{multline*}
Now, applying Lemma \ref{lmm:invcurl} to each element of $\omega_{e}$, using that $h_{T_{e}}^{-1}\leq h_{e}^{-1}$, together with Lemmas \ref{inversein} and \ref{lmm:bubble_estimates}, we obtain
\begin{align*}
\|\bw_{h}\|_{0,e}^2\lesssim h_{e}^{-1/2}\|\btau_{h}-\btau\|_{0,\omega_{e}}\|\bw_h\|_{0,e}.
\end{align*}
This conclude the proof.
\end{proof}

As a consequence of the above lemma, we have the following results
\begin{equation}
h_T^2\left\|\bdiv\left(\dfrac{1}{\mu}\bsig_h^r\right)\right\|_{0,T}^2\lesssim \|\bsig-\bsig_h\|_{0,T}^2,\,\,\,\,\,\text{and}\,\,\,
\quad h_e\left\|\jumpp{\dfrac{1}{\mu}\bsig_h^r}\right\|_{0,e}^2\lesssim\|\bsig-\bsig_h\|_{0,\omega_e}^2,\label{eq:term22}
\end{equation}
for all  $e\in \mathcal{E}_{h}(\O)$, and the hidden constants are independent of $h$.
Finally, for the  term $\|\Theta_h\bu_h-\bu_h\|_{0,T}^2$, we add and subtract $\Theta_h\bu$ and $\bu$, apply triangle inequality, and Lemma \ref{postprocessing}, leading to 	
\begin{equation}
\label{eq:super22}
\|\Theta_h\bu_h-\bu_h\|_{0,T}^2\lesssim\|\bu-\bu_h\|_{0,T}^2+\|\Theta_h\bu_h-\Theta_h\bu\|_{0,T}^2+\|\Theta_h\bu-\bu\|_{0,T}^2.
\end{equation}
Note that the last term of \eqref{eq:super22} is asymptotically negligible thanks to Lemma \ref{postprocessing}.

Gathering the previous results, namely \eqref{eq:term1}--\eqref{eq:super22}, we are in a position to establish the efficiency $\eta$, which is stated in the following result.
\begin{theorem}[Efficiency]
The following estimate holds
$$
\eta^2:=\sum_{T\in\CT_{h}}\eta_{T}^{2} \lesssim\|\bu-\bu_h\|_{0,\Omega}^2 + \|\boldsymbol{\sigma}-\boldsymbol{\sigma}_h\|_{0,\Omega}^2+\text{h.o.t},
$$
where the hidden constant is independent of $h$ and the discrete solution.
\end{theorem}
\begin{proof}
The proof is  a consequence of \eqref{eq:term1}--\eqref{eq:term22} and Lemma \ref{postprocessing}.
\end{proof}
\begin{remark}
Through our paper, we have considered a formulation that eliminates the pressure, which can be recovered by a postprocess of the stress tensor.
However, it is possible to consider a formulation in terms of the velocity, pressure and velocity as the one studied in \cite{MR2835711} for the source problem. This leads to a more expensive finite element scheme, but flexible in the choice of finite elements. All the computations that we performed along our paper, can be replicated to this formulation that incorporates the pressure.
\end{remark}

\section{Numerical experiments}
\label{sec:numerics}
In this section we report some numerical tests in order to assess the performance of the proposed mixed element methods. We divide this section into two parts: in the first part, we are interested in the computation of the spectrum and the order of convergence for the eigenvalues. This is with the goal to verify the accuracy of the methods  and compare the methods. The second part is related to assess  the performance of the proposed a posteriori error estimator.


We have implemented the discrete eigenvalue problem in a FEniCS code \cite{logg2012automated,AlnaesBlechta2015a}. The rates of convergence have been computed with a least-square fitting. 

With the computed results at hand, we compare the schemes that only differ on the $\mathbb{H}(\curl,\O)$ finite element space. 
In what follows, $N$ denotes the mesh resolution, with $h\sim N^{-1}$, and $\text{dof}$ denotes the degrees of freedom, which will depends on the numerical scheme used.

In each test we plot selected eigenfunctions. The velocity field is recovered directly from solving the eigenproblem, whereas the pressure and vorticity are recovered in postprocessing by
$$
p_h=-\frac{1}{2}(\bsig_h:\mathbb{J}),\qquad \underline{\curl}(\bu_h)=\frac{1}{\mu}(\bsig_h+ p_h\mathbb{J}).
$$
Finally, we denote by $\mathrm{P}_k^{2}\text{-}\mathbb{NED}^{(\ell)}_{\ell+k-1}$, with $\ell\in\{1,2\}$ and $k=0,1,2$, the numerical scheme using piecewise elements of order $k$ to approximate $\bu$ and the Nédelec family $\mathbb{NED}^{(\ell)}_{\ell+k-1}$ of order $\ell+k-1$ to approximate $\bsig$.
\subsection{Test 1: Square}
In this test we consider as computational domain the square $\O:=(-1,1)^2$, where the number of elements scales as $2N^2$. Examples of meshes used in the example are depicted in Figure \ref{meshes_square}.
\begin{figure}
\centering
\begin{minipage}{0.35\linewidth}
\centering\includegraphics[scale=0.15,trim= 0 5.5cm 0 5.5cm, clip]{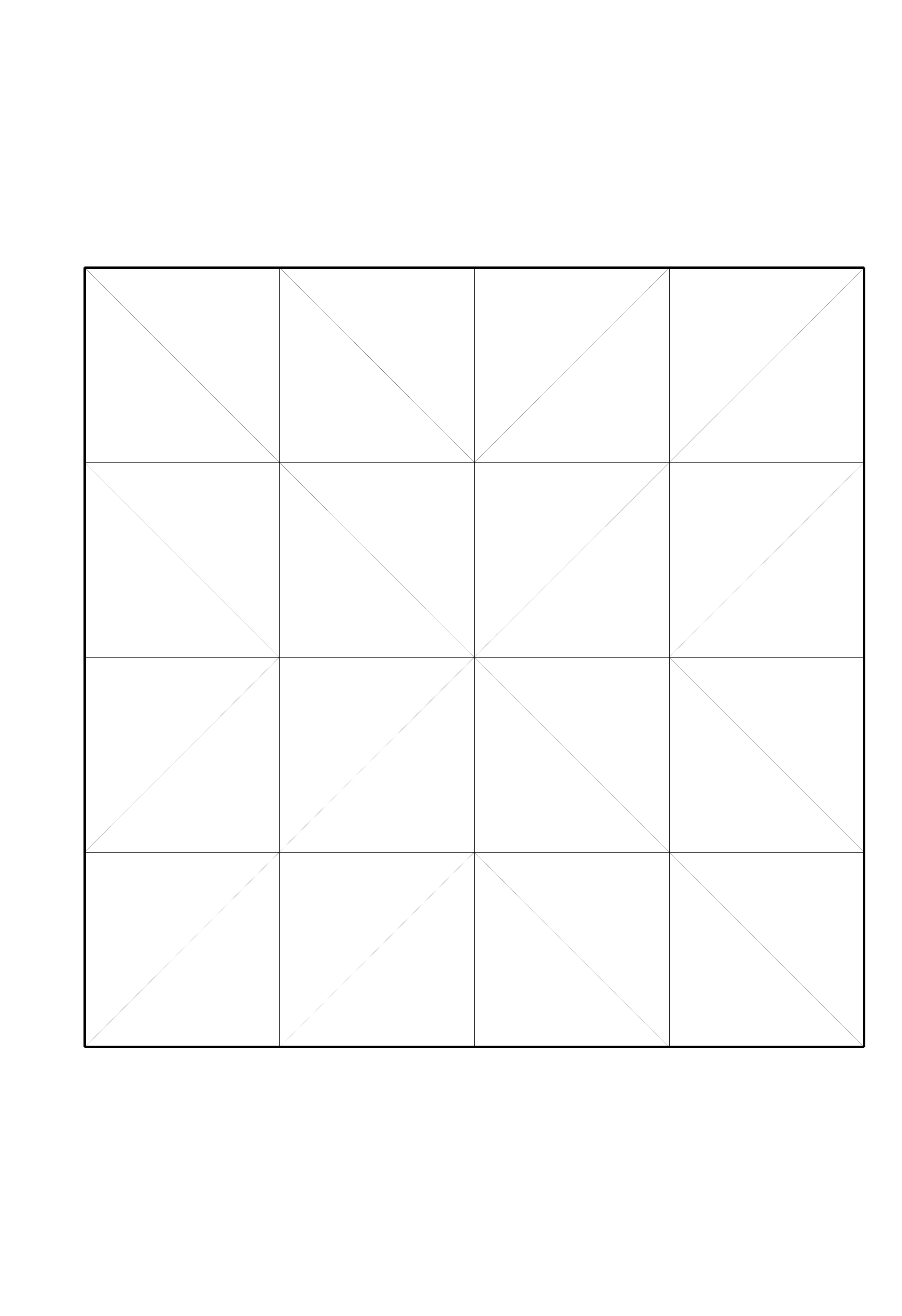}
\end{minipage}
\begin{minipage}{0.35\linewidth}
\centering\includegraphics[scale=0.15,trim= 0 5.5cm 0 5.5cm, clip]{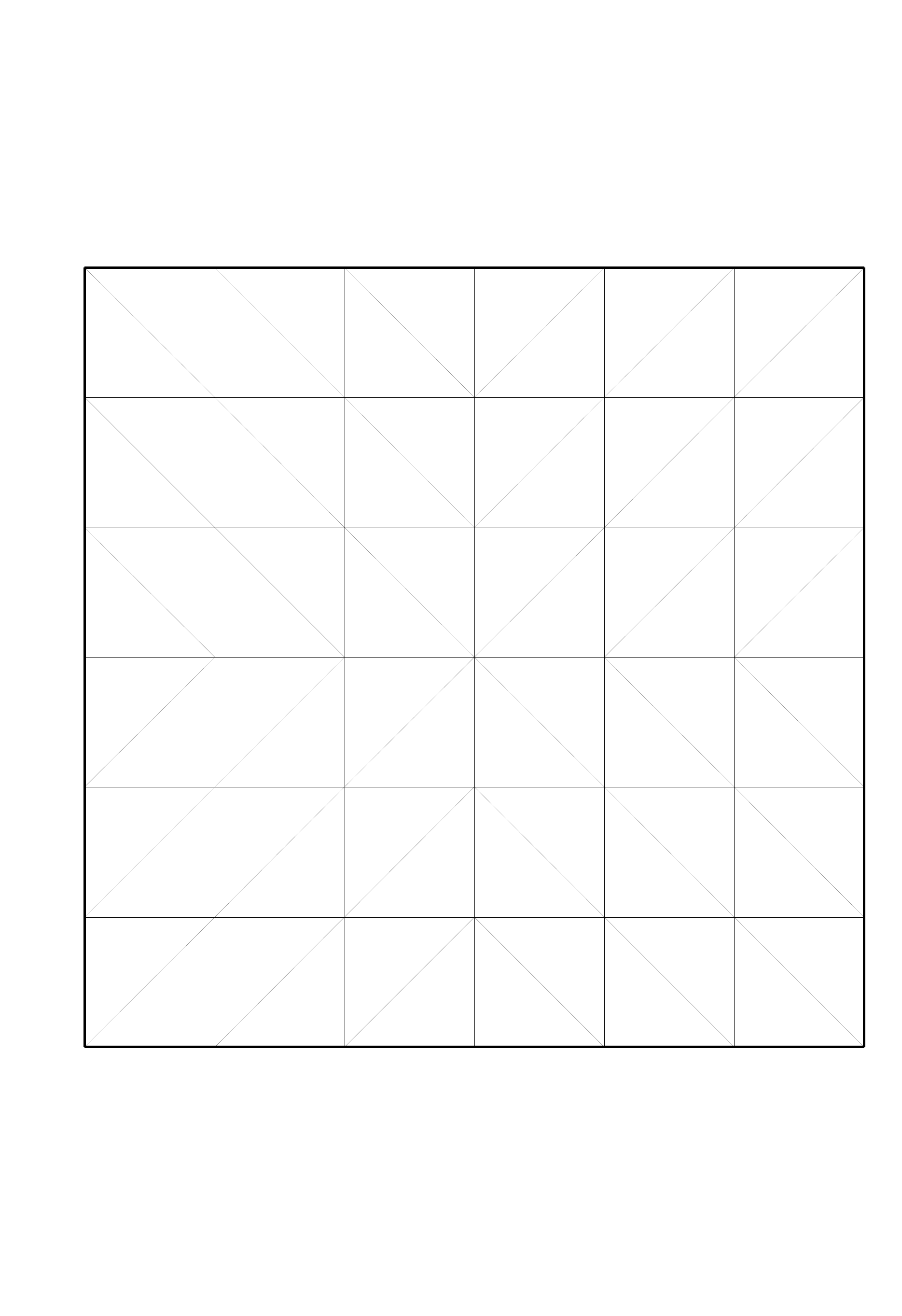}
\end{minipage}
\caption{Test 1. Examples of the meshes used in the unit square.}
\label{meshes_square}
\end{figure}
The convexity of this domain allows to obtain sufficiently smooth eigenfunctions. This implies that the convergence rates will be optimal, i.e., a behavior $\mathcal{O}(h^{2(k+1)})$, for $k=0,1,2$, is expected.
\begin{table}
{\footnotesize
\begin{center}
\caption{Test 1. Lowest computed eigenvalues for polynomial degrees $k=0, 1, 2$ using the $\mathrm{P}_k^{2}\text{-}\mathbb{NED}^{(1)}_k$   scheme. }
\begin{tabular}{c |c c c c |c| c|c}
\toprule
$k $        & $N=20$             &  $N=30$         &   $N=40$         & $N=50$ & Order & $\lambda_{extr}$&\cite{MR2473688} \\ 
\midrule
& 13.07172& 13.07948  &  13.08235   & 13.08371 & 1.88& 13.08636 &13.086    \\
& 22.92407& 22.98365  &  23.00442   & 23.01402 & 2.03 &23.03084 &23.031   \\
\multirow{2}{0.15cm}{0}
& 22.92407& 22.98365  &  23.00442   & 23.01402&  2.03&23.03084&23.031    \\
& 31.92158& 31.99380   & 32.01930  & 32.03116&  2.00& 32.05232&32.053   \\
& 38.18216& 38.37946   & 38.44657  & 38.47729&  2.09& 38.52901 &38.532   \\

\hline

&13.08610&13.08615&13.08616 &13.08617&  3.56&13.08617&13.086      \\
&23.03127&23.03112&23.03110 &23.03110&  4.60&23.03109&23.031     \\
\multirow{2}{0.15cm}{1} 
&23.03127&23.03112&23.03110 &23.03110&  4.60&23.03109&23.031    \\
&32.05268&32.05242&32.05240 &32.05239&  5.59&32.05239&32.053     \\
&38.53319&38.53172&38.53147 &38.53141&  4.04&38.53136&38.532     \\

\hline

& 13.08617 &13.08617&13.08617&13.08617  &  5.79 &13.08617& 13.086   \\
& 23.03109 &23.03109&23.03109&23.03109  &  5.71 &23.03109& 23.031      \\
\multirow{2}{0.15cm}{2}   
& 23.03109 &23.03109&23.03109&23.03109 &   5.71 &23.03109 & 23.031     \\
& 32.05238 &32.05239&32.05239&32.05239 &   5.42 &32.05239& 32.053    \\
& 38.53137 &38.53136&38.53136&38.53136 &   6.02 &38.53136& 38.532    \\

\bottomrule             
\end{tabular}
\label{tabla:square}
\end{center}}

\end{table}
\begin{table}
{\footnotesize
\begin{center}
\caption{Test 1. Lowest computed eigenvalues for polynomial degrees $k$$\,=\,$$0$,$\,1$,$\,2$  using the $\mathrm{P}_k^{2}\text{-}\mathbb{NED}^{(2)}_{k}$   scheme. }
\begin{tabular}{c |c c c c |c| c| c}
\toprule
$k $        & $N=20$             &  $N=30$         &   $N=40$         & $N=50$ & Order & $\lambda_{extr}$&\cite{MR2473688} \\ 
\midrule
& 13.18088& 13.12837  &   13.10993  & 13.10138 & 2.02& 13.08631 &13.086    \\
& 23.32433& 23.16178  &   23.10467  & 23.07821 & 2.02 &23.03156 &23.031   \\
\multirow{2}{0.15cm}{0}
& 23.32433& 23.16178  &   23.10467  & 23.07821&  2.02&23.03156&23.031    \\
& 32.61702& 32.30485   &  32.19470  & 32.14355& 2.00& 32.05163&32.053   \\
& 39.35261& 38.89728   &  38.73736  & 38.66325&  2.02&38.53259 &38.532   \\

\hline

&13.08642&13.08622&13.08618 &13.08617&  4.00&13.08617& 13.086      \\
&23.03240&23.03135&23.03118 &23.03113&  4.00&23.03109& 23.031     \\
\multirow{2}{0.15cm}{1} 
&23.03240&23.03135&23.03118 &23.03113&  4.00&23.03109& 23.031    \\
&32.05619&32.05315&32.05263 &32.05249&  3.98&32.05239& 32.053     \\
&38.53707&38.53250&38.53172 &38.53151&  4.00&38.53136& 38.532     \\

\hline

& 13.08617 &13.08617&13.08617&13.08617  &  6.13&13.08617& 13.086   \\
& 23.03110 &23.03109&23.03109&23.03109  &  6.06 &23.03109& 23.031      \\
\multirow{2}{0.15cm}{2}   
& 23.03110 &23.03109&23.03109&23.03109 &  6.06  &23.03109 & 23.031     \\
& 32.05240 &32.05239&32.05239&32.05239 &   6.04 &32.05239& 32.053    \\
& 38.53138 &38.53136&38.53136&38.53136 &   6.00 &38.53136& 38.532    \\
\bottomrule             
\end{tabular}\label{tabla:square-NED2}
\end{center}}
\end{table}

In Table \ref{tabla:square} we report the first five eigenvalues computed with the $\mathrm{P}_k^{2}\text{-}\mathbb{NED}^{(1)}_k$   scheme, considering several refinement levels and $k=0,1,2$. The column 
$\lambda_{extr}$ shows extrapolated values, obtained with a least square fitting. The values are compared those of \cite{MR2473688}, where a similar experiment was performed.

For $k=0,1$ the order of approximation  is clearly $\mathcal{O}(N^{-(k+1)})$. Meanwhile for $k=2$ the computed convergence for the fourth eigenvalue is lower than optimal. This is expected since for this numerical scheme, the computed eigenvalues on each refinement are very close to the extrapolated value, which affect the convergence rate. However, we observe that they match with those from the literature.

On the other hand, Table \ref{tabla:square-NED2} shows the computed eigenvalues when using the $\mathrm{P}_k^{2}\text{-}\mathbb{NED}^{(2)}_{k+1}$   scheme,  where we observe that an optimal rate of convergence is reached for all choices of $k$. In this case, the deterioration of the convergence order for $k=2$ is not observed since the eigenvalues calculated with all possible decimal places always remain at a sufficient distance from the extrapolated value. This suggests a superior stability of the $\mathrm{P}_k^{2}\text{-}\mathbb{NED}^{(2)}_{k+1}$ scheme at higher orders. 

In Figure \ref{figura:error-cuadrado}, a comparison of the error behavior between the two schemes is observed. We report curves for $k=1,2$ since for  $k=0$ the results are similar. Here, we consider the relative errors $e_{\lambda_i}$, for $i=1,...,5$, where
$$
e_{\lambda_i}:=\frac{\vert \lambda_{h_i}-\lambda_{extr_i}\vert}{\vert \lambda_{extr_i}\vert}.
$$
Also, we denote by $e_{\lambda_i}(\mathbb{NED}_k^{(1)})$ and $e_{\lambda_i}(\mathbb{NED}_{k+1}^{(2)})$ the relative errors obtained using $\mathrm{P}_k^{2}\text{-}\mathbb{NED}_k^{(1)}$   and $\mathrm{P}_k^{2}\text{-}\mathbb{NED}_{k+1}^{(2)}$   schemes, respectively. It is clear that the slopes of the methods behaves like $\mathcal{O}(h^{2(k+1)})$.

For completeness, in Figure \ref{figura:modos-cuadrado} we depict the velocity field and the postprocessed pressure on the square domain for the lowest computed eigenvalue. In Figure \ref{figura:modos-cuadrado-vorticidad} we present the postprocessed vorticity components for the first computed eigenfunction.
\begin{figure}
\centering\begin{minipage}{0.48\linewidth}
\centering\includegraphics[scale=0.45]{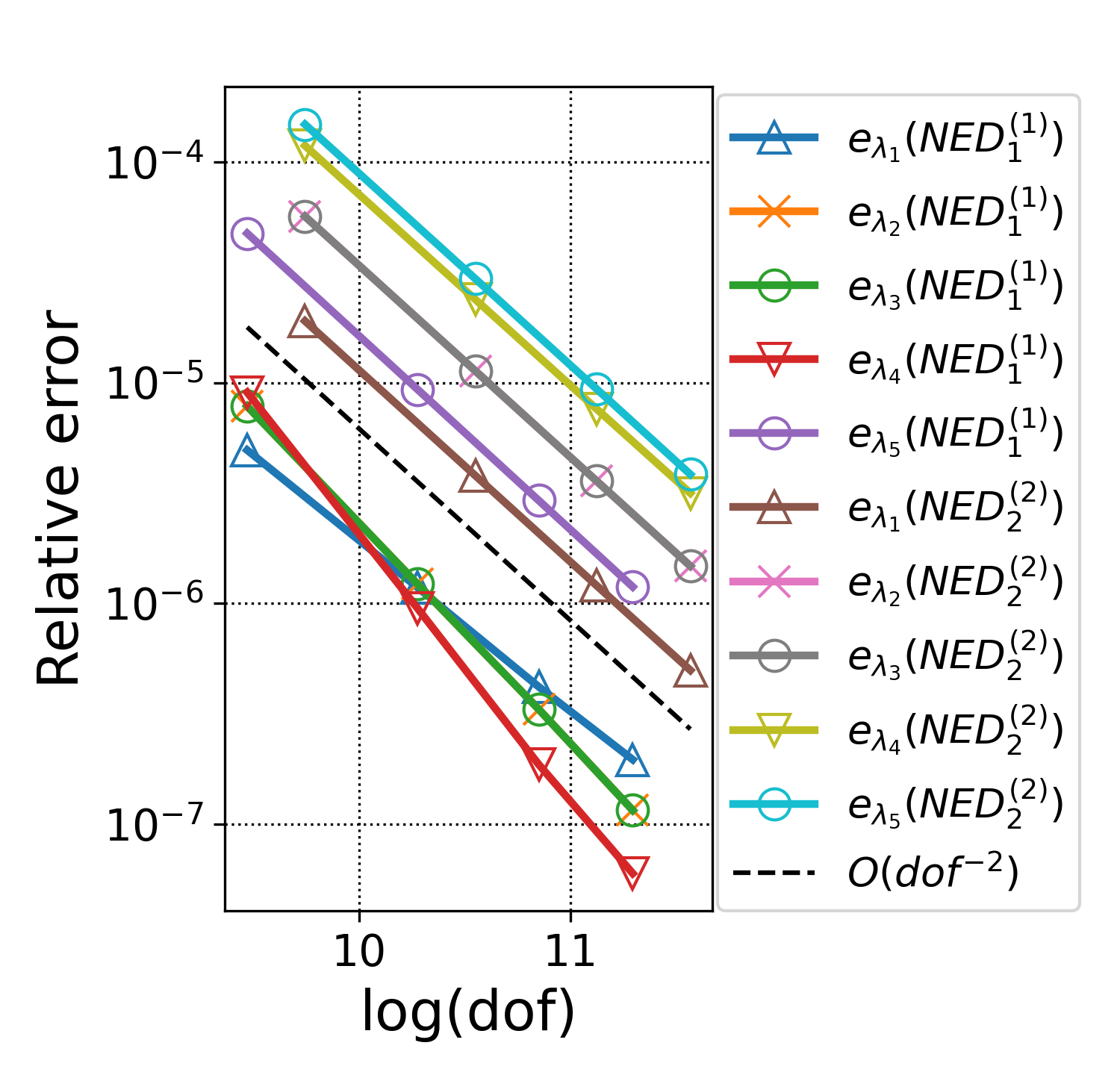}
\end{minipage}
\begin{minipage}{0.48\linewidth}
\centering\includegraphics[scale=0.45]{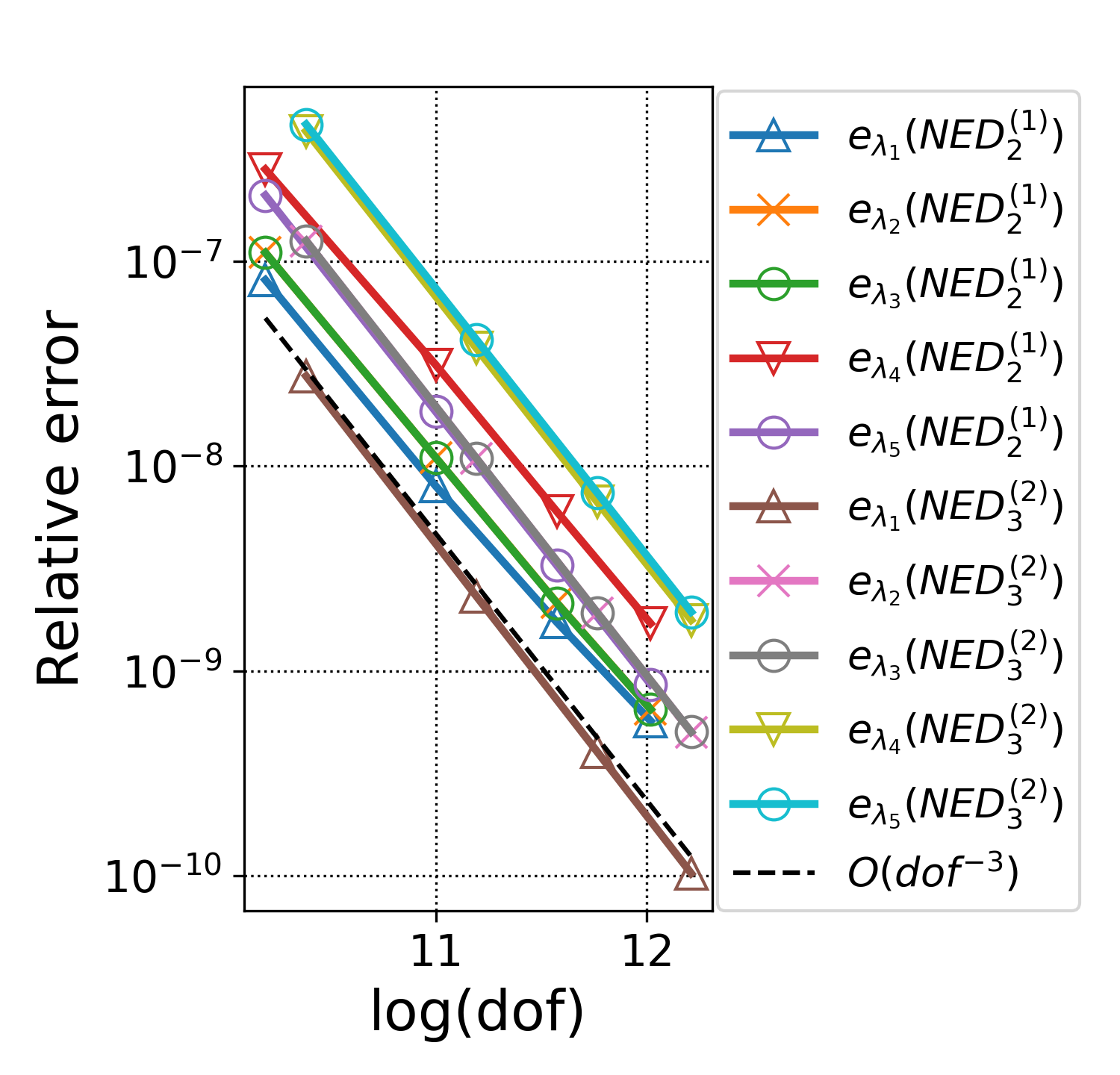}
\end{minipage}
\caption{Test 1. Comparison of the eigenvalues error curves in the square domain using $\mathrm{P}_k^{2}\text{-}\mathbb{NED}^{(1)}_{k}$ and $\mathrm{P}_k^{2}\text{-}\mathbb{NED}^{(2)}_{k+1}$, for $k=1,2$.}
\label{figura:error-cuadrado}
\end{figure}
\begin{figure}
\centering\begin{minipage}{0.45\linewidth}
\centering\includegraphics[scale=0.065, trim= 32cm 3cm 32cm 3cm, clip]{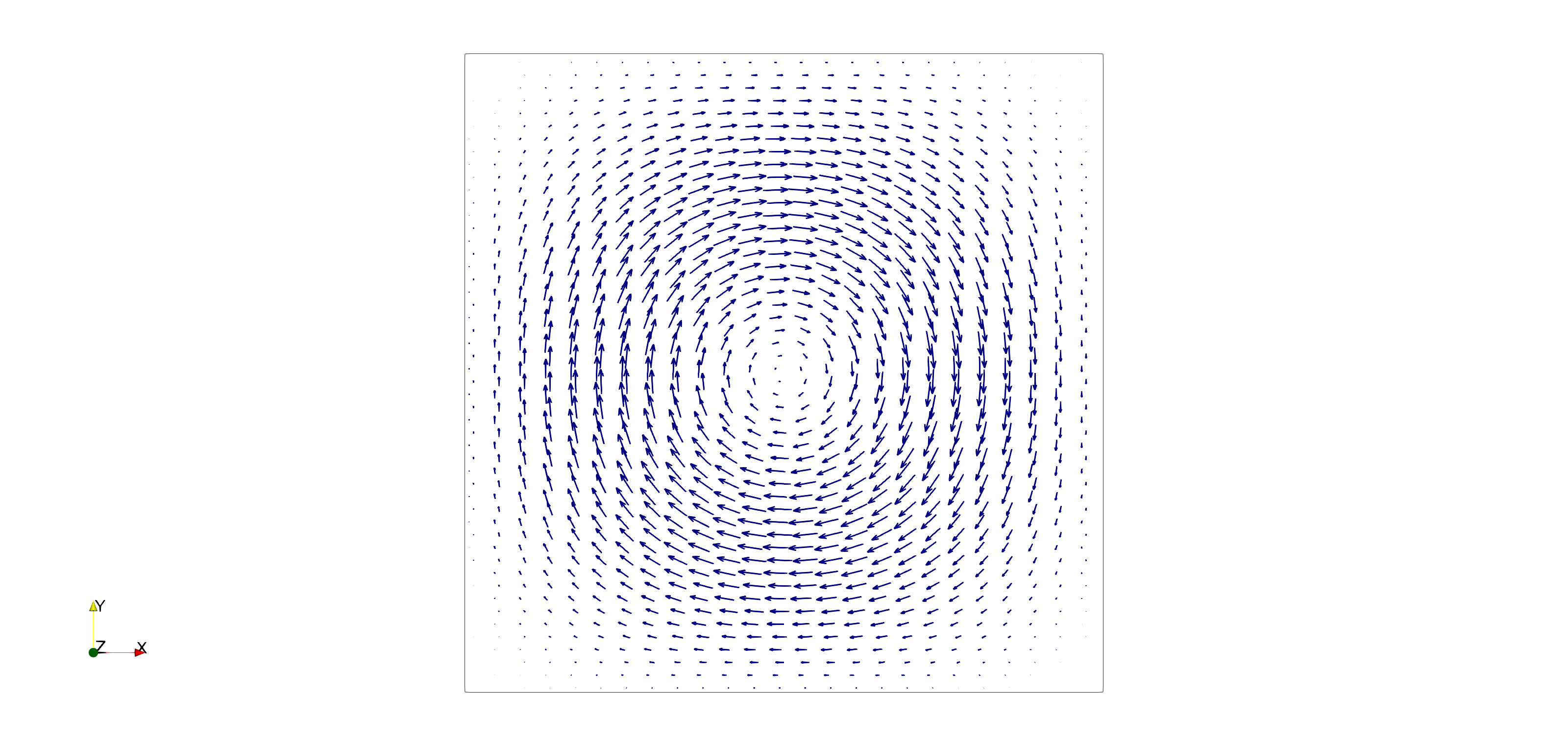}
\end{minipage}
\begin{minipage}{0.45\linewidth}
\centering\includegraphics[scale=0.065, trim= 32cm 3cm 32cm 3cm, clip]{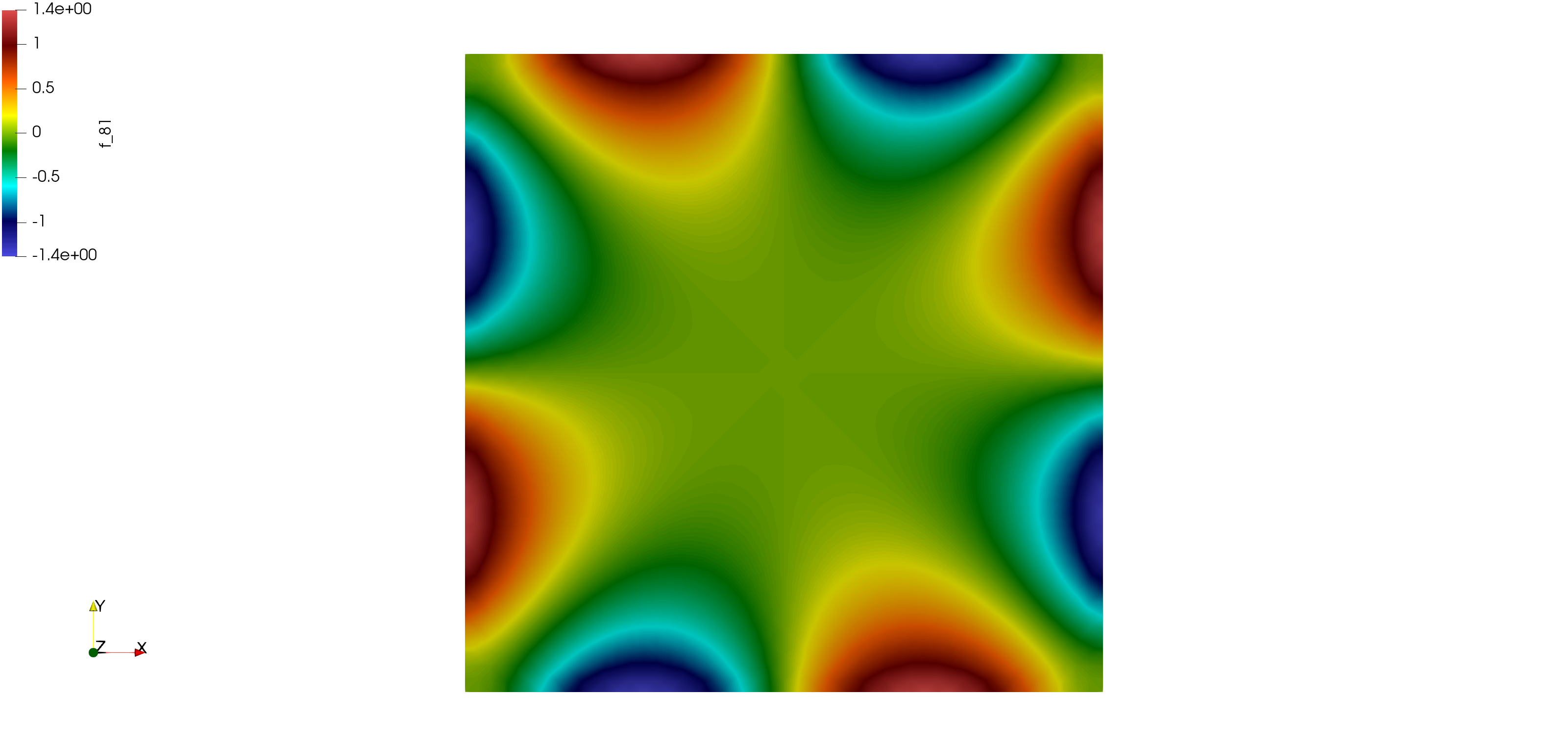}
\end{minipage}
\caption{Test 1. Approximate velocity field $\bu_h$ (left) and postprocessed pressure $p_h$ (right), corresponding to the first eigenvalue in the square domain.}
\label{figura:modos-cuadrado}
\end{figure}
\begin{figure}
\centering\begin{minipage}{0.32\linewidth}
\centering\includegraphics[scale=0.065, trim= 32cm 3cm 32cm 3cm, clip]{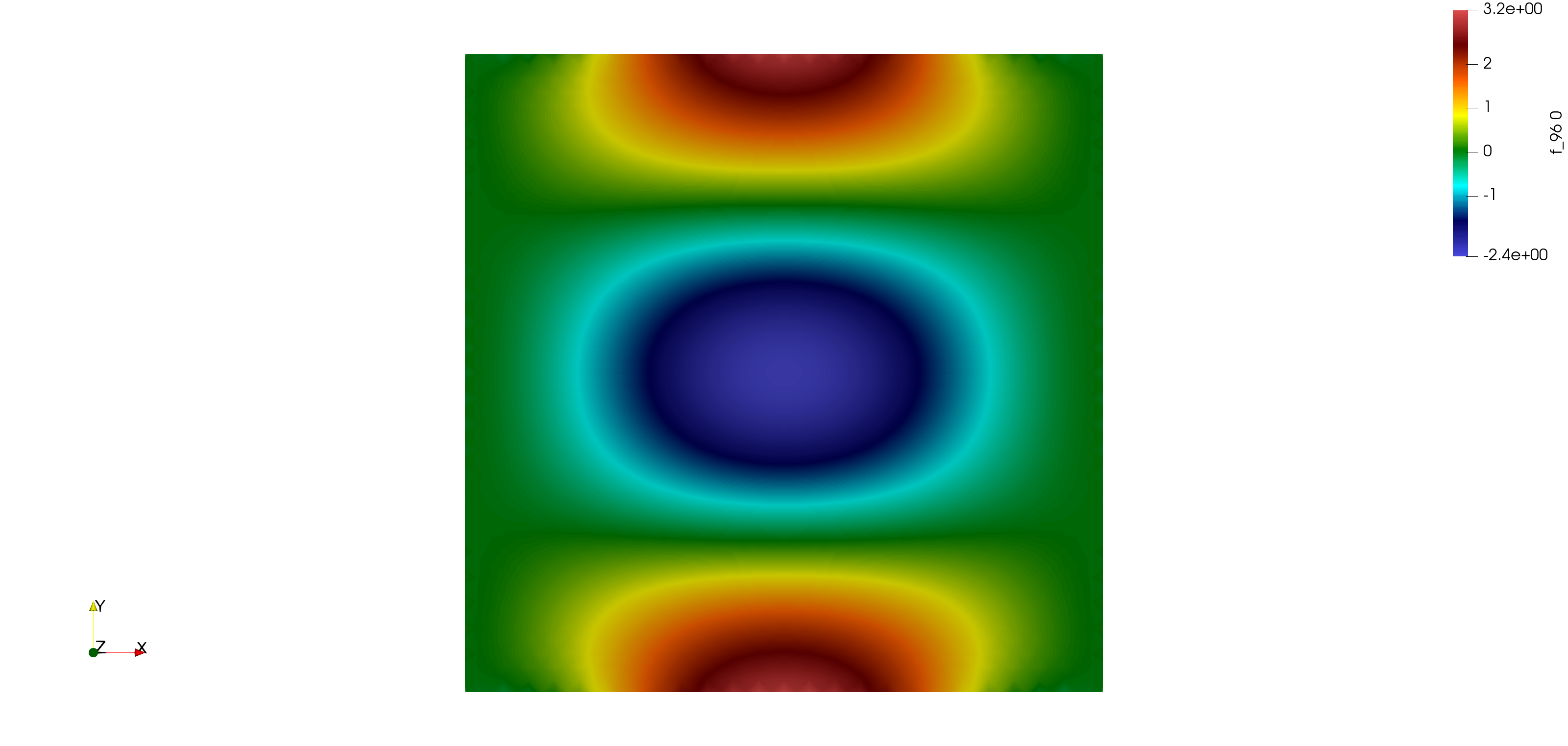}
\end{minipage}
\begin{minipage}{0.32\linewidth}
\centering\includegraphics[scale=0.065, trim= 32cm 3cm 32cm 3cm, clip]{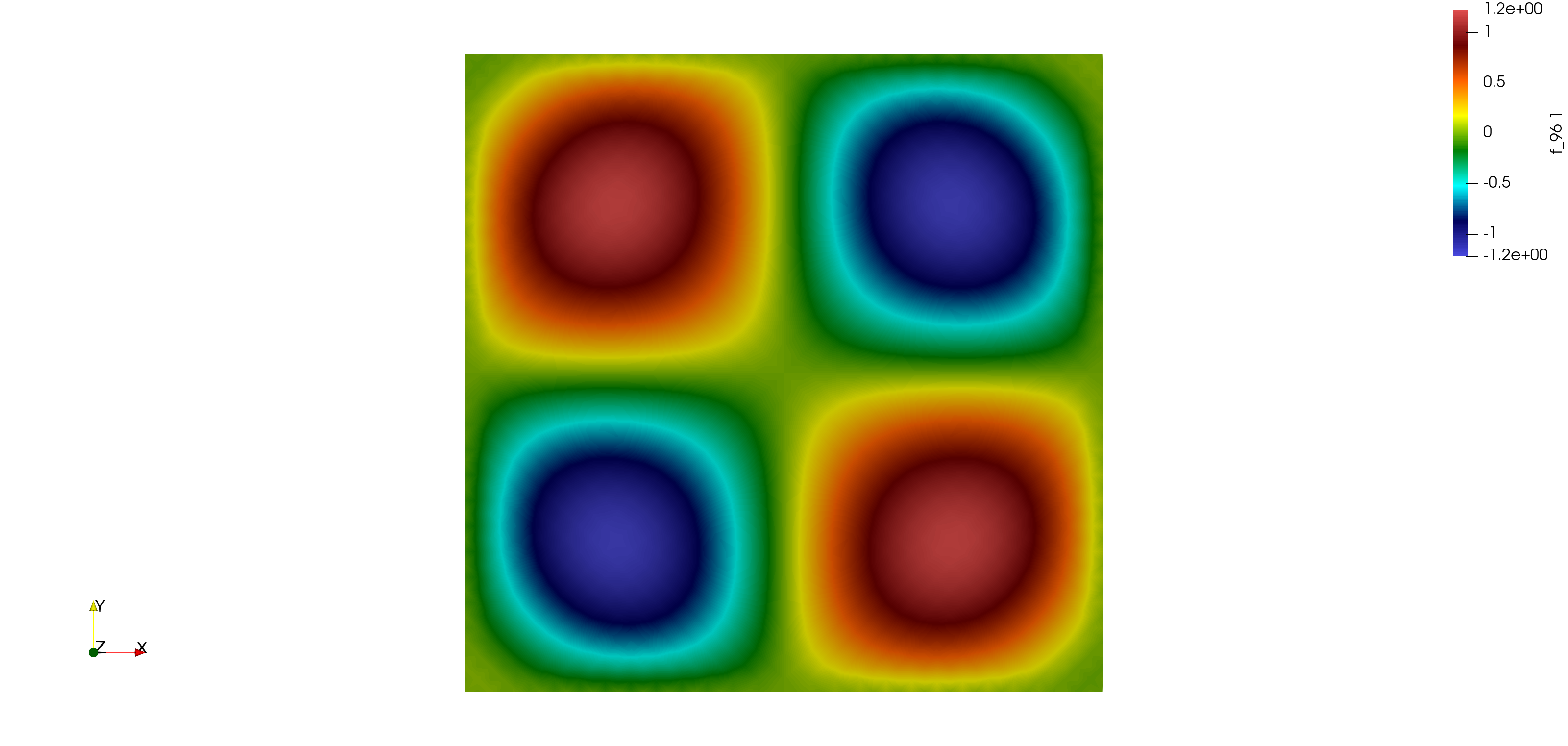}
\end{minipage}
\begin{minipage}{0.32\linewidth}
\centering\includegraphics[scale=0.065, trim= 32cm 3cm 32cm 3cm, clip]{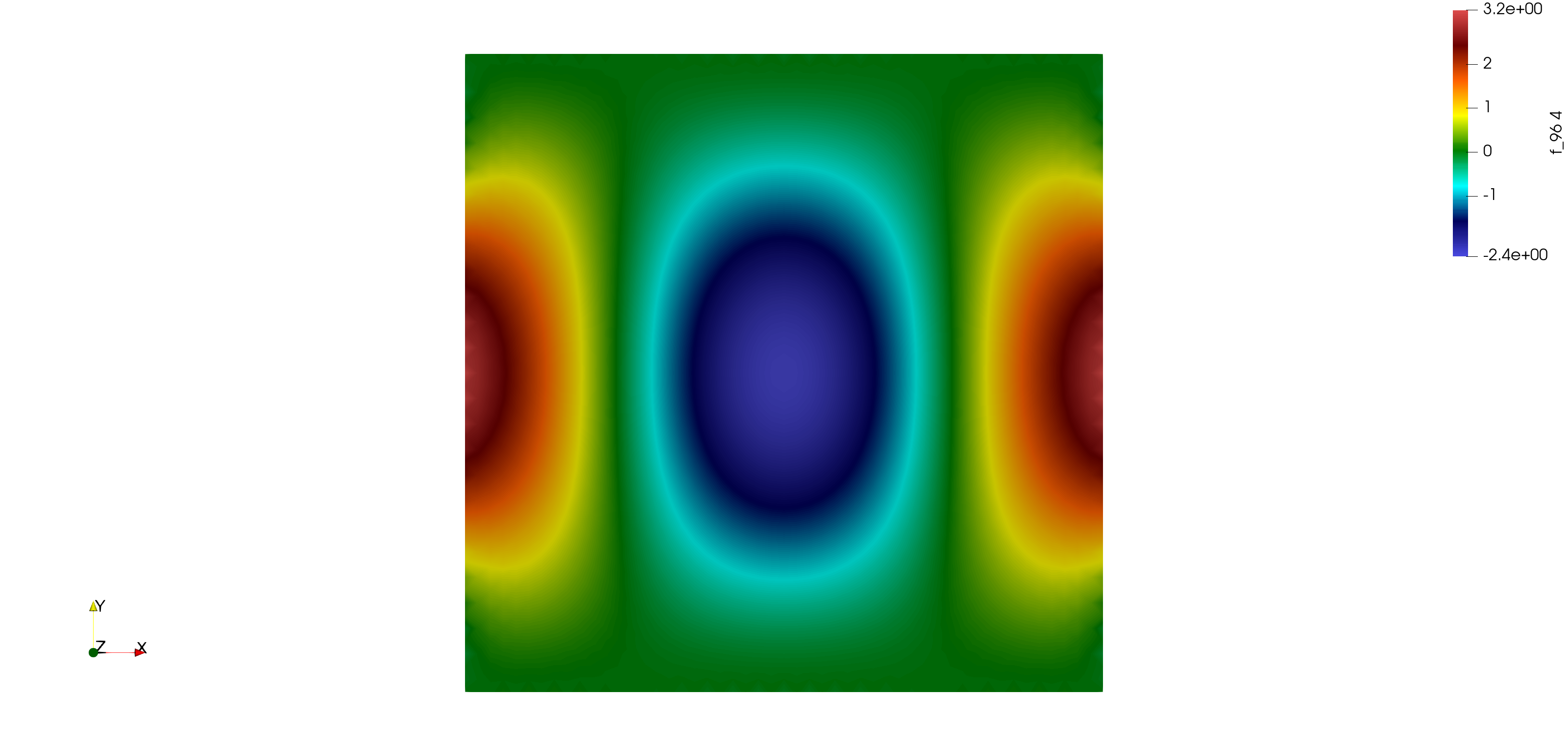}
\end{minipage}
\caption{Test 1. Postprocessed vorticity components  $\underline{\curl}(\bu_h)_{11}$ (left), $\underline{\curl}(\bu_h)_{12}$ (center) and $\underline{\curl}(\bu_h)_{22}$ (right) corresponding to the first eigenvalue in the square domain.}
\label{figura:modos-cuadrado-vorticidad}
\end{figure}
\subsection{Test 2: Non-polygonal domain}
In this experiment we take a curved domain and approximate it by polygonal meshes. This leads to a variational crime, which will affect the order of convergence. The domain for this experiment is the unit circle $\O:=\{(x,y)\in\mathbb{R}^2\,:\, x^2+y^2\leq 1\}$, and in Figure \ref{meshes_circle} we show examples of the meshes we consider to approximate this domain. We recall that $N$ represents the mesh resolution such that the number of elements is asymptotically $6N^2$.

First we present in Table \ref{tabla:circle} the results from approximating the eigenproblem using the $\mathrm{P}_k^{2}\text{-}\mathbb{NED}^{(1)}_{k}$ scheme. It is observed that, for the case $k=0$ we have the desired convergence. However, for $k>0$ we observe that the convergence remains at $\mathcal{O}(h^2)\simeq \mathcal{O}(\text{dof}^{-1})$, showing explicitly the effect of variational crime. This is also reflected in Table \ref{tabla:circle-NED2}, where despite applying the $\mathrm{P}_k^{2}\text{-}\mathbb{NED}^{(2)}_{k}$ scheme, which contains more dofs, the convergence does not improve for $k>0$. However, the results obtained are are in good agreement with those predicted by theory. The results from using the $\mathrm{P}_k^{2}\text{-}\mathbb{NED}^{(2)}_{k+1}$  scheme are described in Table \ref{tabla:circle-NED2}, where similar rates of convergence are observed.  We further explore the results by presenting Figure \ref{figura:modos-circulo} and \ref{figura:modos-circulo-vorticidad}. In Figure \ref{figura:modos-circulo} we can observe the velocity and postprocessed pressure for the fourth normal mode approximation, while the vorticity components calculated by postprocessing are observed in Figure \ref{figura:modos-circulo-vorticidad}.
\begin{figure}
\begin{center}
\begin{minipage}{0.3\linewidth}
\centering\includegraphics[scale=0.11]{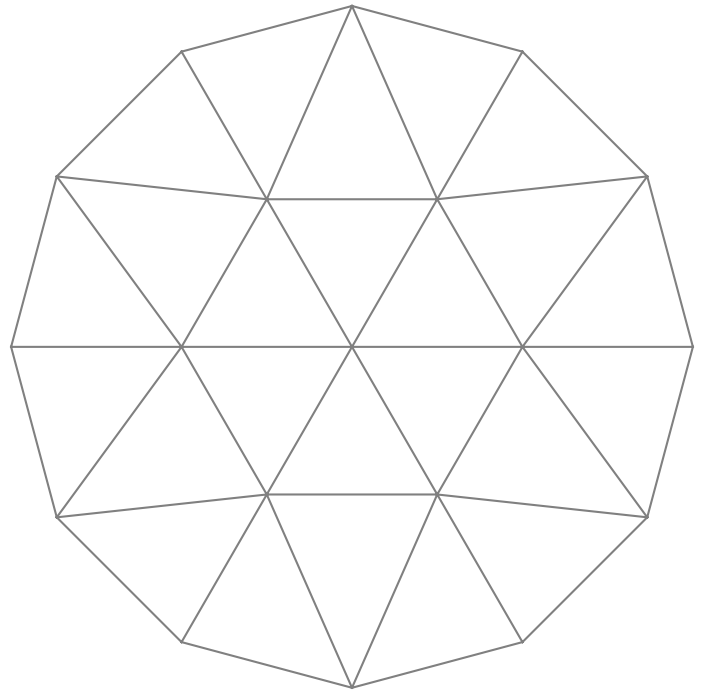}
\end{minipage}
\begin{minipage}{0.3\linewidth}
\centering\includegraphics[scale=0.11]{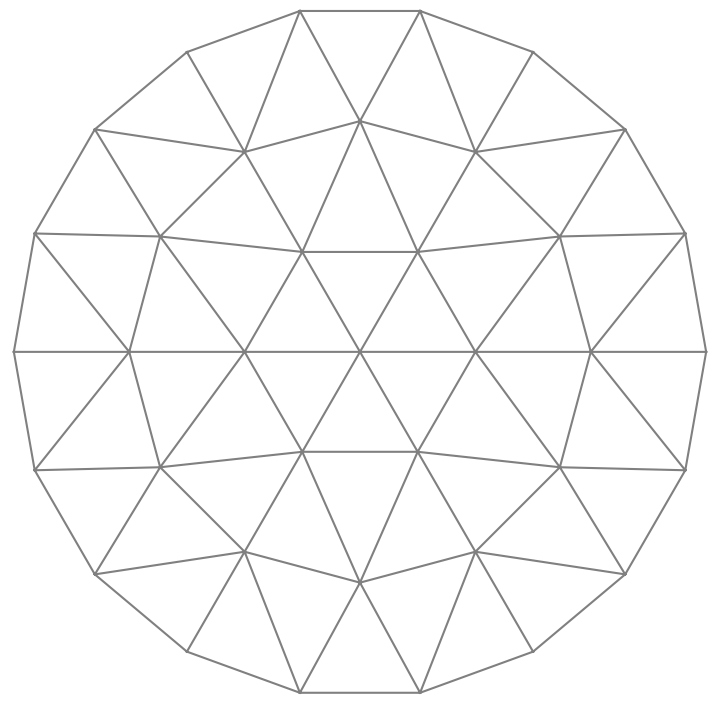}
\end{minipage}
\begin{minipage}{0.3\linewidth}
\centering\includegraphics[scale=0.11]{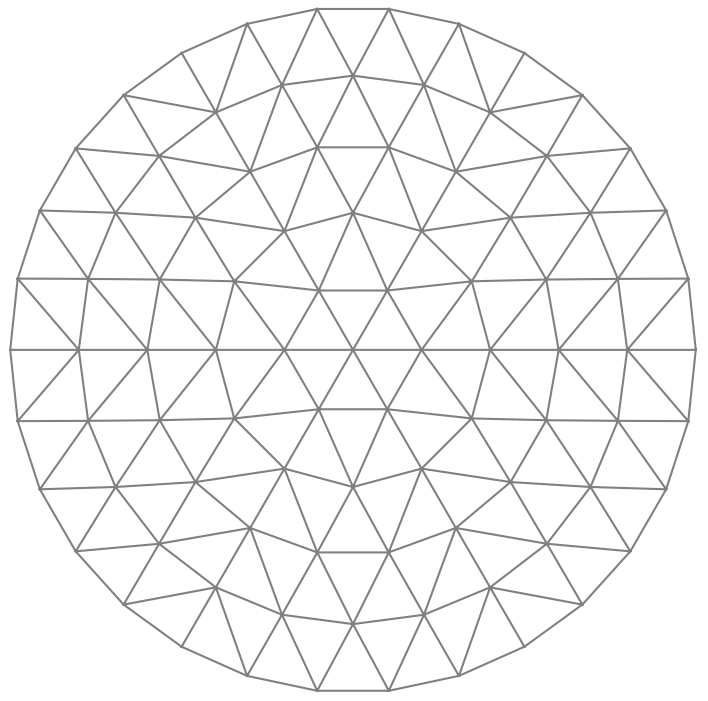}
\end{minipage}
\caption{Test 2. Example of meshes used in the circular domain.}
\label{meshes_circle}
\end{center}
\end{figure} 
\begin{table}
{\footnotesize
\begin{center}
\caption{Test 2. Lowest computed eigenvalues for polynomial degrees $k=0, 1, 2$ using the  $\mathrm{P}_k^{2}\text{-}\mathbb{NED}^{(1)}_k$   scheme.}
\begin{tabular}{c |c c c c |c| c| c }
\toprule
$k $        & $N=20$             &  $N=30$         &   $N=40$         & $N=50$ & Order & $\lambda_{extr}$ &\cite{ MR4077220} \\ 
\midrule
&  14.70187 & 14.69082 & 14.68695&  14.68515   &2.01&  14.68198  &14.68345  \\
&  26.39800 & 26.38501 & 26.38046&  26.37835  &2.01&  26.37463  &26.37840   \\
\multirow{2}{0.15cm}{0}
&  26.39800 & 26.38501 & 26.38046&  26.37835   &2.01&  26.37463  &26.37862   \\
&  40.73553 & 40.71969 & 40.71396&  40.71128  &1.93 & 40.70625   &40.71434    \\
&  40.73553 & 40.71969 & 40.71396&  40.71128  &1.93 & 40.70625   &40.71606   \\

\hline

&  14.68868&  14.68495&  14.68364 & 14.68304&   2.01& 14.68197 &14.68345    \\
&  26.38672&  26.37998&  26.37763 & 26.37654&   2.02& 26.37464 &26.37840     \\
\multirow{2}{0.15cm}{1} 
&  26.38672&  26.37998&  26.37763 & 26.37654&   2.02& 26.37464 &26.37862	\\
&  40.72530&  40.71477&  40.71113 & 40.70944&   2.03& 40.70651 &40.71434	 \\
&  40.72530&  40.71477&  40.71113 & 40.70944&   2.03& 40.70651 &40.71606	 \\

\hline

&  14.68871&  14.68496&  14.68365 & 14.68304 &  2.01 & 14.68361  & 14.68345  \\
&  26.38673&  26.37999&  26.37763 & 26.37654 &  2.01 & 26.37680  &26.37840 	\\
\multirow{2}{0.15cm}{2}   
&  26.38673&  26.37999&  26.37763 & 26.37654 &  2.01 & 26.37680 &26.37862	\\
&  40.72516&  40.71476&  40.71112 & 40.70944 &  2.01 & 40.70647 &40.71434	 \\
&  40.72516&  40.71476&  40.71112 & 40.70944 &  2.01 & 40.70647 &40.71606	 \\

\bottomrule             
\end{tabular}\label{tabla:circle}
\end{center}}

\end{table}
\begin{table}
{\footnotesize
\begin{center}
\caption{Test 2. Lowest computed eigenvalues for polynomial degrees $k$$\,=\,$$0$,$\,1$,$\,2$ using the $\mathrm{P}_k^{2}\text{-}\mathbb{NED}^{(2)}_{k+1}$   scheme. }
\begin{tabular}{c|c c c c |c| c| c }
\toprule
$k$        & $N=20$             &  $N=30$         &   $N=40$         & $N=50$ & Order & $\lambda_{extr}$ &\cite{ MR4077220} \\ 
\midrule
&  14.82469 & 14.71768 & 14.69784&  14.69090   &2.01&  14.68199  &14.68345  \\
&  26.77392 & 26.47427 & 26.41889&  26.39951  &2.02&  26.37450  &26.37840   \\
\multirow{2}{0.15cm}{0}
&  26.77392 & 26.47427 & 26.41889&  26.39951   &2.02&  26.37450  &26.37862   \\
&  41.56881 & 40.92423 & 40.80343&  40.76105  &2.01 & 40.70545   &40.71434    \\
&  41.56881 & 40.92423 & 40.80343&  40.76105  &2.01 & 40.70545   &40.71606   \\

\hline

&  14.68873&  14.68496&  14.68365 & 14.68304&   2.02& 14.68198 &14.68345    \\
&  26.38682&  26.38000&  26.37764 & 26.37655&   2.03& 26.37464 &26.37840     \\
\multirow{2}{0.15cm}{1} 
&  26.38682&  26.38000&  26.37764 & 26.37655&   2.03& 26.37464 &26.37862	\\
&  40.72553&  40.71483&  40.71115 & 40.70945&   2.05& 40.70654 &40.71434	 \\
&  40.72553&  40.71483&  40.71115 & 40.70945&   2.05& 40.70654 &40.71606	 \\

\hline

&  14.68874&  14.68497&  14.68365 & 14.68304 &  2.02 & 14.68198  & 14.68345  \\
&  26.38678&  26.38000&  26.37764 & 26.37655 &  2.02 & 26.37464  &26.37840 	\\
\multirow{2}{0.15cm}{2}   
&  26.38678&  26.38000&  26.37764 & 26.37655 &  2.02 & 26.37464 &26.37862	\\
&  40.72524&  40.71478&  40.71113 & 40.70945 &  2.01 & 40.70646 &40.71434	 \\
&  40.72524&  40.71478&  40.71113 & 40.70945 &  2.01 & 40.70646 &40.71606	 \\

\bottomrule             
\end{tabular}\label{tabla:circle-NED2}
\end{center}}

\end{table}
\begin{figure}
\centering\begin{minipage}{0.45\linewidth}
\centering\includegraphics[scale=0.065, trim= 32cm 3cm 32cm 3cm, clip]{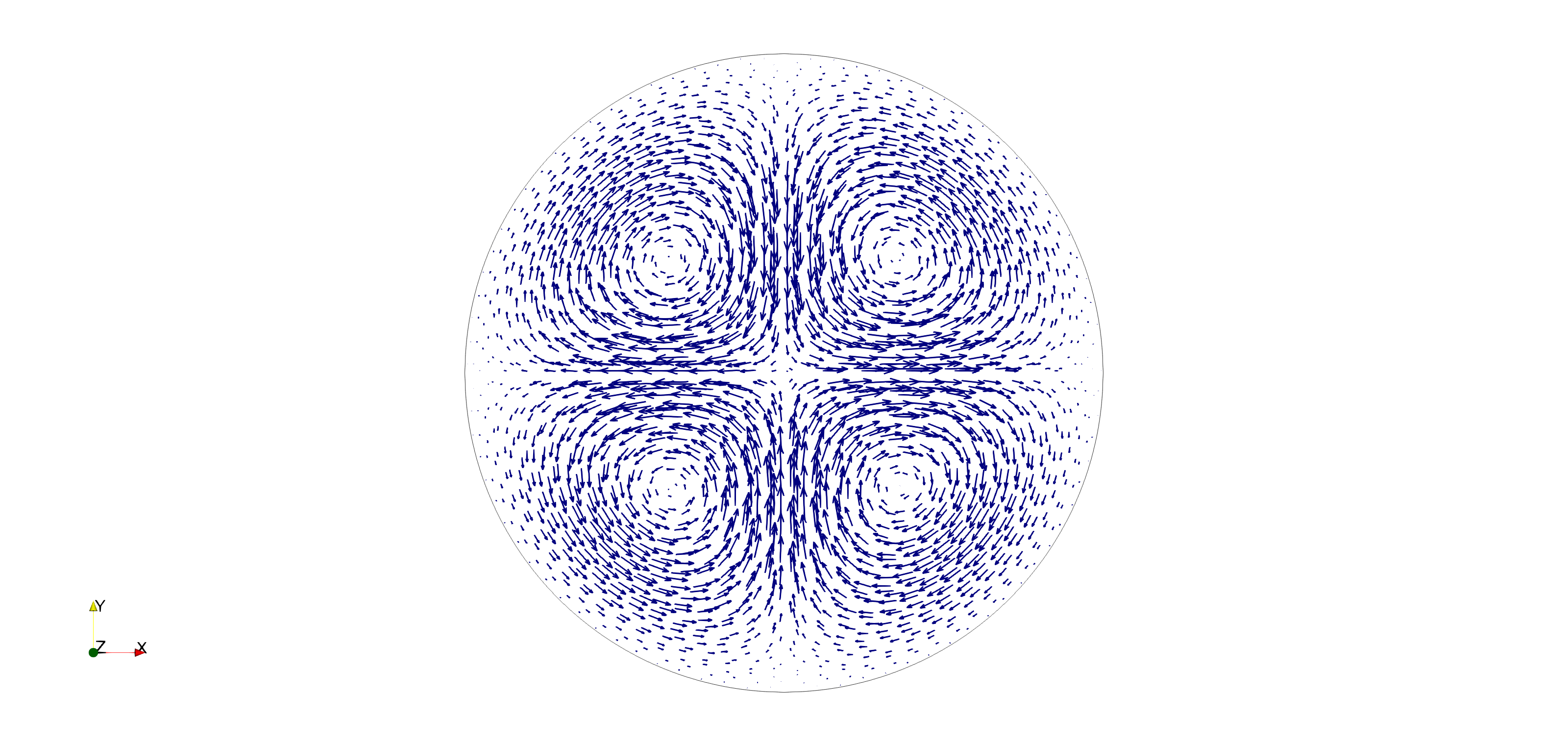}
\end{minipage}
\begin{minipage}{0.45\linewidth}
\centering\includegraphics[scale=0.065, trim= 32cm 3cm 32cm 3cm, clip]{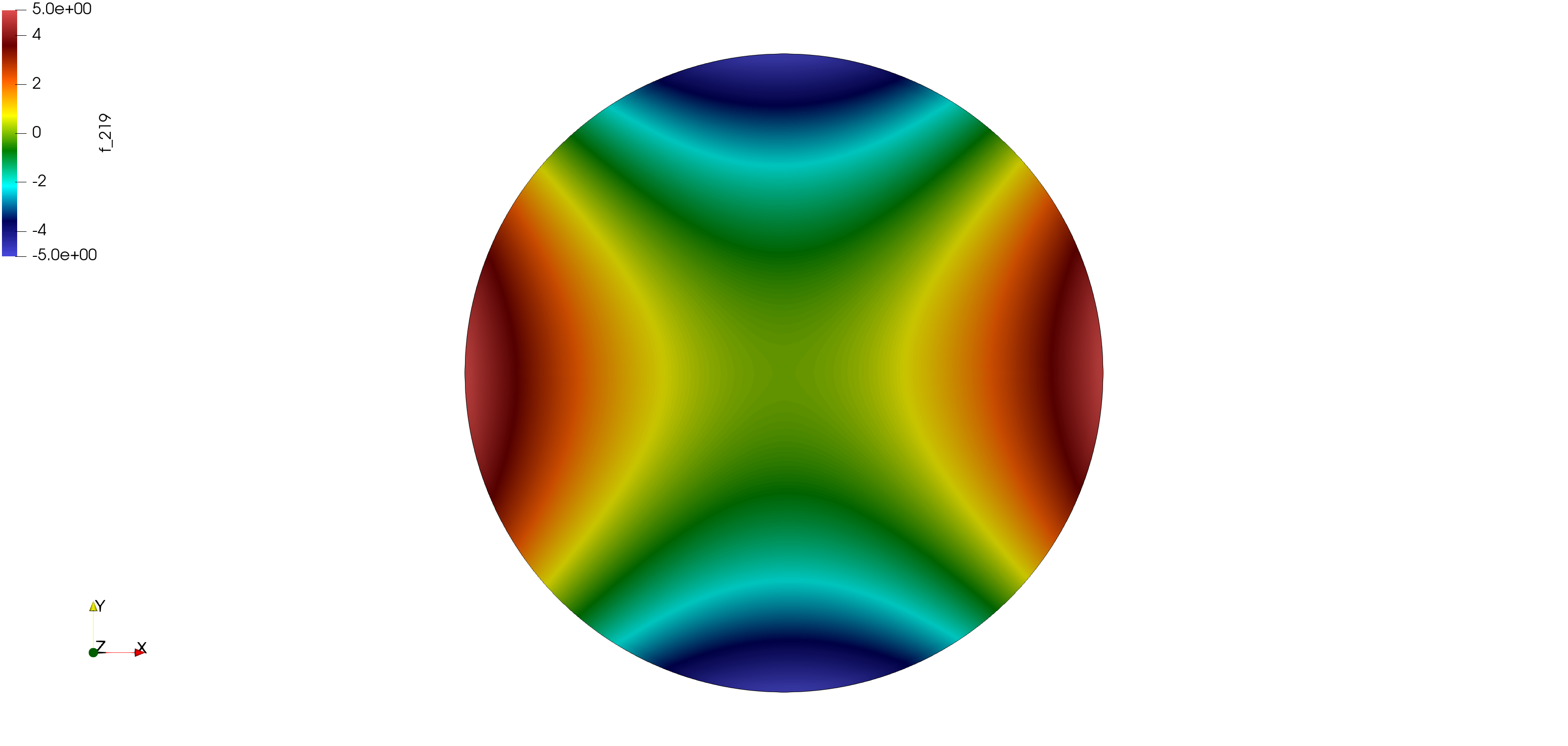}
\end{minipage}
\caption{Test 2. Approximate velocity field $\bu_h$ (left) and postprocessed pressure $p_h$ (right), corresponding to the fourth eigenvalue in the unit circular domain.}
\label{figura:modos-circulo}
\end{figure}
\begin{figure}
\centering\begin{minipage}{0.32\linewidth}
\centering\includegraphics[scale=0.065, trim= 32cm 3cm 32cm 3cm, clip]{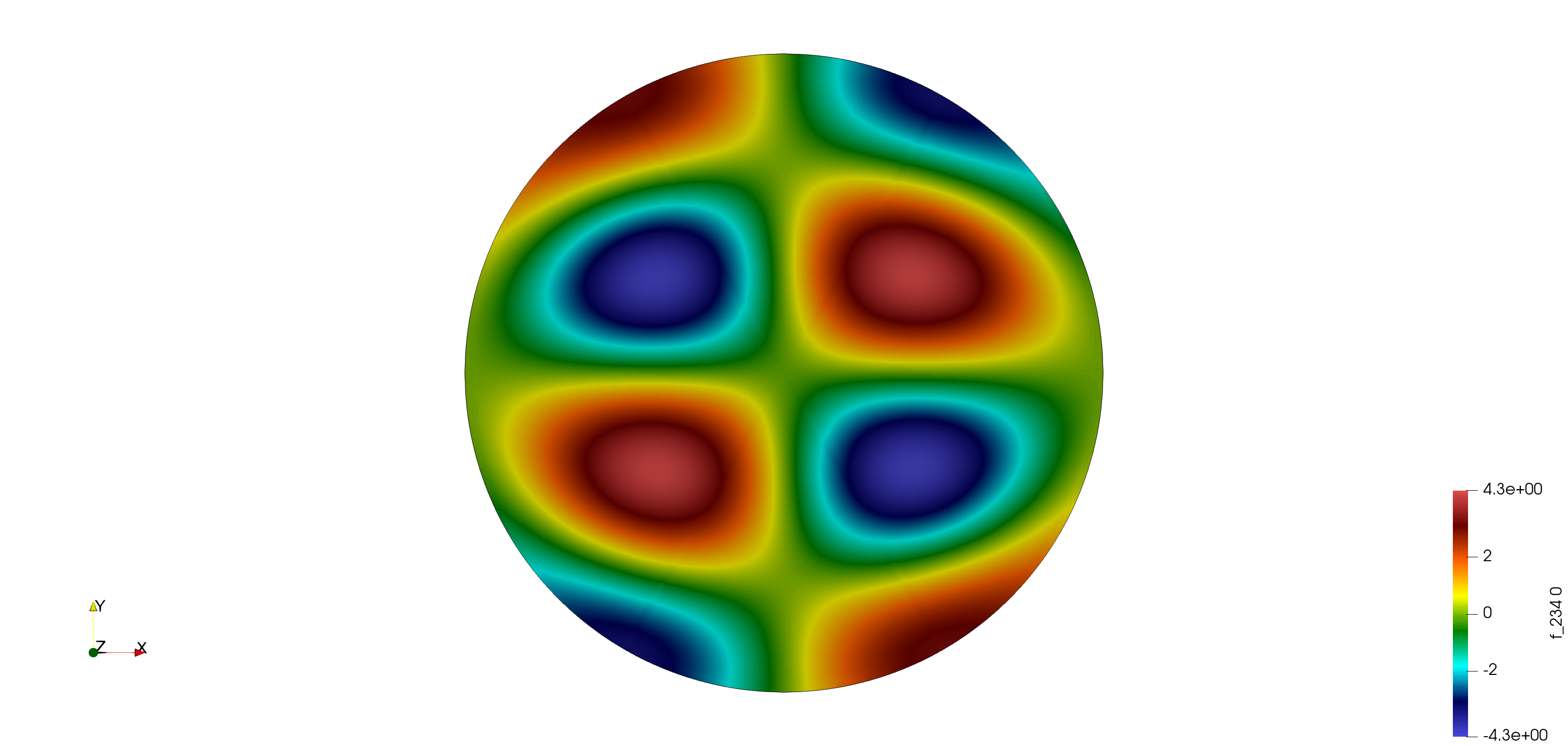}
\end{minipage}
\begin{minipage}{0.32\linewidth}
\centering\includegraphics[scale=0.065, trim= 32cm 3cm 32cm 3cm, clip]{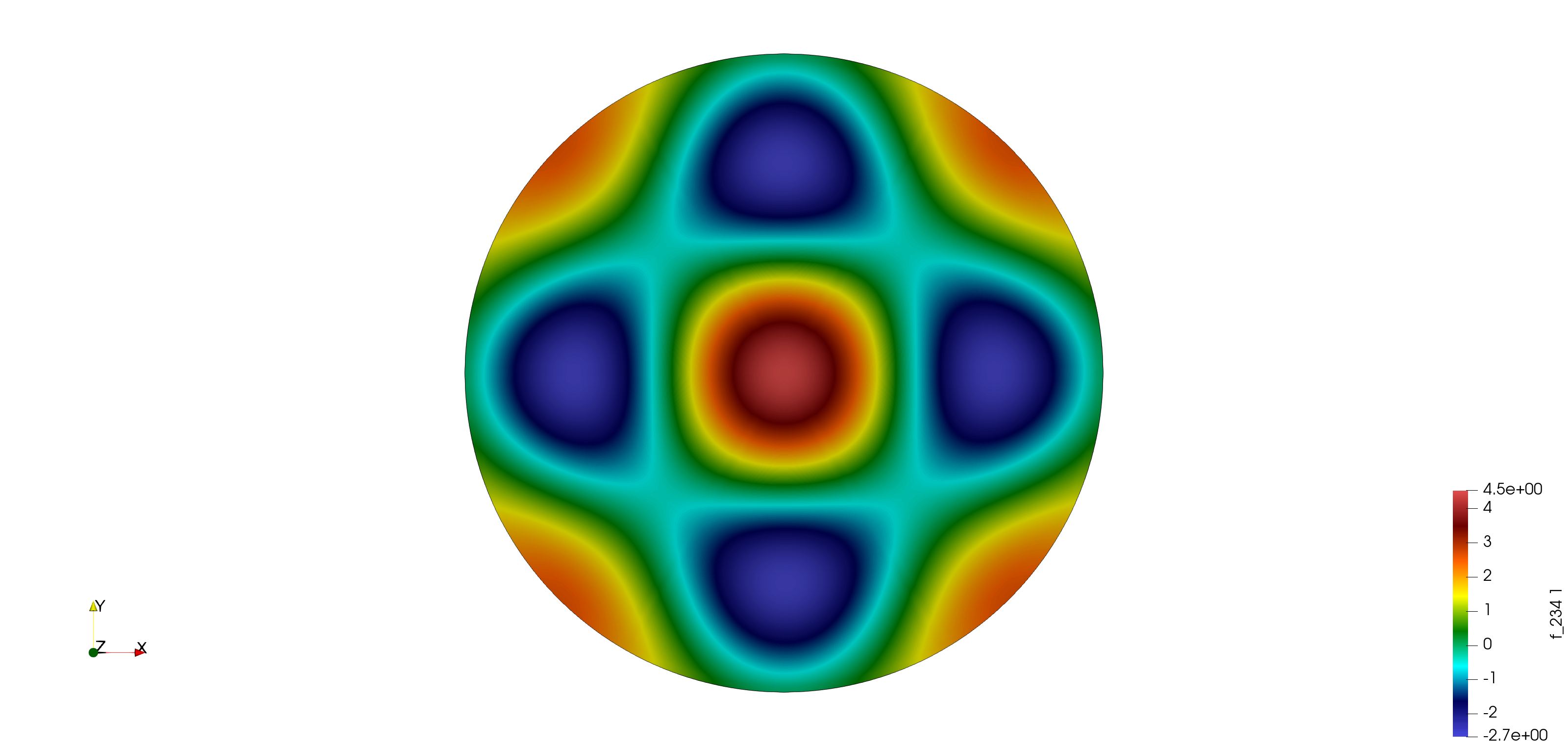}
\end{minipage}
\begin{minipage}{0.32\linewidth}
\centering\includegraphics[scale=0.065, trim= 32cm 3cm 32cm 3cm, clip]{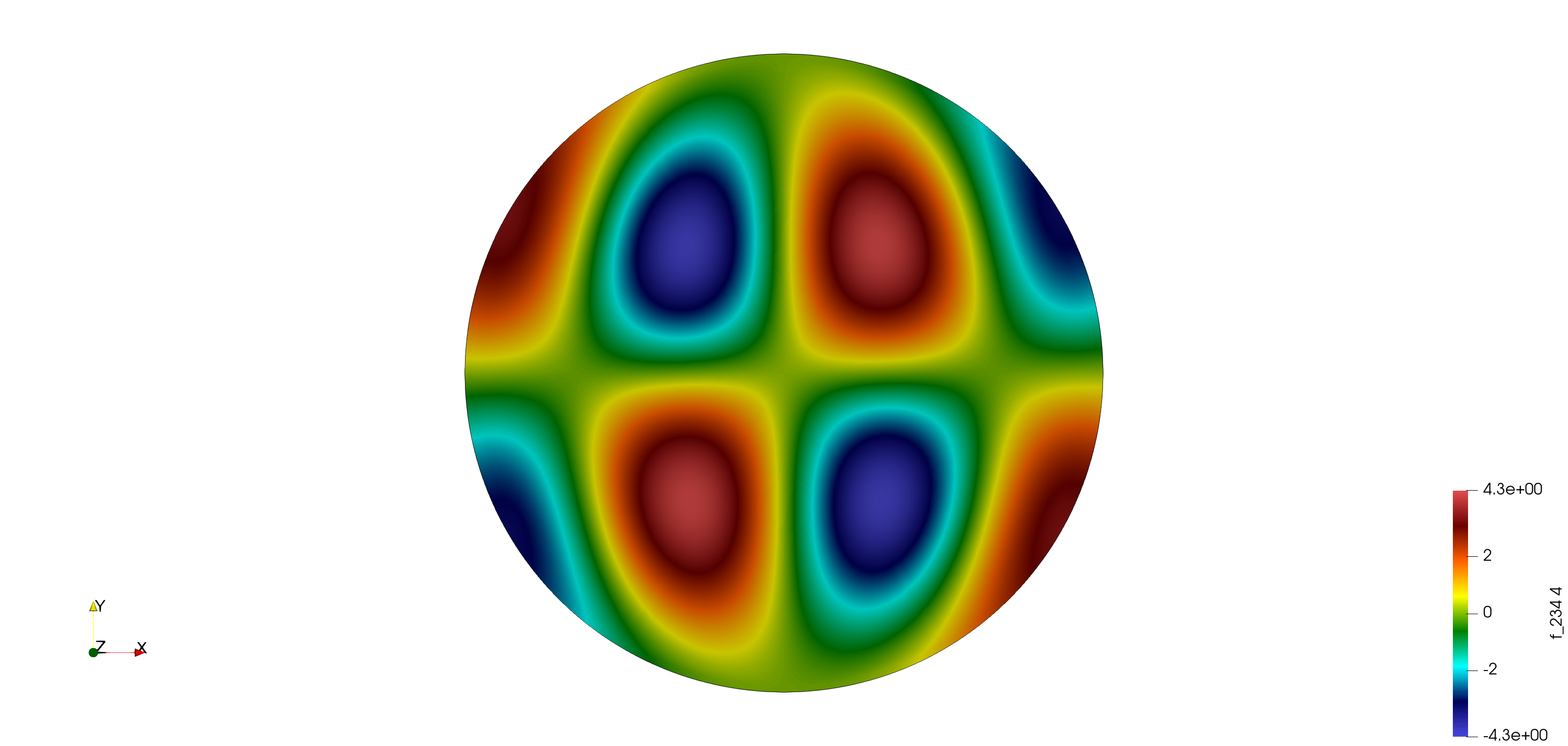}
\end{minipage}
\caption{Test 1. Postprocessed vorticity components  $\underline{\curl}(\bu_h)_{11}$ (left), $\underline{\curl}(\bu_h)_{12}$ (center) and $\underline{\curl}(\bu_h)_{22}$ (right) corresponding to the fourth eigenvalue in the circular domain.}
\label{figura:modos-circulo-vorticidad}
\end{figure}
\subsection{Test 3: Mixed boundary conditions}

The aim of the following test is to explore the performance of  the proposed  method in a more general eigenvalue problem. To do this task, we consider the the boundary $\partial\O$ of our domain is separated into two section by $\partial\O:=\Gamma_D\cup\Gamma_N$, where $\Gamma_D$ and $\Gamma_N$ represents the part of the boundary where we impose Dirichlet and Neumann boundary conditions, respectively. We assume that both $\Gamma_D$ and $\Gamma_N$ have positive measure.  With these definitions at hand,  the problem to consider is the following: Find $\lambda\in \mathbb{R}$, the stress $\bsig$, the velocity $\bu$ and the pressure $p$ such that 
$$
\left\{
\begin{array}{rcll}
\bsig-\mu\underline{\curl}(\bu)-p\mathbb{J}& = & \boldsymbol{0}&  \text{ in } \quad \Omega, \\
\curl(\bsig)& = & -\lambda\bu & \text{ in } \quad \Omega, \\
p&=&\displaystyle-\frac{1}{2}(\bsig:\mathbb{J})&\text{ in } \quad \Omega, \\
\bu & = & \boldsymbol{0} & \,\,\text{on} \quad \Gamma_D,\\
\bsig\boldsymbol{s}&=&\boldsymbol{0}&\,\,\text{on}\quad\Gamma_N,
\end{array}
\right.
$$
where $\boldsymbol{s}$ corresponds to the tangential component of the unitary vector on $\Gamma_N$. In what follows we will consider $\O:=(0,1)^2$ as computational domain. For this square, we assume that the bottom is fixed and the rest 
of its sides are free of stress. 	
\begin{table}
{\footnotesize
\begin{center}
\caption{Test 3. Lowest computed eigenvalues using the $\mathrm{P}_0^{2}\text{-}\mathbb{NED}^{(1)}_0,$ and $\mathrm{P}_0^{2}\text{-}\mathbb{NED}^{(2)}_{1},$  schemes in the square domain with mixed boundary conditions. }
\begin{tabular}{c |c c c c |c| c| c}
	\toprule
	scheme        & $N=20$             &  $N=30$         &   $N=40$         & $N=50$ & Order & $\lambda_{extr}$&\cite{MR3335223}\\ 
	\midrule
	& 2.46736  & 2.46738  &2.46739   &2.46739  &2.14 &2.46740  &2.4674\\
	& 6.24652  & 6.26420  &6.27066   &6.27372  &1.91 &6.27952  &6.2799\\
	\multirow{1}{0.11\linewidth}{$\mathrm{P}_0^{2}\text{-}\mathbb{NED}^{(1)}_0$}
	& 15.16639 & 15.18837 &15.19693  &15.20112 &1.74 &15.21010 &15.2090\\
	& 22.20329 & 22.20521 &22.20584  &22.20612 &2.18 &22.20657 &22.2065\\
	& 26.84469 & 26.91450 &26.92896  &26.93579 &1.92 &26.94869 &26.9479\\
	
	\hline
	
	&2.46824   &2.46777   &2.46761   &2.46753  &2.03 &2.46740  &2.4674\\
	&6.28434   &6.28163   &6.28065   &6.28019  &1.95 &6.27935  &6.2799\\
	\multirow{1}{0.11\linewidth}{$\mathrm{P}_0^{2}\text{-}\mathbb{NED}^{(2)}_1$} 
	&15.23974  &15.22297  &15.21701  &15.21423 &1.98 &15.20917 &15.2090\\
	&22.27513  &22.23705  &22.22373  &22.21756 &2.03 &22.20678 &22.2065\\
	&27.04797  &26.99278  &26.97339  &26.96439 &2.01 &26.94835 &26.9479\\
	\bottomrule             
\end{tabular}\label{tabla:square-suelto}
\end{center}}

\end{table}
We observe from Table \ref{tabla:square-suelto} that the computed eigenvalues are accurately 
recovered with our both numerical schemes, with a clearly quadratic order of convergence. Moreover, our extrapolated values are close to those presented on \cite{MR3335223} for an alternative formulation of the Stokes spectral problem. On the other hand, we present in Figures \ref{figura:modos-cuadrado-suelto} and \ref{figura:modos-cuadrado-suelto-vorticidad} plots of the velocity field, pressure and vorticity components associated to the third eigenfunction of the problem with mixed boundary.
\begin{figure}
\centering\begin{minipage}{0.45\linewidth}
\centering\includegraphics[scale=0.059, trim= 32cm 0cm 30cm 3cm, clip]{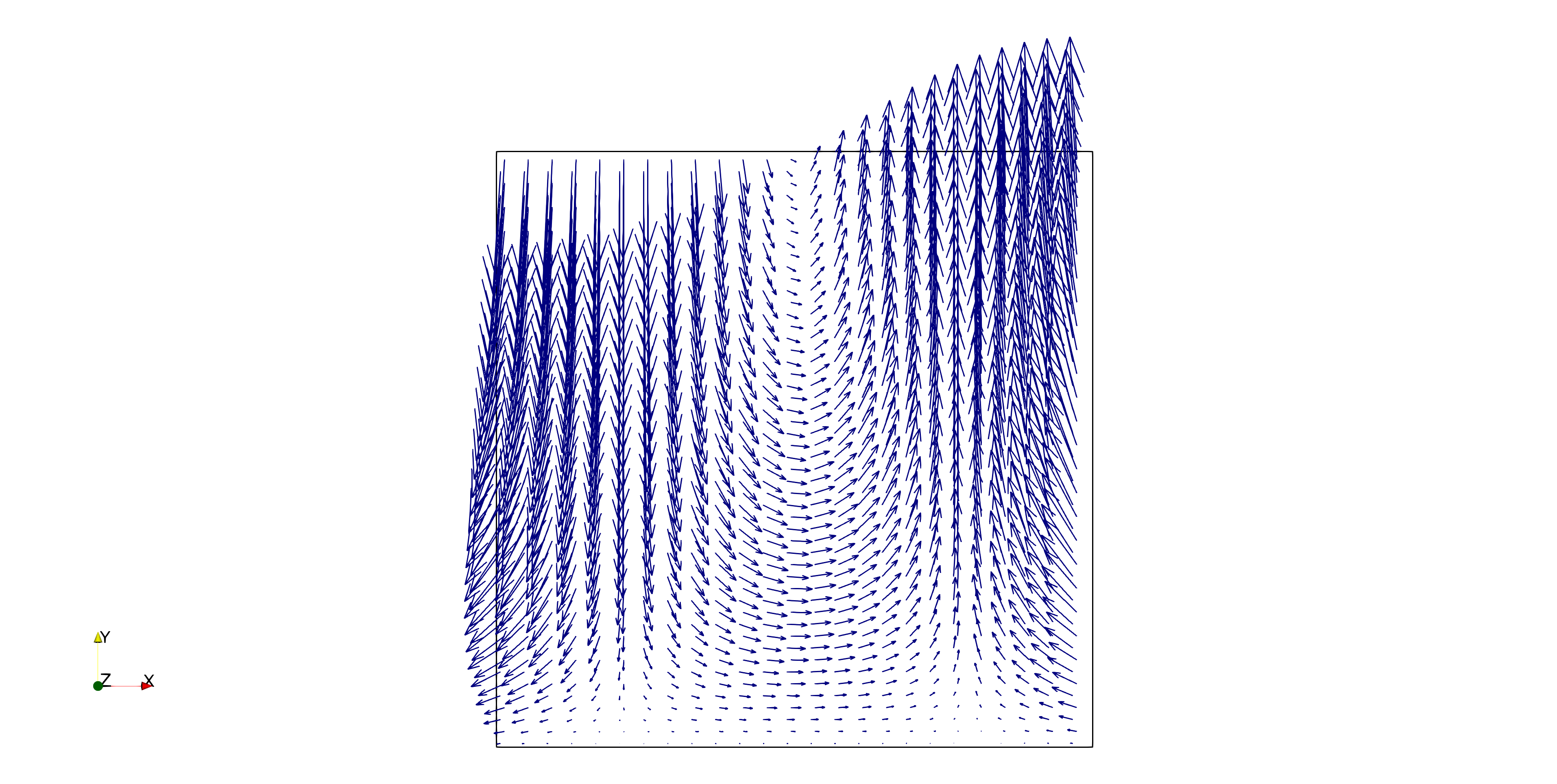}
\end{minipage}
\begin{minipage}{0.45\linewidth}
\centering\includegraphics[scale=0.065, trim= 32cm 3cm 32cm 0cm, clip]{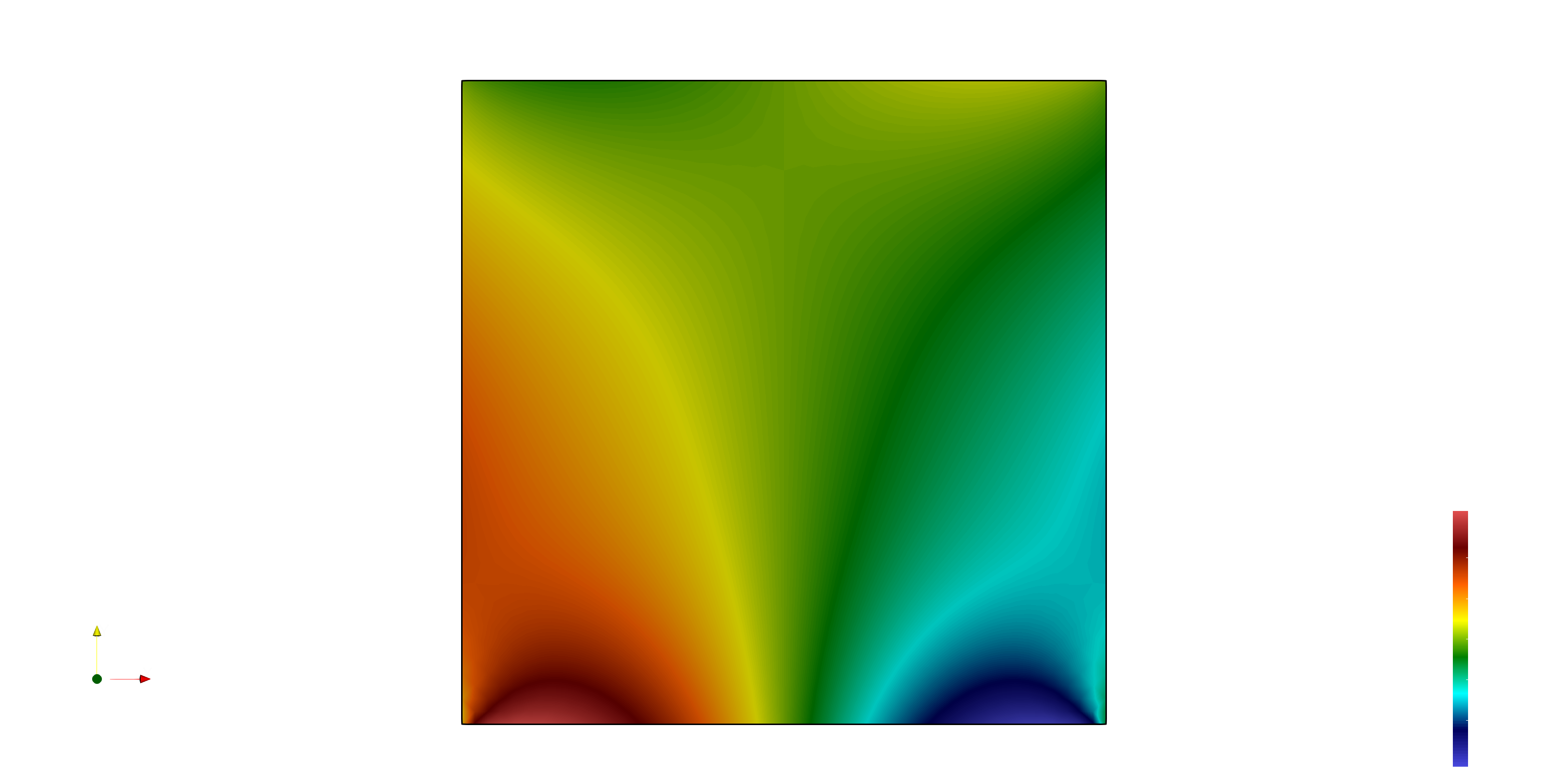}
\end{minipage}
\caption{Test 3. Approximate velocity field $\bu_h$ (left) and postprocessed pressure $p_h$ (right), corresponding to the third eigenvalue in the square domain with mixed boundary conditions.}
\label{figura:modos-cuadrado-suelto}
\end{figure}
\begin{figure}
\centering\begin{minipage}{0.32\linewidth}
\centering\includegraphics[scale=0.065, trim= 32cm 3cm 32cm 3cm, clip]{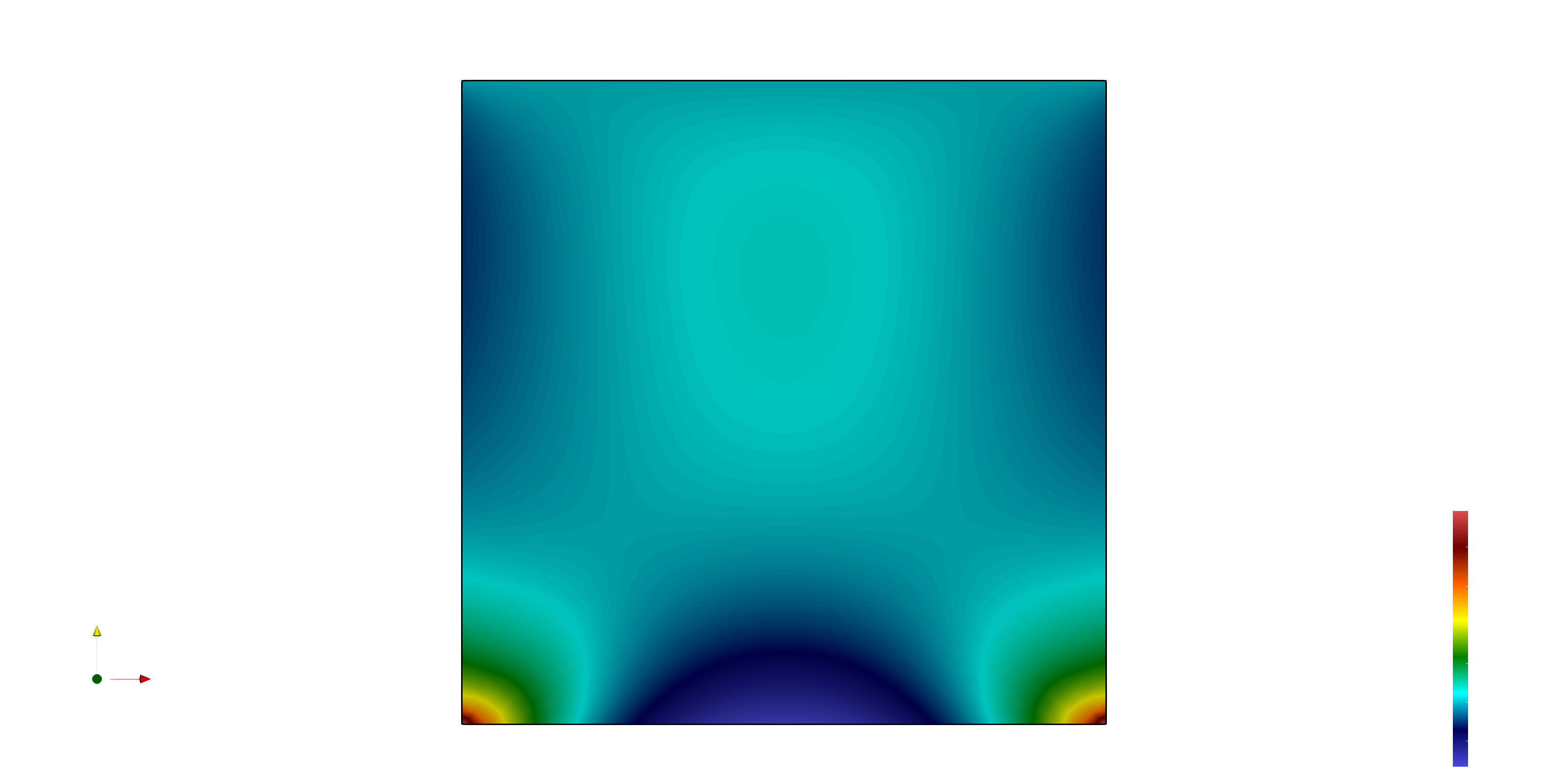}
\end{minipage}
\begin{minipage}{0.32\linewidth}
\centering\includegraphics[scale=0.065, trim= 32cm 3cm 32cm 3cm, clip]{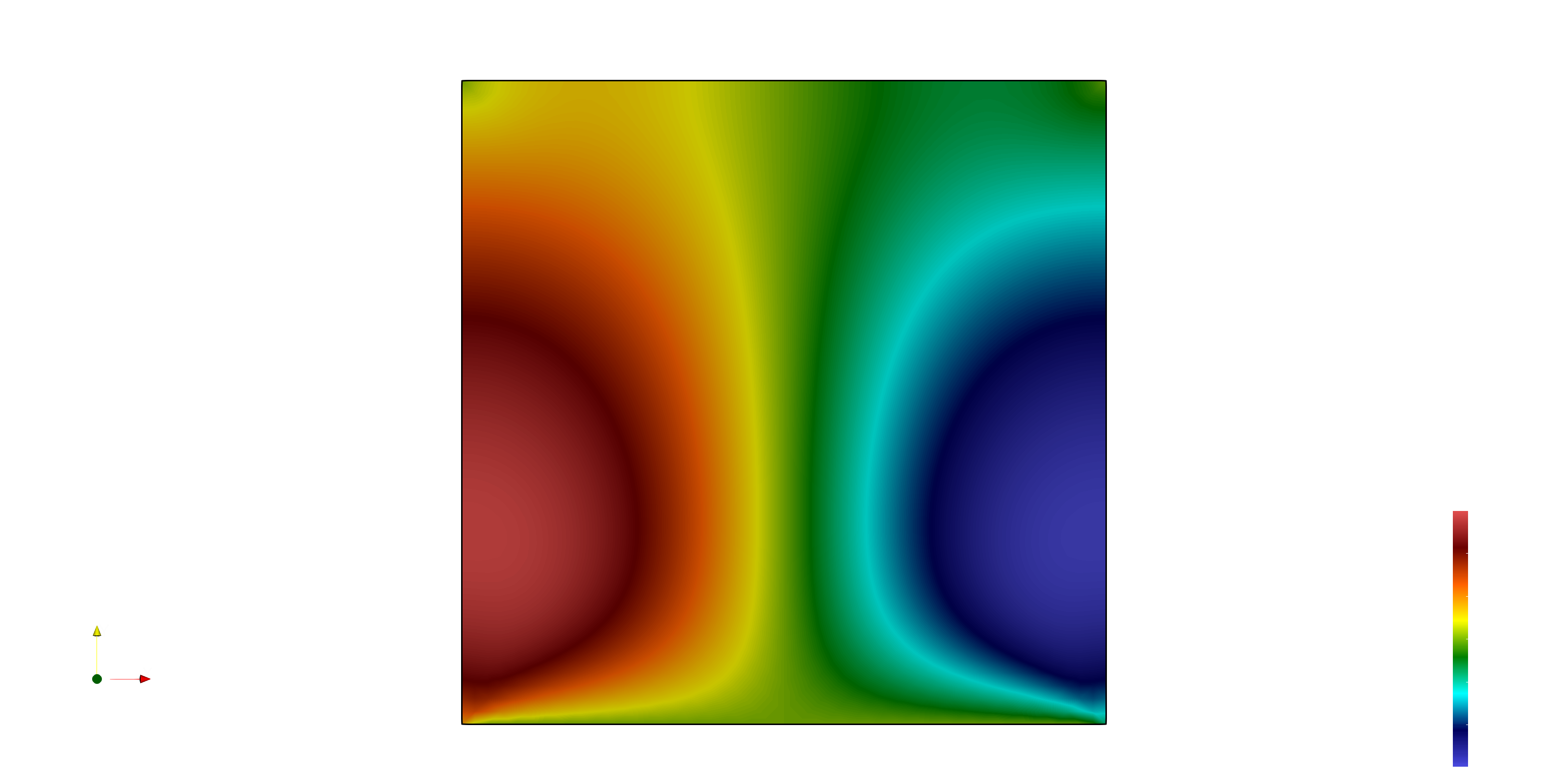}
\end{minipage}
\begin{minipage}{0.32\linewidth}
\centering\includegraphics[scale=0.065, trim= 32cm 3cm 32cm 3cm, clip]{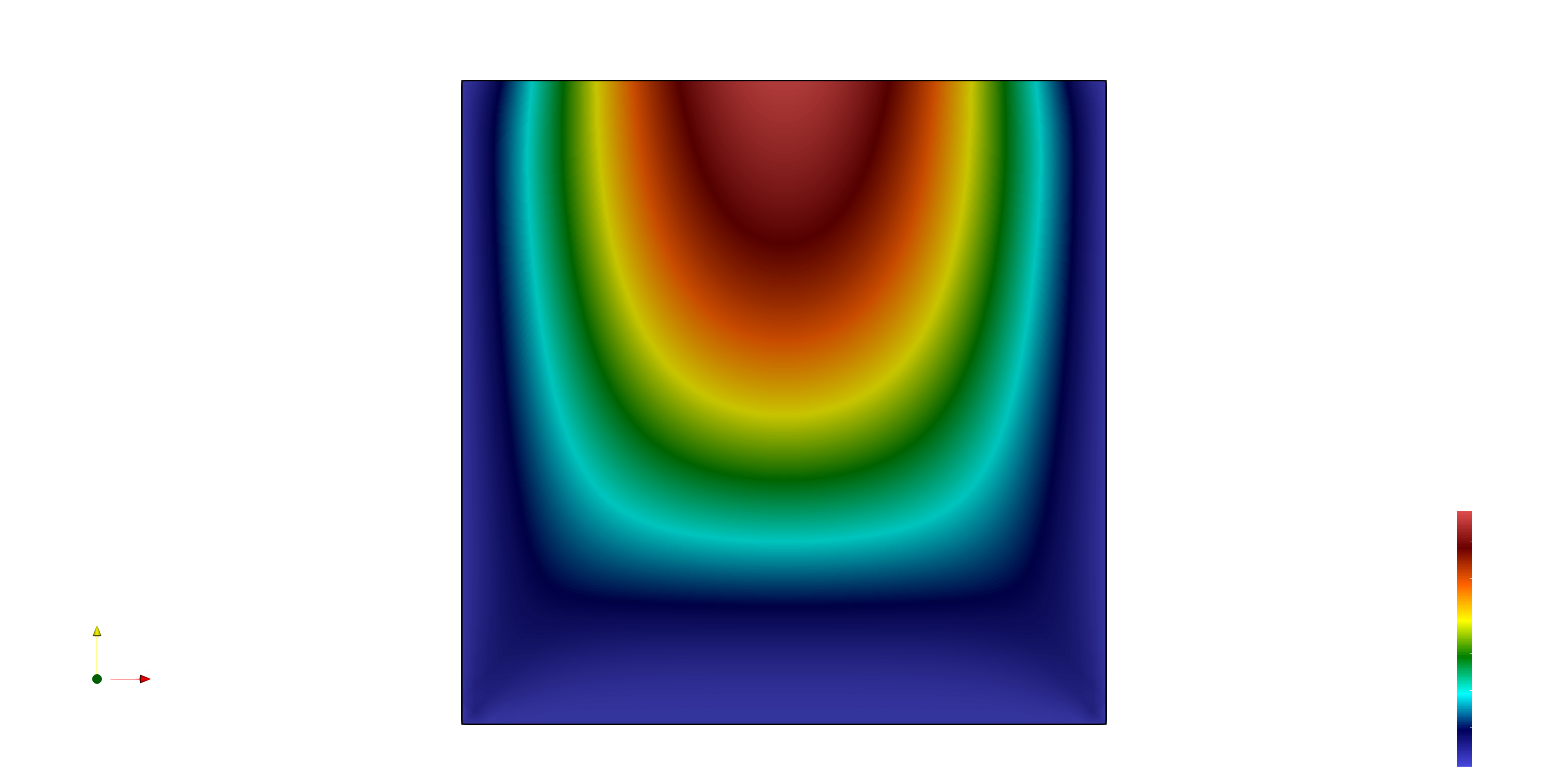}
\end{minipage}
\caption{Test 3. Postprocessed vorticity components  $\underline{\curl}(\bu_h)_{11}$ (left), $\underline{\curl}(\bu_h)_{12}$ (center) and $\underline{\curl}(\bu_h)_{22}$ (right) corresponding to the third eigenvalue in the square domain with mixed boundary conditions.}
\label{figura:modos-cuadrado-suelto-vorticidad}
\end{figure}
\subsection{Test 4. A posteriori test on a non-convex domain}
We end our numerical test section with results for the proposed a posteriori estimator. To do this task, we focus 
on simple eigenvalues of the spectrum of $\bT$. The computational domain for this test is 
$\Omega:=(-1,1)\times(-1,1) \backslash \big((-1,0)\times (-1,0)\big)$ and the only boundary condition is  $\bu=\boldsymbol{0}$. Since the reentrant angle of this domain 
leads to a lack of regularity for some eigenfunctions associated to $\bT$,  our goal is to recover the optimal order of convergence with the proposed estimator. The initial mesh for this test is depicted in Figure \ref{fig:lshapeinitialmesh}. 

It is well known that the regularity of the eigenfunctions in this geometry satisfy $2r\geq1.08$, so that under uniform refinements, suboptimal error rates are expected since $s\approx2\min\{r,k+1\}$  (see, for instance \cite{MR2473688,LRVSISC}).  The extrapolated value for this experiment have been obtained through sufficiently fine meshing and least squares fitting. We choose $\lambda_1=32.13183$ as an exact solution, which is in good agreement with the references above. 

The adaptively refinement procedure is based on the blue-green marking strategy, consisting of refining the triangle $T$ that satisfy
$$
\eta_{T}\geq 0.5\max_{T'\in\CT_h}\eta_{T'}.
$$

In Table \ref{tabla-l-shape-ned1} we observe the behavior of the estimator  $\eta$ defined in \eqref{eq:global_est_reduced} when the families $\mathbb{NED}_0^{(1)}$ and $\mathbb{NED}_1^{(2)}$ are used, with 15 iterations of the adaptive refinement. Note that $\vert \lambda_1-\lambda_{1h}\vert \approx C \mathrm{dof}^{-1.04}\approx Ch^{2.08}$. We also note that the additional degrees of freedom of the $\mathrm{P}_0^2-\mathbb{NED}_1^{(2)}$ scheme allow the method to be more efficient in the sense that, the elements marked for refinement are fewer than those when using $\mathrm{P}_0^2-\mathbb{NED}_0^{(1)}$. This is also observed in the intermediate meshes used in the adaptive algorithm shown in Figure \ref{fig:lshape-mesh-refinamientos}. The column corresponding to the effectivity $\vert\lambda_1-\lambda_{1h}\vert/\eta^2$  shows that our estimator remains properly bounded above and below, away from zero. 

A graphical description of these results can be seen in Figure \ref{fig:lshape-error}, where we can observe the errors and the values of the estimator for each method.  It is observed that the errors behave similar to $\eta^2$, i.e., they decay as $\mathcal{O}(h^2)$, so the efficiency and reliability are verified. Moreover, the plot includes a
line with slope $-1.0$, which corresponds to the optimal order of convergence for the proposed schemes. The slopes of the lines obtained by a least squares fitting of the values computed with the adaptive scheme are $-1.04$.
\begin{table}
{\footnotesize
\begin{center}
\caption{Test 4. Computed eigenfunction $\lambda_{1h}$, error and effectivity indexes using the $\mathrm{P}_0^{2}\text{-}\mathbb{NED}^{(1)}_0$ and $\mathrm{P}_0^{2}\text{-}\mathbb{NED}^{(2)}_1$ schemes with adaptively refinements. }
\begin{tabular}{c|c  c |c c c}
	\toprule
	scheme&$\text{dof}$&$\lambda_{1h}$&$\vert\lambda_1-\lambda_{1h}\vert$&$\eta^2$&$\vert\lambda_1-\lambda_{1h}\vert/\eta^2$\\\hline
	\multirow{15}{0.11\linewidth}{$\mathrm{P}_0^{2}\text{-}\mathbb{NED}^{(1)}_0$}
	&1181      &30.19673  &1.93509e+00 &2.50851e+01  &7.71413e-02\\
	&1399      &31.01101  &1.12082e+00 &2.06562e+01  &5.42608e-02\\
	&1975      &31.33653  &7.95304e-01 &1.48616e+01  &5.35140e-02\\
	&2919      &31.60707  &5.24757e-01 &1.01735e+01  &5.15805e-02\\
	&4471      &31.75965  &3.72184e-01 &7.04846e+00  &5.28036e-02\\
	&6561      &31.88079  &2.51042e-01 &4.92335e+00  &5.09900e-02\\
	&10023     &31.98109  &1.50742e-01 &3.31763e+00  &4.54366e-02\\
	&14067     &32.00944  &1.22386e-01 &2.38015e+00  &5.14196e-02\\
	&21599     &32.05984  &7.19859e-02 &1.59321e+00  &4.51830e-02\\
	&31619     &32.08211  &4.97245e-02 &1.09533e+00  &4.53967e-02\\
	&45401     &32.09712  &3.47087e-02 &7.61922e-01  &4.55541e-02\\
	&66797     &32.10917  &2.26548e-02 &5.21796e-01  &4.34171e-02\\
	&97183     &32.11631  &1.55191e-02 &3.59593e-01  &4.31573e-02\\
	&143721    &32.12159  &1.02378e-02 &2.44116e-01  &4.19381e-02\\
	&204461    &32.12516  &6.66500e-03 &1.70746e-01  &3.90346e-02\\
	\midrule
	&Order   &$\mathcal{O}(\text{dof}^{-1.04})$ &  &  &\\
	&$\lambda_1$   &32.13183   &  &&\\
	\midrule  
	\multirow{15}{0.11\linewidth}{$\mathrm{P}_0^{2}\text{-}\mathbb{NED}^{(2)}_1$}
	&1907      &33.05942  &9.27595e-01 &1.44740e+01  &6.40871e-02\\
	&2103      &33.23013  &1.09830e+00 &8.41858e+00  &1.30462e-01\\
	&2347      &33.29788  &1.16605e+00 &5.44523e+00  &2.14142e-01\\
	&2575      &33.34447  &1.21264e+00 &4.17494e+00  &2.90457e-01\\
	&2893      &33.31662  &1.18480e+00 &3.29619e+00  &3.59443e-01\\
	&3609      &32.96793  &8.36099e-01 &2.16033e+00  &3.87023e-01\\
	&4677      &32.66458  &5.32752e-01 &1.36450e+00  &3.90438e-01\\
	&5559      &32.58139  &4.49556e-01 &1.06559e+00  &4.21885e-01\\
	&8325      &32.41385  &2.82018e-01 &6.42697e-01  &4.38804e-01\\
	&11651     &32.34513  &2.13299e-01 &4.04288e-01  &5.27590e-01\\
	&16179     &32.27964  &1.47810e-01 &2.59997e-01  &5.68505e-01\\
	&21025     &32.24409  &1.12260e-01 &1.86314e-01  &6.02532e-01\\
	&31091     &32.20603  &7.42023e-02 &1.14587e-01  &6.47564e-01\\
	&40873     &32.18761  &5.57772e-02 &8.11620e-02  &6.87233e-01\\
	&57171     &32.17095  &3.91252e-02 &5.45027e-02  &7.17859e-01\\
	\midrule
	&Order   &$\mathcal{O}(\text{dof}^{-1.04})$ &  &  &\\
	&$\lambda_1$   &32.13183   &  &&\\  
	\bottomrule         
\end{tabular}\label{tabla-l-shape-ned1}
\end{center}}

\end{table}
\begin{figure}
	\centering
	\includegraphics[scale=0.062,trim= 16cm 0 16cm 0, clip]{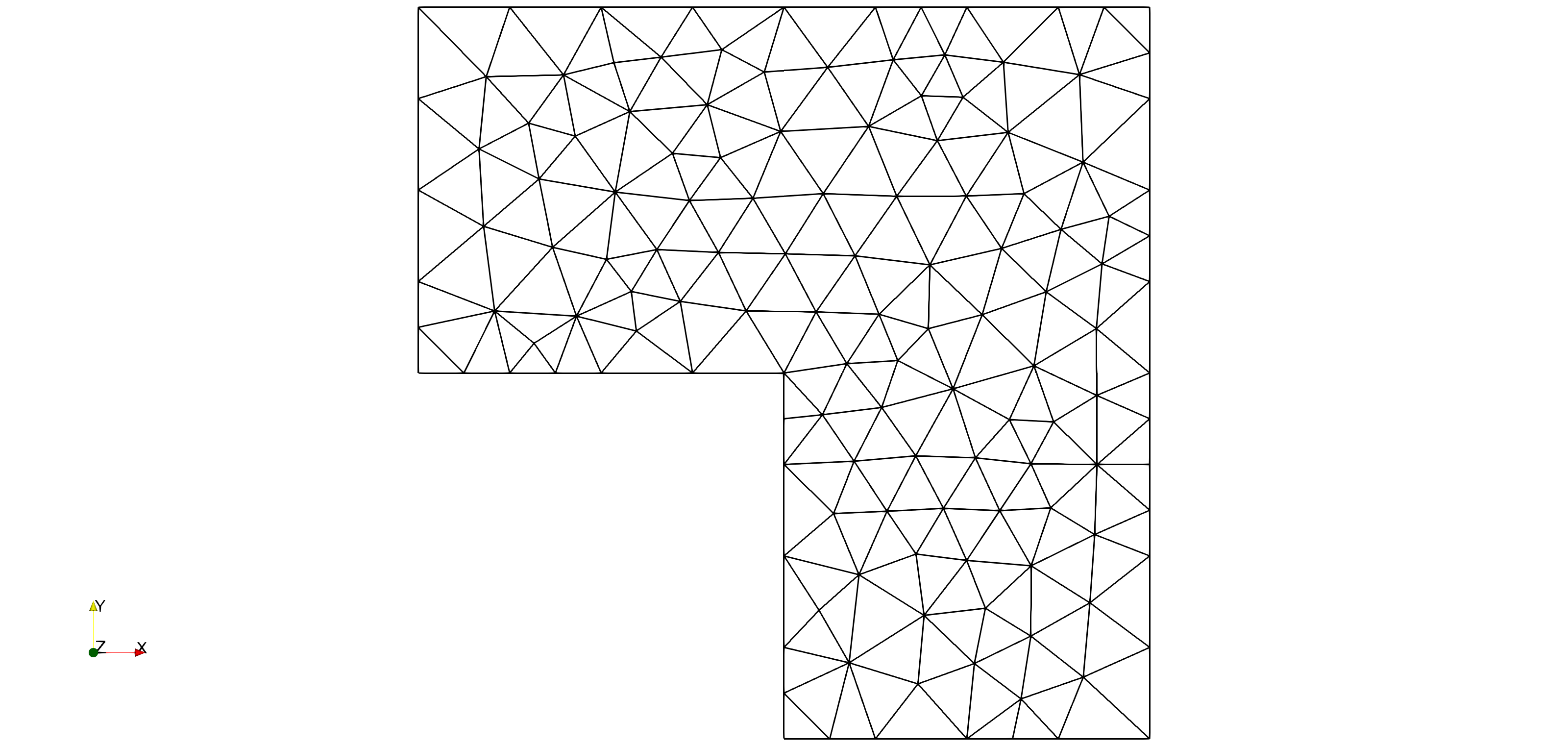}
	\caption{Test 4. Initial mesh on the L-shaped domain.}
	\label{fig:lshapeinitialmesh}
\end{figure}
\begin{figure}
	\centering
	\begin{minipage}{\linewidth}\centering
		\includegraphics[scale=0.39,trim=0cm 0.2cm 0cm 0.3cm,clip]{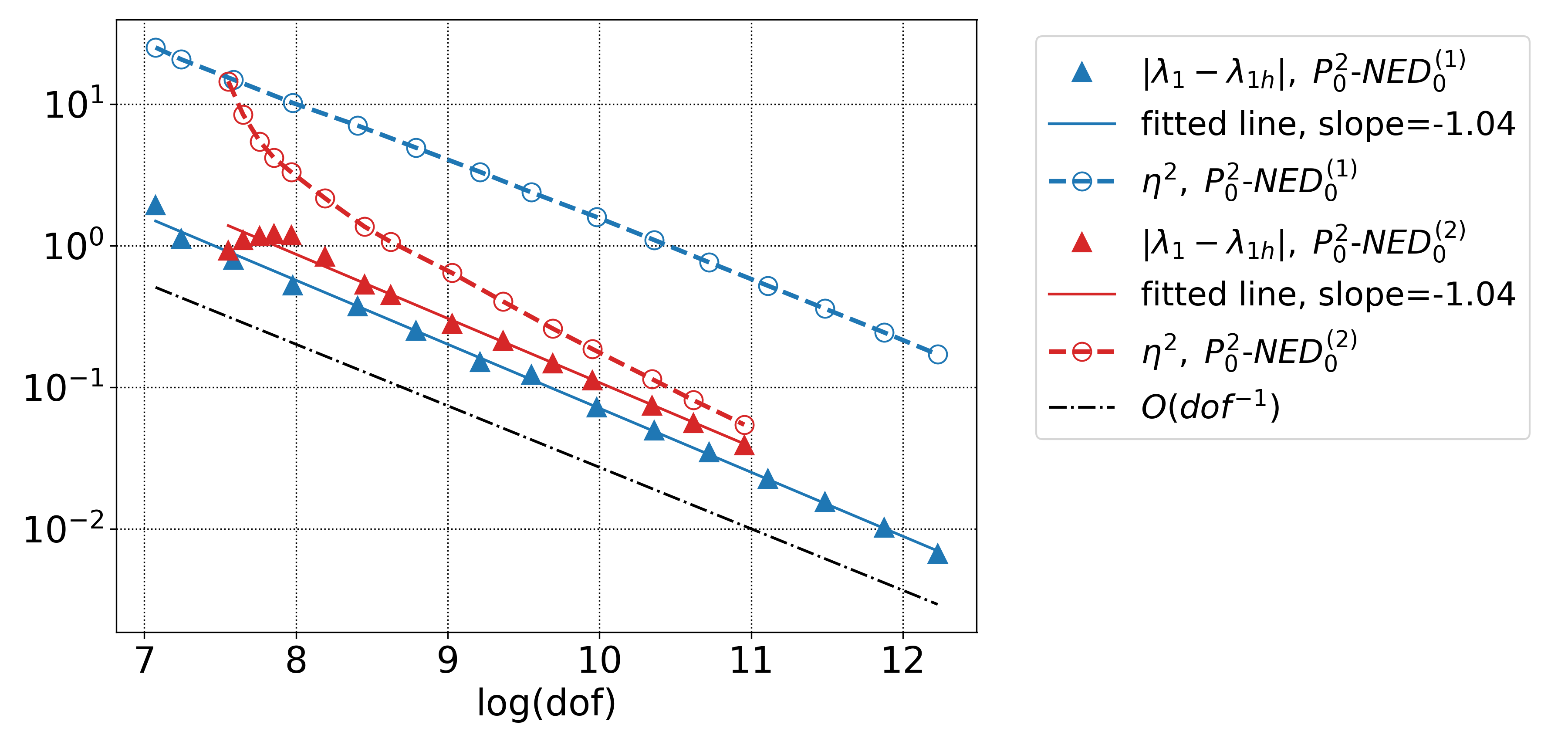}
	\end{minipage}
	\caption{Test 4. Comparison between error, estimators and fit lines in the adaptive refinenment using the lowest order $\mathbb{NED}_0^{(1)}$ and $\mathbb{NED}_1^{(2)}$ families. }
	\label{fig:lshape-error}
\end{figure}
\begin{figure}
	\centering
	\begin{minipage}{0.32\linewidth}
		\includegraphics[scale=0.0608,trim= 28cm 0cm 27cm 0cm, clip]{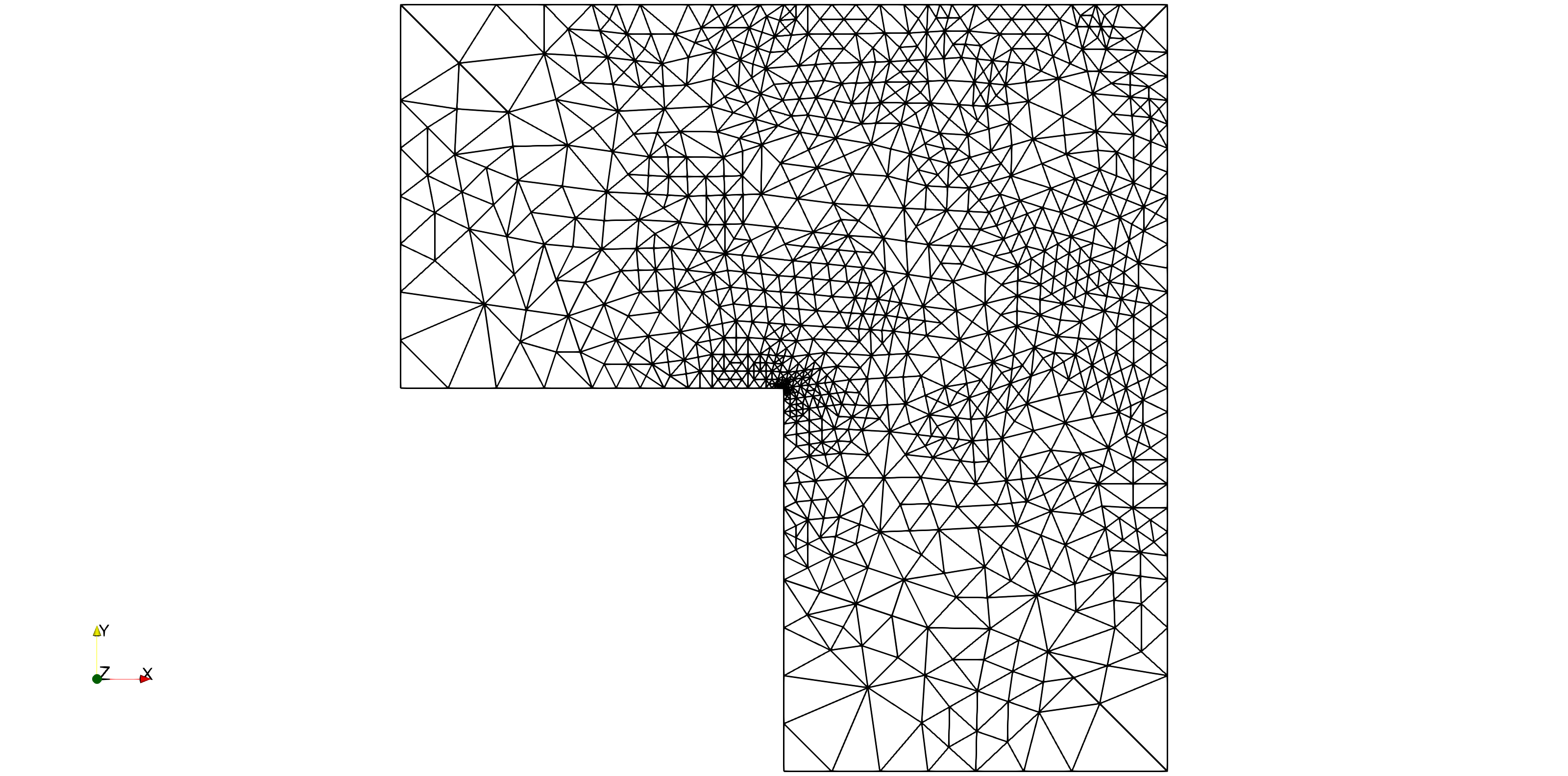}
	\end{minipage}
	\begin{minipage}{0.32\linewidth}
		\includegraphics[scale=0.0608,trim= 28cm 0cm 27cm 0cm,, clip]{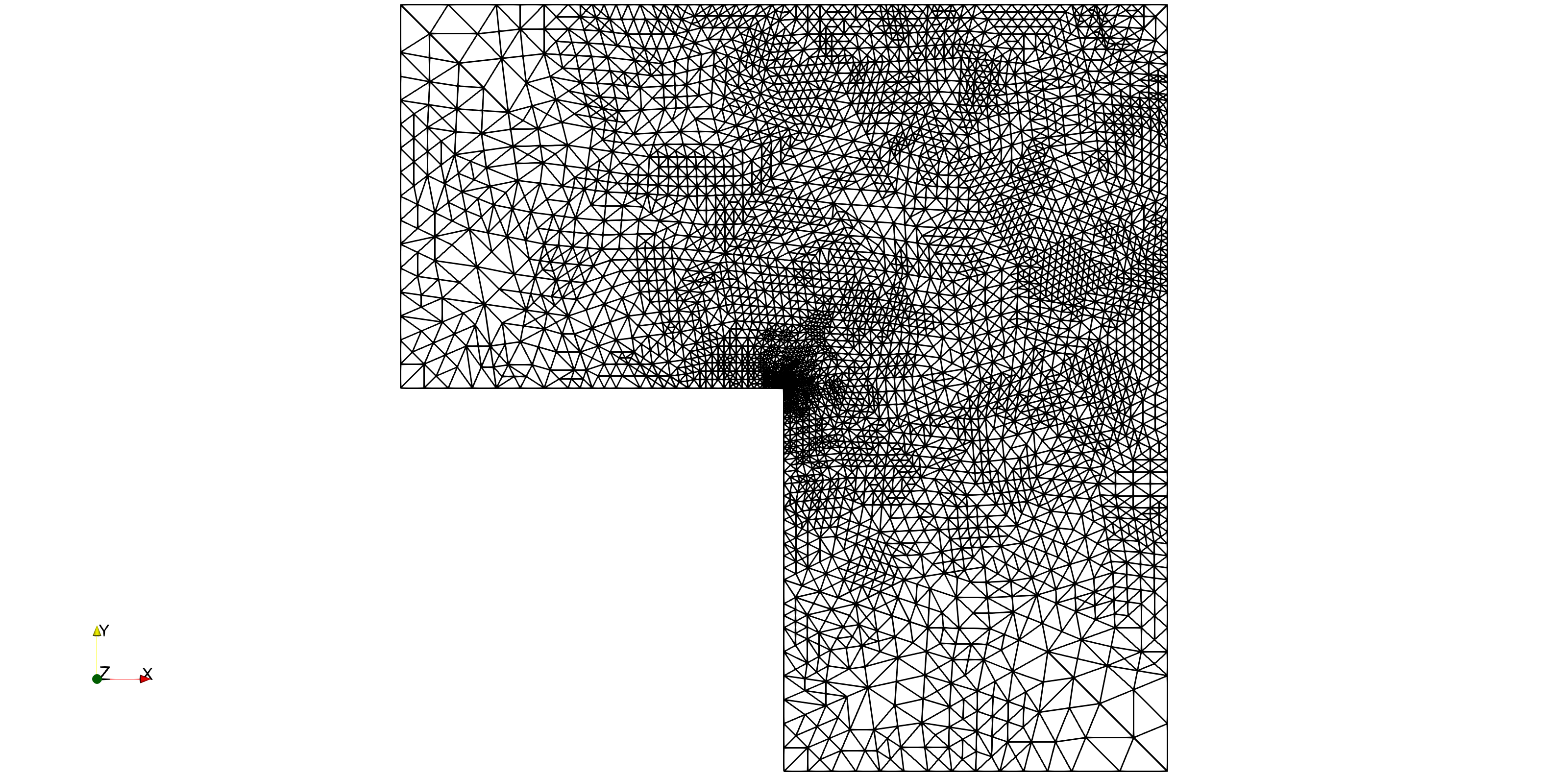}
	\end{minipage}
	\begin{minipage}{0.32\linewidth}
		\includegraphics[scale=0.0608,trim= 28cm 0cm 27cm 0cm,, clip]{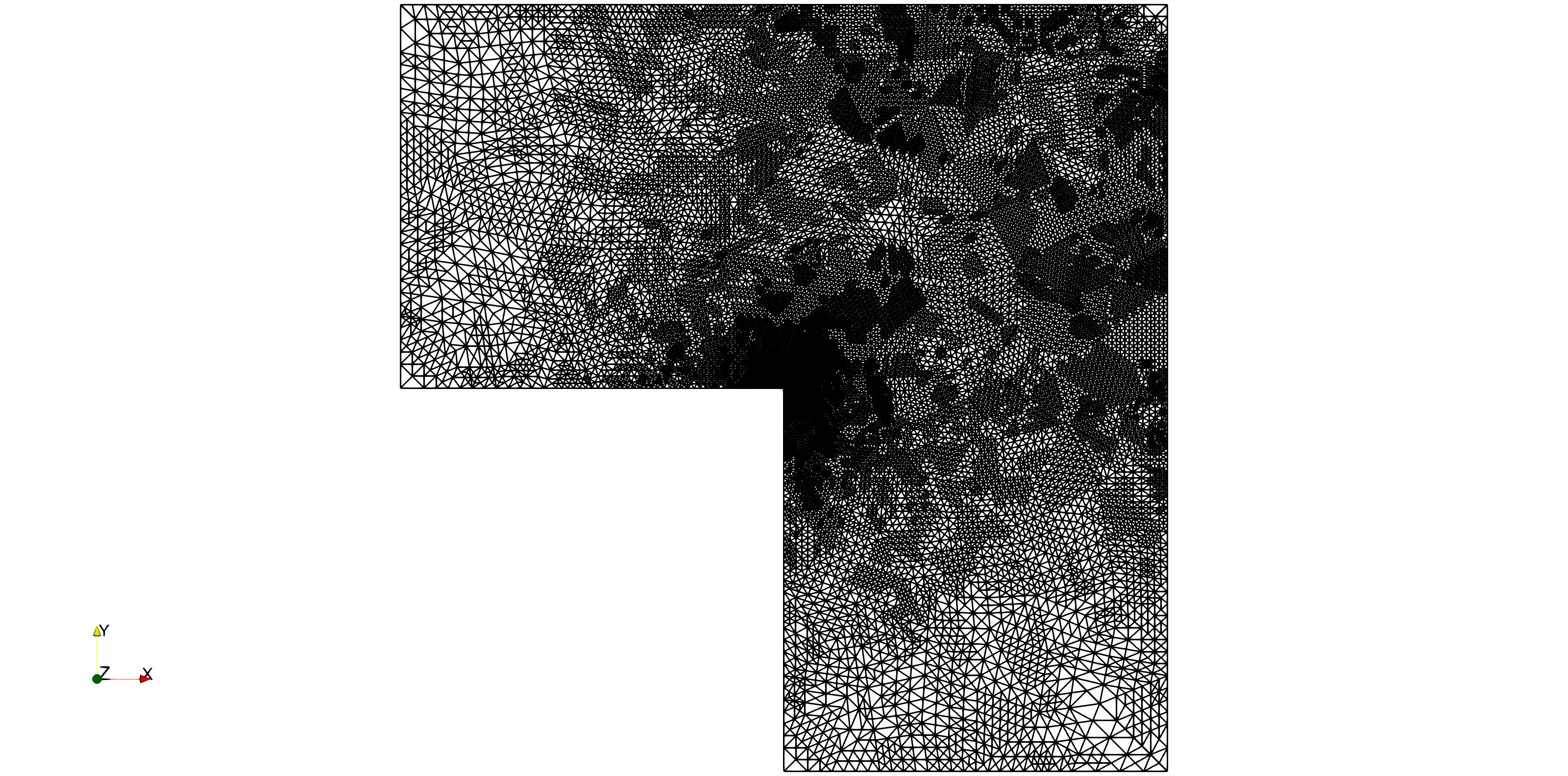}
	\end{minipage}\\\hspace*{0.01cm}
	\begin{minipage}{0.32\linewidth}
		\includegraphics[scale=0.06,trim= 28cm 0cm 27cm 0cm,, clip]{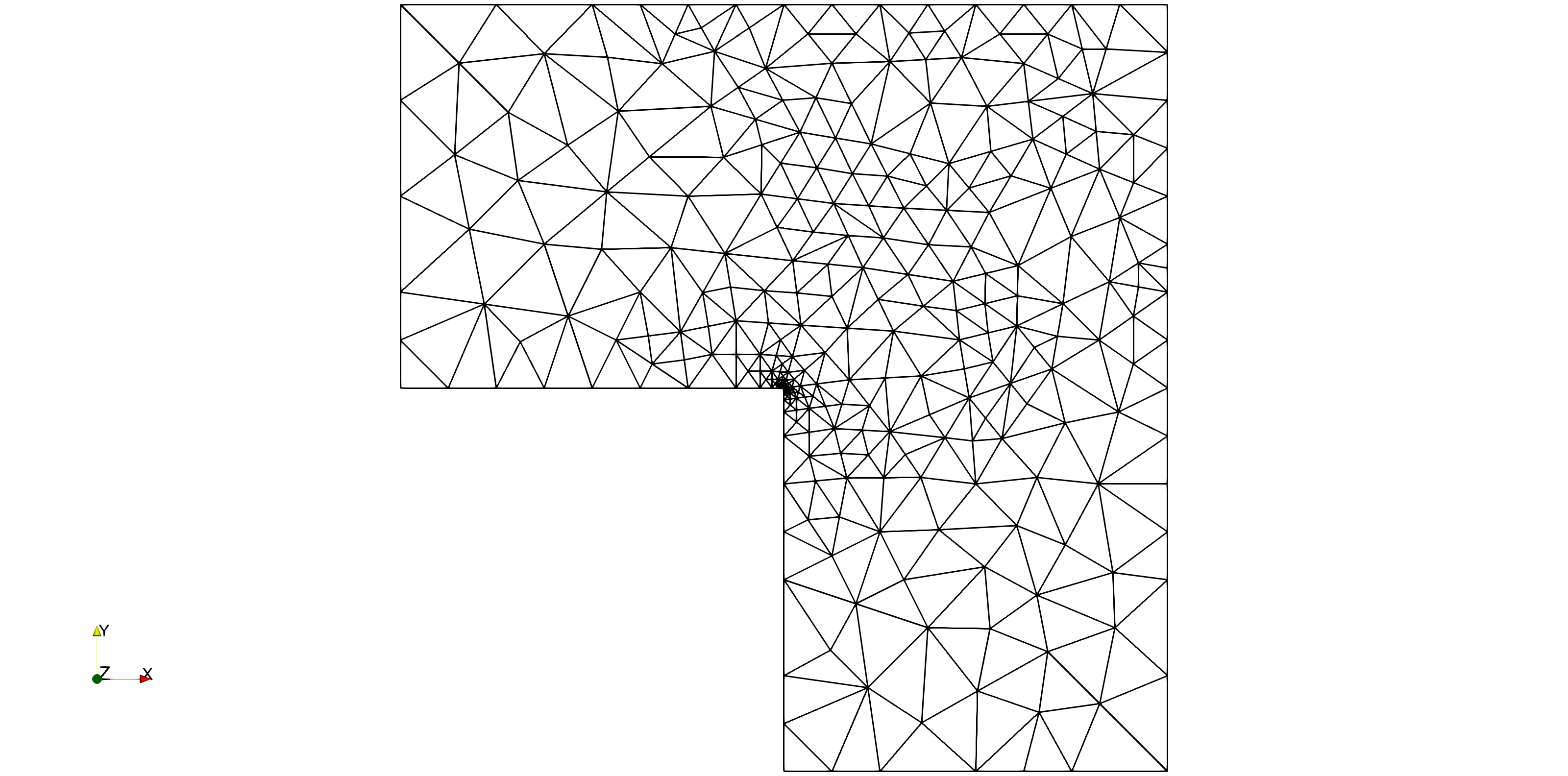}
	\end{minipage}
	\begin{minipage}{0.32\linewidth}
		\includegraphics[scale=0.06,trim= 28cm 0cm 27cm 0cm,, clip]{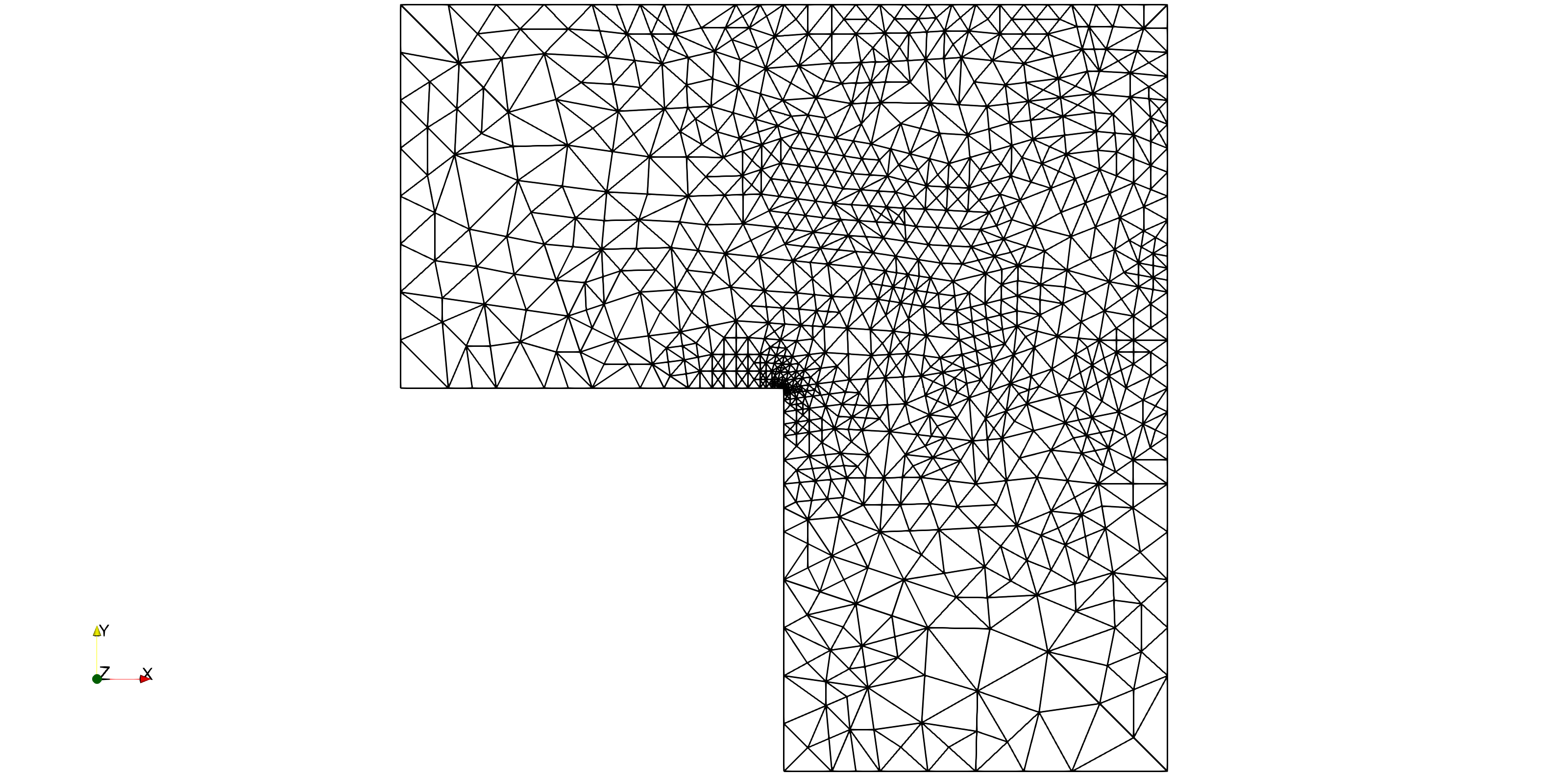}
	\end{minipage}
	\begin{minipage}{0.32\linewidth}
		\includegraphics[scale=0.06,trim= 28cm 0cm 27cm 0cm,, clip]{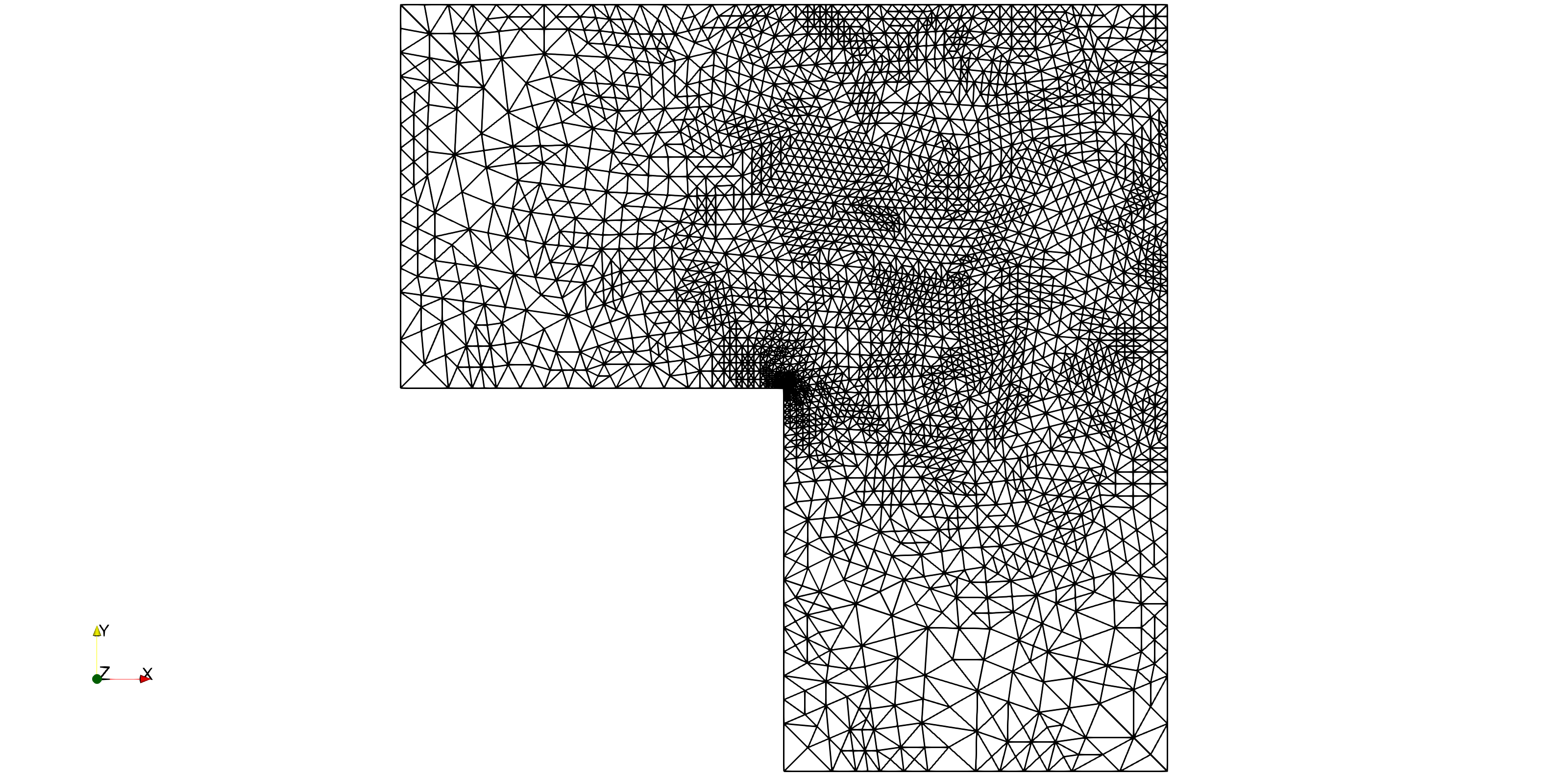}
	\end{minipage}
	\caption{Test 4. Adapted meshes associated to estimator $\eta$ in the seventh, eleventh and last iteration. Top row: $\mathrm{P}_0^{2}\text{-}\mathbb{NED}^{(1)}_0$ scheme  with $10023$, $45401$ and $204461$ degrees of freedom. Bottom row: $\mathrm{P}_0^{2}\text{-}\mathbb{NED}^{(2)}_0$ with $4677, 16179$ and $57171$ degrees of freedom. }
	\label{fig:lshape-mesh-refinamientos}
\end{figure}
\bibliographystyle{siamplain}
\bibliography{LRV_vorticity}
\end{document}